\documentclass[11pt]{article}
\title{Genus, Fiberedness, $\tau$ and $\epsilon$ of Satellite Knots with $n$-Twisted Generalized Mazur patterns}
\author{Holt Bodish}
\date{}
\usepackage{amsmath,amsthm, amssymb,amsfonts}
\usepackage[alphabetic]{amsrefs}
\usepackage{graphicx}
\usepackage{mathtools}
\usepackage{mathrsfs}
\usepackage{tikz-cd} 
\usetikzlibrary{decorations.pathreplacing}
\usepackage[margin=1.5in]{geometry}

\theoremstyle{plain}
\newtheorem{theorem}{Theorem}[section]
\newtheorem{lemma}[theorem]{Lemma}

\newtheorem{definition}[theorem]{Definition}
\newtheorem{corollary}[theorem]{Corollary}

\newcommand{\HFhat}{\widehat{\mathrm{HF}}}
\newcommand{\CFAhat}{\widehat{\mathrm{CFA}}}
\newcommand{\txi}{\tilde{\xi}}
\newcommand{\R}{\mathbb{R}}

\newcommand{\Z}{\mathbb{Z}}

\newcommand{\Q}{\mathbb{Q}}
\newcommand{\F}{\mathbb{F}}
\newcommand{\Ord}{\mathrm{Ord}}
\newcommand{\HFK}{\mathrm{HFK}}

\newcommand{\rank}{\text{rank}}

\newcommand{\CFDhat}{\widehat{\mathrm{CFD}}}
\newcommand{\HFKhat}{\widehat{\mathrm{HFK}}}

\newcommand{\teta}{\tilde{\eta}}

\newcommand{\CFKhat}{\widehat{\mathrm{CFK}}}
\newcommand{\rk}{\text{rk}}

\newcommand{\CFK}{\mathrm{CFK}}
\newcommand{\tikzcirc}[2][black,fill=white]{\tikz[baseline=-0.5ex]\draw[#1,radius=#2] (0,0) circle ;}
\newcommand{\tikzcircle}[2][black,fill=black]{\tikz[baseline=-0.5ex]\draw[#1,radius=#2] (0,0) circle ;}
\newcommand{\tikzcirclee}[2][black,fill=green]{\tikz[baseline=-0.5ex]\draw[#1,radius=#2] (0,0) circle ;}
\newcommand{\talpha}{\tilde{\alpha}}
\newcommand{\tbeta}{\tilde{\beta}}
\begin{document}
\maketitle
\begin{abstract}
    We study a family of $(1,1)$-pattern knots that generalize the Mazur pattern, and compute the concordance invariants $\tau$ and $\epsilon$ of $n$-twisted satellites formed from these patterns. We show that none of the $n$-twisted patterns from this family act surjectively on the smooth or rational concordance group. We also determine when the $n$-twisted generalized Mazur patterns are fibered in the solid torus, compute their genus in $S^1 \times D^2$, and show that $n$-twisted satellites with generalized Mazur patterns and non-trivial companions are not Floer thin. 
\end{abstract}

\section{Introduction}\label{Mazursection}

In this paper, we compute the genus and determine the fiberedness and the Heegaard Floer concordance invariants $\tau$ and $\epsilon$ of satellite knots with arbitrary companions $K$ and patterns from a family of knots in the solid torus, which we denote $Q_n^{i,j}$, shown in Figure \ref{Qij}. Here $j \in \Z_{>0}$ is the winding number of the pattern, $n \in \Z$ is the number of full twists around the meridian, and $i \in \Z_{\geq 0}$ is the number of full twists added to the clasp region in the box labelled $i$ in Figure \ref{Qij}. We refer to the patterns $Q^{i,j}_n$ as \emph{$n$-twisted generalized Mazur patterns}, since $Q_0^{0,1}$ is the Mazur pattern and $Q^{i,1}_0$ is a generalized Mazur pattern in analogy with the generalized Whitehead doubles of \cite{Truong}. (See recent work of \cite{2024generalized} for a different family of patterns also called generalized Mazur patterns.) Given a knot $K$, the satellite knot with $n$-twisted generalized Mazur pattern $Q^{i,j}_n(K)$ can either be viewed as a $0$-twisted satellite with pattern $Q^{i,j}_n$ or as an $n$-twisted satellite with pattern $Q^{i,j}_0$. In this paper, we mostly adopt the latter perspective.


In \cite{Levinemazur}, Levine computed $\tau$ and $\epsilon$ of $0$-twisted satellites with Mazur pattern and arbitrary companions by explicitly determining the bordered bimodule $\widehat{\mathrm{CFDA}}(X_Q)$ associated to the complement of the Mazur pattern in the solid torus and using the bordered pairing theorem of \cite{LOT}. Levine used this to compute $\tau$ and $\epsilon$ of $0$-twisted satellites with Mazur pattern. 
More recently, in \cite{chenhanselman}, Chen and Hanselman showed that the $UV=0$ quotient of the full knot Floer complex of satellite knots with $(1,1)$-patterns can be computed using the immersed curve pairing theorem. They then recovered, in a more direct way, Levine's computation of $\tau$ and $\epsilon$ of $0$-twisted satellites with Mazur pattern \cite[Theorem 6.9]{chenhanselman}. 

One consequence of Levine's computation of $\epsilon$ of satellites with Mazur pattern is that the Mazur pattern does not act surjectively on the smooth or $\Q$-homology concordance group.  Levine then used this to construct a knot in the boundary of a contractible $4$-manifold that does not bound a PL disk there or in any other contractible $4$-manifold with the same boundary, answering a question of Kirby and Akbulut \cite[Theorem 1.2]{Levinemazur}. 

In this work, we extend these computations to determine $\tau$ and $\epsilon$ of $n$-twisted satellites with patterns $Q^{i,j}_0$ and non-trivial companions $K$: $Q^{i,j}_n(K)$. As a special case of our work, we show that $\tau$ of an $n$-twisted satellite knot with Mazur pattern and companion $K$ depends only on the value of $n$ relative to $2\tau(K)$, which echos the computations of $\tau$ of $n$-twisted Whitehead doubles \cite{Heddenwhitehead}. Interestingly this is not the case for $\tau$ of satellites with patterns $Q^{i,j}_n$ with winding number $j>1$, where we show that the value of $\tau$ depends linearly on $n$ and quadratically on $j$. Further, we show that for any companion knot $K$, $\epsilon(Q^{i,j}_n(K)) \neq -1$. This shows that for all $i \geq 0$, $j > 0$ and $n \in \Z$ the patterns $Q^{i,j}_n$ do not act surjectively on the smooth or $\Q$-homology concordance group. Using this, we construct a bi-infinite family of patterns $P^i_n$ so that $P^i_n(K)$ is not slice in any $\Q$-homology $4$-ball for any knot $K \subset S^3$. See \cite{2024generalized} for another infinite family of patterns with this property.

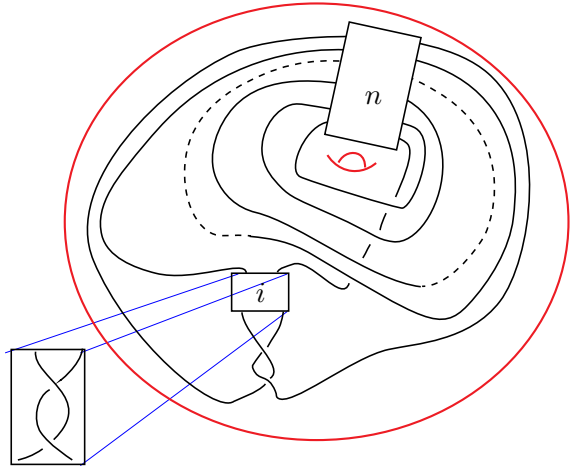
\begin{figure}[!tbp]
\begin{center}
  \begin{tikzpicture}
  \draw[very thin, blue] (1.1,.1) to (3.85,2.15);
  \draw[very thin, blue] (1.1,1.6) to (3.85,2.65);
  \draw[very thin, blue] (.1,1.6) to (3.2,2.65);
  \node at (3.5,2.4) {$i$};
\node[anchor=south west,inner sep=0] at (0,0)   {\includegraphics[scale=.2]{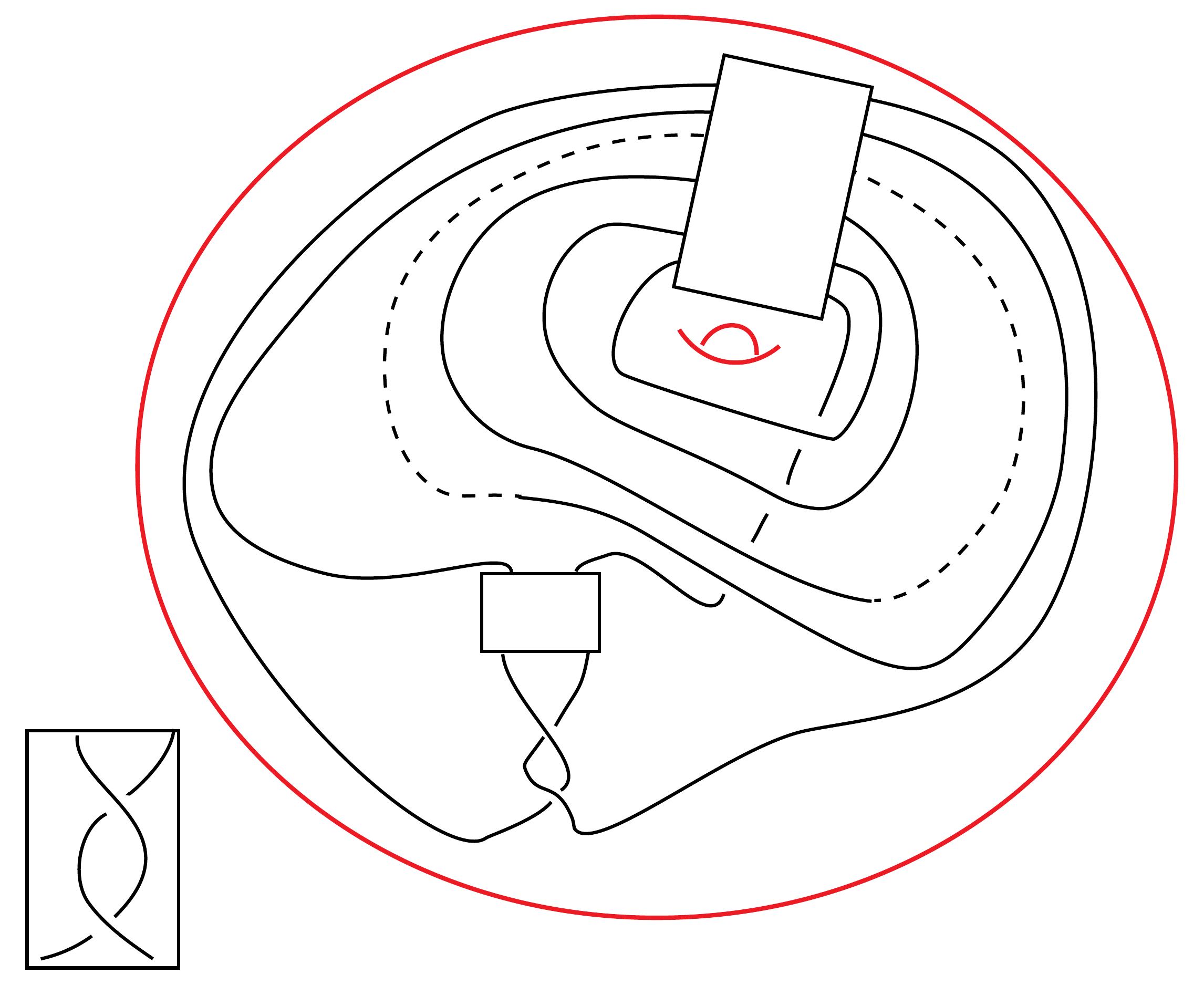}};
\node at (5,5) {$n$};
\end{tikzpicture}
    \caption{The pattern $Q^{i,j}_n$. In the box labelled $i$, there are $i$ full twists on two strands as shown in the box on the bottom left. In the box labelled $n$ insert $n$ full twists on $j+2$ strands}\label{Qij}
    \end{center}
\end{figure}

In another direction, we extend recent computations of Petkova and Wong in \cite{petkovamazur}, where they showed that the genus and fiberedness of the $n$-twisted Mazur pattern in the solid torus can be determined from the bordered type A structure $\CFAhat(S^1\times D^2,Q^{0,1})$, using the bordered pairing theorem and classical results about the genus and fiberedness of satellites knots \cites{schubert, HMS}. We expand on these computations and give closed formulas for the genus of $n$-twisted satellite knots with patterns $Q^{i,j}$ and arbitrary companions, and we determine for which $i,j$ and $n$ the pattern knots $Q^{i,j}_n$ are fibered in the solid torus. We also show that for any non-trivial companion $K$ the satellite knot $Q^{i,j}_n(K)$ is not Floer thin. 


\subsection{Statement of Results}

Recall that for the $n$-twisted satellite knot $P_n(K)$ with non-trivial companion knot $K$, we have \cite{schubert} \begin{equation}\label{schuberteq}g(P_n(K))=|w(P)|g(K)+g(P_n),\end{equation} where $w(P_n)=(P_n\cap (\{pt\} \times D^2))$ is the winding number of the pattern and $g(P_n)$ is the genus of a relative Seifert surface for $P_n$. A consequence of this formula is that to determine $g(P_n(K))$ for an arbitrary non-trivial companion knot $K$, it is enough to determine $g(P_n(T_{2,3}))$. We use this observation together with the fact that knot Floer homology detects the genus of knots in $S^3$ to prove the following:

\begin{theorem}\label{genusnontrivial} For $K$ be a non-trivial knot in $S^3$, $j \in \Z_{>0},$ $i \in \Z_{\geq 0}$ and $n \in \Z,$

$$g(Q^{i,j}_n(K))=\begin{cases} jg(K)+\dfrac{j(j+1)}{2}n+1 & n \geq 0\\
jg(K)+\dfrac{j(j+1)}{2}|n|+(1-j) & n<0\,. \end{cases}$$
\end{theorem}

Equation \ref{schuberteq}, and so the proof of Theorem \ref{genusnontrivial}, requires the companion knot to be non-trivial. However, a similar computation gives $g(Q^{i,j}_n(U))$:

\begin{theorem}\label{genusunknotsatellite}
For $j \in \Z_{> 0}$, $i \in \Z_{\geq 0}$ and $n \in \Z,$ $$g(Q^{i,j}_n(U))=\begin{cases} \dfrac{j(j+1)}{2}n+1 & n >0 \\ 0 & n=0 \\ \dfrac{j(j+1)}{2}|n|+1-j & n<0\,. \end{cases}$$
\end{theorem}
Note that, when $j=1$ and $i=0$, Theorems \ref{genusnontrivial} and \ref{genusunknotsatellite} recover \cite[Theorem 1.0.5]{petkovamazur}

Recall from \cite{HMS} that a satellite knot $P_n(K)$ is fibered if and only if the companion knot $K$ is fibered in $S^3$ and the pattern knot $P_n$ is fibered in $S^1 \times D^2$. This implies that to show that a satellite knot $P_n(K)$ is fibered, it is enough to show that the satellite knot $P_n(T_{2,3})$ is fibered. Since a knot $K \subset S^3$ with $g(K)=g$ is fibered in $S^3$ if and only if $\rank(\HFKhat(S^3,K,g))=1$ \cites{Nifibered, Juhaszfibered}, we see that to determine if a pattern $P_n$ is fibered it is enough determine the rank of the top Alexander graded piece of the knot Floer homology of $P_n(T_{2,3})$. For $P=Q^{i,j}$, in Lemma \ref{rankintop} we compute the rank of the top Alexander graded piece of the knot Floer homology of $Q^{i,j}_n(T_{2,3})$ and show
\begin{theorem}\label{fiberedness}
    Let $K$ be a non-trivial fibered knot in $S^3$. Then the satellite knot $Q^{i,j}_n(K)$ is fibered if and only if either $j \geq 2,$ $i=0$, and $n \neq 0$, or $j=1$, $i=0$ and $n \neq -1,0$. 
\end{theorem}
Note that the case $j=1$ and $i=0$ of Theorem \ref{fiberedness} recovers \cite[Theorem 1.0.6]{petkovamazur}. 

Recall that a knot is called Floer thin if all the generators of the knot Floer homology are supported in the same $\delta$ grading, where $\delta(x)=M(x)-A(x)$. We show
\begin{theorem}\label{theoremthinnontrivial}
For any non-trivial companion knot $K$, the satellite knots $Q^{i,j}_n(K)$ are not Floer thin. 
\end{theorem}
We also consider the case when the companion knot is trivial. 
\begin{theorem}\label{theoremthintrivial}
For $K=U$ the satellite knots $Q^{i,j}_n(U)$ are Floer thin if and only if $j=1$ and $n=-1$. 
\end{theorem}
Note that Theorems \ref{theoremthinnontrivial} and \ref{theoremthintrivial}, in the case $i=0$ and $j=1$, recover \cite[Theorem 1.01]{petkovamazur}.\\

In \cite{OStau} and \cite{Hom} two smooth concordance invariants of knots derived from the $UV=0$ quotient of the full knot Floer complex $\CFK^{\infty}$ are introduced, called $\tau(K)$ and $\epsilon(K)$. These invariants have proved fruitful in the study of the knot concordance group \cites{Heddenwhitehead, Hom, Levinemazur}. We give an explicit computation of $\tau$ and $\epsilon$ of satellite knots with arbitrary companion knots $K$ and patterns $Q^{i,j}_n$. 

\begin{theorem}\label{tau}
If $K$ is a knot in $S^3$ with $\epsilon(K)=-1$, then for all $i\geq 0$, $j \geq 1$ and $n \in \Z$ $$\tau(Q^{i,j}_n(K))=j(\tau(K)+1)+\dfrac{j(j-1)}{2}n.$$\\

If $K$ is a knot in $S^3$ with $\epsilon(K)=1$, then for all $i\geq 0$, $j \geq 1$ and $n \in \Z,$ $$\tau(Q^{i,j}_n(K))=\begin{cases} j\tau(K)+\dfrac{j(j-1)}{2}n+1 & n<2\tau(K) \\ j\tau(K)+\dfrac{j(j-1)}{2}n & n \geq 2\tau(K)\,. \end{cases}$$\\

If $K$ is a knot in $S^3$ with $\epsilon(K)=0$, then for all $i\geq 0$, $j \geq 1$ and $n \in \Z,$ $$\tau(Q^{i,j}_n(K))= \begin{cases}\dfrac{j(j-1)}{2}n & n\geq 0\\ \dfrac{j(j-1)}{2}n+j & n<0\,. \end{cases}$$
\end{theorem}

\begin{theorem}\label{epsilon}
For any knot $K$ and for any $i\geq 0$, $j \geq 1$ and $n \in \Z$, we have $\epsilon(Q^{i,j}_n(K))\in \{0,1\}$. 
\end{theorem}

The invariant $\epsilon$ is a concordance invariant, and takes values in $\{0,1,-1\}$. If we let $\mathcal{C}_{\Q}$ denote the $\Q$-homology knot concordance group then an immediate Corollary of Theorem \ref{epsilon} is
\begin{corollary}\label{nonsurjective}
    F or all $i\geq 0$, $j \geq 1$ and $n \in \Z$, the satellite operators $Q^{i,j}_n: \mathcal{C}_{\Q} \to \mathcal{C}_{\Q}$ are not surjective. 
\end{corollary}
As mentioned above, this shows that when we add full twists to the clasp region of the Mazur pattern (by increasing the parameter $i$) and add add meridional twists to the Mazur pattern, we get a bi-infinite family of winding number $1$ patterns that do not act surjectively on the $\Q$-homology concordance group and thus gives infinitely many examples of pattern knots in homology spheres that do not bound PL disks in any contractible $4$-manifold. See the recent work of \cite{2024generalized} for another infinite family of winding number $1$ patterns with the same property. Our construction also gives many patterns of arbitrarily large winding number and various knot types in $S^3$ that also are not surjective satellite operators, and in particular shows that for these patterns, the concordance invariant $\epsilon$ is not sensitive to adding twists in the clasp region and adding full meridional twists to the pattern knot.

\section*{Acknowledgments}
The author is thankful for the guidance of his advisor Robert Lipshitz is very grateful to Subhankar Dey for many helpful discussions and suggestions. The author also thanks Jonathan Hansleman and Wenzhao Chen for helpful correspondence. 

\section{Background}\label{background}

In this section we review some concepts from the immersed curve reformulation of bordered Floer homology and the bordered pairing theorem for $(1,1)$-patterns. We assume the reader is familiar with the various flavors of knot Floer homology and the work of \cites{LOT}. We quickly review the necessary background to state the immersed curve reformulation of the bordered invariants and bordered pairing theorem from \cites{Chen, chenhanselman, HRW}. In Section \ref{CFD} we introduce some notation and prove a structure theorem for the immersed curve associated to an $n$ framed knot complement. Then in Section \ref{(1,1)} we discuss $(1,1)$-patterns and the work of \cite{Chen} with an eye towards understanding the $UV=0$ quotient of the knot Floer complex from the pairing diagram as in \cite{chenhanselman}, and then in Section \ref{beta} we discuss the specific family of $(1,1)$-patterns that gives rise to the patterns knots $Q^{i,j}$.

\subsection{Immersed Curves for $n$-Framed Knot Complements}\label{CFD}

In this section, we recall the algorithm from \cite[Section 11.5, Theorem 11.26]{LOT} for computing $\CFDhat(S^3- \nu(K),n)$ from $\CFK^{-}(K)$. For the definitions of reduced, filtered basis, we refer the reader to the original source (see also \cite{geography}).
We call a filtered reduced basis over $\F[U]$ vertically simplified if for each basis element $x_i$ exactly one of the following conditions is satisfied

\begin{itemize}

\item There is a unique incoming vertical arrow, and no outgoing vertical arrow, or
\item There is a unique outgoing vertical arrow and no incoming vertical arrow, or
\item There are no vertical arrows. 

\end{itemize}

A horizontally simplified basis is defined similarly, replacing vertical by horizontal in the above. Given a knot $K$ and a framing $n$, exists a pair of bases $\teta=\{\teta_1,\dots,\teta_{2k}\}$ and $\txi=\{\txi_1,\dots,\txi_{2k}\}$ for $\CFK^{-}(K)$ that are horizontally and vertically simplified respectively. They are indexed so that for every pair $\teta_{2i-1}$ and $\teta_{2i}$ there is a horizontal arrow of length $l_i\geq 1$ connecting them and similarly, there is a vertical arrow of length $k_i\geq 1$ from $\txi_{2i-1}$ to $\txi_{2i}$. There are corresponding bases $\xi=\{\xi_0,\dots,\xi_{2k}\}$ and $\eta=\{\eta_0,\dots,\eta_{2k}\}$ for $\iota_0\CFDhat(X_k,n)$ so that if $\txi_j=\sum_{i=0}^{2k} a_{ij}\teta_i$ and $\eta_j=\sum_{i=0}^{2k}b_{ij}\txi_i$, then the corresponding change of bases formulas hold with the coefficients restricted to $U=0$. The summand $\iota_1\CFDhat$ has basis $$\bigcup_{i=1}^k\{\kappa_1^i,\dots,\kappa_{k_i}^i\}\cup\bigcup_{i=1}^k \{\lambda_1^i,\dots,\lambda_{l_i}^i\} \cup\{\mu_1,\dots,\mu_{|2\tau(K)-n|}\}.$$

There are non-zero coefficient maps induced from the horizontal and vertical arrows in the complex for $CFK^{-}$ as follows. A length $k_i$ vertical arrow from $\xi_{2i-1}$ to $\xi_{2i}$ gives the sequence of type D operations, where an arrow from $x$ to $y$ labelled with $\rho_I$ means that $\delta(x)=\rho_I \otimes y+\cdots$

$$\xi_{2i-1} \xrightarrow{\rho_1} \kappa_1^i \xleftarrow{\rho_{23}} \kappa_2^{i} \dots\xleftarrow{\rho_{23}}\kappa^i_{k_i}\xleftarrow{\rho_{123}}\xi_{2i}$$

Similarly, for each length $l_i$ horizontal arrow from $\eta_{2i-1}$ to $\eta_{2i}$, we get the sequence of type D operations

$$\eta_{2i-1} \xrightarrow{\rho_3} \lambda_1^i \xrightarrow{\rho_{23}} \lambda^i_{2}\xrightarrow{\rho_{23}} \dots \xrightarrow{\rho_{23}}\lambda_{l_i}^i \xrightarrow{\rho_2} \eta_{2i}$$

Additionally, there are coefficient maps from $\xi_0$ to $\eta_{0}$ depending on the framing and the value of the invariant $\tau(K)$.

\begin{itemize}

\item $\xi_0 \xrightarrow{\rho_{12}} \eta_0 \quad \text{if} \quad n=2\tau(K)$

\item $\xi_0 \xrightarrow{\rho_1} \mu_1 \xleftarrow{\rho_{23}} \dots \xleftarrow{\rho_{23}} \mu_m \xleftarrow{\rho_3} \eta_0  \quad \text{if} \quad n<2\tau(K) \quad m=2\tau(K)-n$

\item $\xi_0 \xrightarrow{\rho_{123}} \mu_1 \xrightarrow{\rho_{23}} \dots \xrightarrow{\rho_{23}} \mu_m \xrightarrow{\rho_2} \eta_0  \quad \text{if} \quad n>2\tau(K), \quad m=n-2\tau(K)$

\end{itemize}

For example, for the knot $K=T_{2,3}$, the right-handed trefoil, $\CFK^{-}(T_{2,3})$ has a simultaneously vertically and horizontally simplified $\F[U]$ basis $\{\txi_0,\txi_1,\txi_2\}$ with differential given by $\partial(\txi_1)=U\txi_0+\txi_2$. Applying the above algorithm, we get the type $D$ structure shown in Figure \ref{trefoil0framed}.

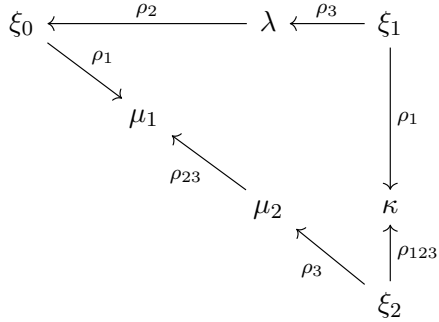
\begin{figure}[!tbp]
\begin{center}
\begin{tikzcd}
\xi_0\arrow{dr}{\rho_1} & & \lambda \arrow{ll}[swap]{\rho_2} & \xi_1 \arrow{l}[swap]{\rho_3}\arrow{dd}{\rho_1}\\
& \mu_1 & &\\
&& \mu_2 \arrow{ul}{\rho_{23}}& \kappa \\
&&& \xi_2 \arrow{u}[swap]{\rho_{123}}\arrow{ul}{\rho_3}
\end{tikzcd}
\caption{Type D structure for $0$-framed right handed trefoil complement}\label{trefoil0framed}
\end{center}
\end{figure}

For any knot $K$ in $S^3$, there is always a vertically distinguished element of a horizontally simplified basis, which is an element in a horizontally simplified basis with no incoming or outgoing vertical arrows. Similarly, there is a horizontally distinguished element of a vertically simplified basis. In \cite[Lemma 3.2]{Hom}, Hom shows that it is always possible to find a horizontally simplified basis for $\CFK^{\infty}(K)$ so that one of the horizontal basis elements $\xi_0$ is the vertically distinguished generator of some vertically simplified basis. Note that the concordance invariant $\epsilon(K)$ can be defined in terms of the generator $\xi_0$: If $\xi_0$ occurs at the end of a horizontal arrow, then $\epsilon(K)=1$, if $\xi_0$ occurs at the beginning of a horizontal arrow, then $\epsilon(K)=-1$. If there is no horizontal arrow to or from $\xi_0$, then $\epsilon(K)=0$.

Given a type $D$ structure over the torus algebra, like $\CFDhat(S^3\setminus \nu(K),n)$, the work in \cite{HRW} shows how we can represent it as an immersed multicurve with local systems in the torus, which we now describe. The first step is to construct a decorated graph from the type D structure. Let $N$ be a type $D$ structure over the torus algebra, and let $N_i=\iota_iN$. This gives a decomposition $N=N_0\oplus N_1$. Given bases $B_i$ of $N_i$, for $i=0,1$, we construct a decorated graph $\Gamma$ as follows. The vertices of $\Gamma$ are in correspondence with the basis elements and are labelled $\tikzcircle{2pt}$ or $\tikzcirc{2pt}$ depending on if the vertex corresponds to a basis element in $B_0$ or $B_1$ respectively. Suppose now that we have two vertices corresponding to basis elements $x$ and $y$ such that $\delta(x)=\rho_I\otimes y +\cdots$, for $I \in \{\emptyset,1,2,3,12,23,123\}$. In this case we put an edge labelled $\rho_I$ from $x$ to $y$. A decorated graph is called reduced if no edge labelled by $\rho_{\emptyset}$ appears. Since our bases for $\CFK^{-}(S^3,K)$ are reduced, this is automatically satisfied in the type D structure as well. The next step is to take a decorated graph and turn it into an immersed train track in the punctured torus. Let $T^2=\R^2/\Z^2$ and let $w=(1-\epsilon,1-\epsilon)$ be a basepoint. Let $\mu$ and $\lambda$ be the images of the $x$ and $y$ axes respectively and embed the vertices of $\Gamma$ into $T^2$ so that the $\tikzcircle{2pt}$ vertices lie on $\lambda$ in the interval $\{0\} \times [\frac{1}{4},\frac{3}{4}]$ and the $\tikzcirc{2pt}$ vertices lie on $\lambda$ in the interval $[\frac{1}{4},\frac{3}{4}]\times \{0\}$. Then we embed the edges into the torus according to the rules shown in \cite[Figure 19]{HRW}. In general this train track is not necessarily an immersed curve, but work in \cite{HRW} shows that for type $D$ structures that arise from 3-manifolds with torus boundary one can always choose a nice basis so that the train track is an immersed curve (possibly with local systems).

Moving back to understanding the ingredients for computing the knot Floer homology of $n$ twisted satellites, note that the pair $(S^3,P_n(K))$ can be obtained by gluing $S^3 -\nu(K)$ with framing $n$ to $(S^1 \times D^2,P)$ or by gluing $S^3-\nu(K)$ with framing $0$ to the pair $(S^1 \times D^2, P_n)$. We want to study the pairing $\CFDhat(S^3-\nu(K),n) \boxtimes \CFAhat(S^1\times D^2,P)$ which computes $\CFKhat(S^3,P_n(K))$ from the perspective of immersed curves. With this goal in mind, we move to understand the immersed curve associated to an $n$-framed knot complement.

\begin{definition}[\cites{HRW, HRW1, cabling}]
    Given a knot $K \subset S^3$, let $\alpha(K,n)$ denote the immersed multi-curve representing the type D structure $\CFDhat(S^3- \nu(K),n)$.
\end{definition}
\begin{figure}[!tbp]
  \centering
  \begin{minipage}[b]{0.3\textwidth}
\begin{center}
\begin{tikzcd}
\xi_0\arrow[dotted]{dddrrr}{}  & \lambda \arrow{l}[swap]{\rho_2} &\\
& & & &\\
&& & \kappa \\
&&& \eta_0 \arrow{u}[swap]{\rho_{123}}
\end{tikzcd}
\caption{Type D structure for complement of knot $K$ with $\tau(K)>0$ and $\epsilon(K)=1$, where we replace the dotted arrow from $\xi_0$ to $\eta_0$ by the appropriate unstable chain.}\label{typeDtauposepsilon1}
\end{center}
  \end{minipage}
  \hspace{2in}
  \begin{minipage}[b]{0.3\textwidth}
\begin{center}
\begin{tikzcd}
\lambda & \xi_0\arrow[dotted]{dddrr}{}\arrow{l}{\rho_3}& &\\
& & & &\\
&& && \\
&&& \eta_0 \arrow{d}[swap]{\rho_{1}}\\
&&&\kappa&
\end{tikzcd}
\caption{Type D structure for complement of knot $K$ with $\tau(K)>0$ and $\epsilon(K)=-1$, where we replace the dotted arrow from $\xi_0$ to $\eta_0$ by the appropriate unstable chain.}\label{typeDtauposepsilonminus1}
\end{center}
  \end{minipage}
\end{figure}
As in \cite[Proposition 2]{cabling} we single out a special component of the immersed multi-curve $\alpha(K,n)$, denoted $\gamma_0$ and called the \emph{essential component} of the immersed curve (See also \cite[pp 43-44]{HRW}). As mentioned there, if we lift the curve to $\R^2$ and consider the vertical axes $\Z\times \R$, then the essential component is the only component of the immersed curve that crosses from $\{i\} \times \R$ to $\{i+1\} \times \R$ and in the case when the framing $n=0$, this potion of the essential component of the immersed curve component has slope $2\tau(K)$ (That is is spans $2\tau(K)$ rows) and it either turns up, down or continues straight after passing through $\{i+1\} \times \R$ if $\epsilon(K)=-1, 1$ or $0$ respectively. We extend these observations to a structure theorem for a portion of $\gamma_0$ of $n$ framed knot complements.

\begin{figure}[!tbp]
  \centering
  \begin{minipage}[b]{0.3\textwidth}
  \begin{tikzpicture}
\node[anchor=south west,inner sep=0] at (0,0)    {\includegraphics[width=1.2\textwidth]{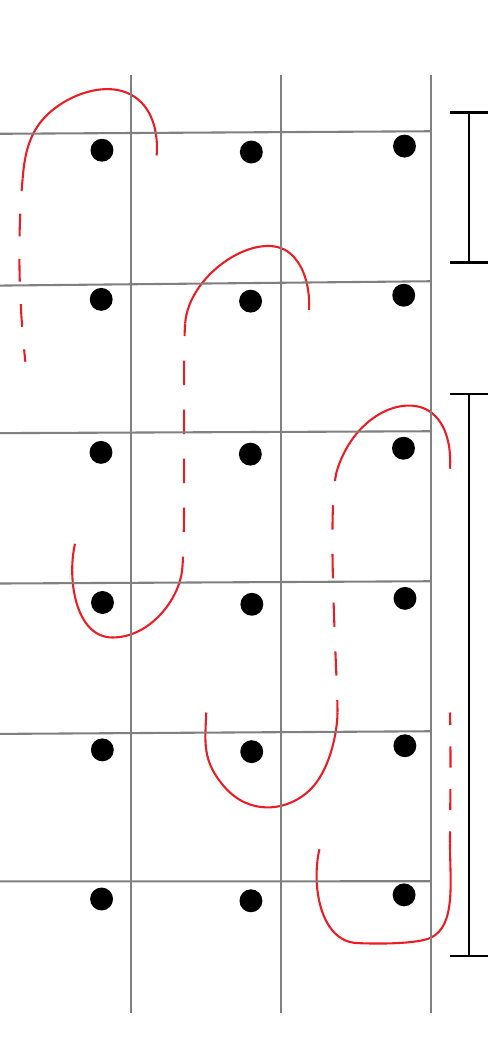}};
\node[font=\tiny] at (1.3,4.1) {$\tikzcirc{2pt}$};
\node[font=\tiny] at (1.3,9.5) {$\tikzcirc{2pt}$};
\node[font=\tiny] at (1.8,4.6) {$\tikzcirc{2pt}$};
\node[font=\tiny] at (1.85,6.15) {$\tikzcirc{2pt}$};
\node[font=\tiny] at (2,7.6) {$\tikzcirc{2pt}$};
\node[font=\tiny] at (2.8,8) {$\tikzcirc{2pt}$};
\node[font=\tiny] at (3.1,7.6) {$\tikzcirc{2pt}$};
\node[font=\tiny] at (.75,4.6) {$\tikzcirc{2pt}$};

\node[font=\tiny] at (1.5,4) {$\eta_0$};
\node[font=\tiny] at (1.5,9.6) {$\xi_0$};
\node[font=\tiny] at (2,4.5) {$\mu_m$};
\node[font=\tiny] at (2,7.9) {$\mu_1$};
\node[font=\tiny] at (3.3,7.75) {$\lambda$};
\node[font=\tiny] at (.6,4.5) {$\kappa$};
\node[font=\tiny] at (2.9,8.2) {$\xi_0$};
\node[rotate=90] at (.1,4) {$2\tau(K)-n$};
\node[rotate=90] at (5,4) {$2\tau(K)$};
\node[rotate=90] at (5,8.5) {$n$};
\node[font=\tiny] at (1.3,0) {$\{i\} \times \R$};
\node[font=\tiny] at (2.8,0) {$\{i+1\} \times \R$};

\node[font=\tiny] at (.8,4.1) {$\rho_{123}$};
\node[font=\tiny] at (1.7,4.2) {$\rho_3$};
\node[font=\tiny] at (2.1,5.5) {$\rho_{23}$};
\node[font=\tiny] at (2.1,6.7) {$\rho_{23}$};
\node[font=\tiny] at (2.3,8) {$\rho_1$};
\node[font=\tiny] at (3.2,8) {$\rho_2$};
\end{tikzpicture}\vspace{1in}
    \caption{The unstable portion of $\alpha(K,n)$ with $\tau(K)\geq 0$ and $\epsilon(K)=1$ and $2\tau(K)>n.$}\label{unstabletauposepsilon1nless2tau}
  \end{minipage}
  \hspace{2in}
  \begin{minipage}[b]{0.3\textwidth}
 \begin{tikzpicture}
\node[anchor=south west,inner sep=0] at (0,0)    {\includegraphics[width=1.2\textwidth]{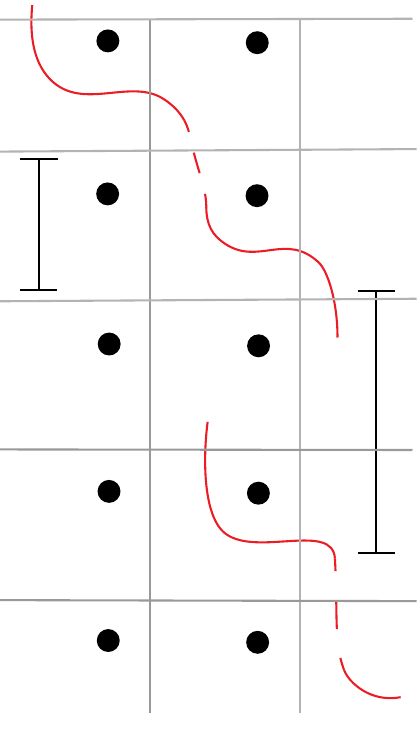}};
\node[rotate=90] at (5,4) {$2\tau(K)$};
\node[font=\tiny] at (3.6,5.8) {$\tikzcirc{2pt}$};
\node[font=\tiny] at (3.6,2.3) {$\tikzcirc{2pt}$};
\node[font=\tiny] at (2.5,3.4) {$\tikzcirc{2pt}$};
\node[font=\tiny] at (1.8,7.65) {$\tikzcirc{2pt}$};
\node[font=\tiny] at (2.35,7) {$\tikzcirc{2pt}$};
\node[font=\tiny] at (4,5.2) {$\tikzcirc{2pt}$};
\node[font=\tiny] at (.4,8.6) {$\tikzcirc{2pt}$};
\node[font=\tiny] at (.6,8.6) {$\kappa$};
\node[font=\tiny] at (3.8,5.2) {$\lambda$};
\node[font=\tiny] at (3.8,5.9) {$\xi_0$};
\node[font=\tiny] at (2.6,7.1) {$\mu_m$};
\node[font=\tiny] at (2,7.8) {$\eta_0$};
\node[font=\tiny] at (3.4,2.1) {$\eta_0$};
\node[font=\tiny] at (2.35,3.25) {$\kappa$};

\node[font=\tiny] at (1.6,0) {$\{i\} \times \R$};
\node[font=\tiny] at (3.4,0) {$\{i+1\} \times \R$};

\node[font=\tiny] at (1,7.55) {$\rho_{123}$};
\node[font=\tiny] at (4.2,5.5) {$\rho_2$};
\node[font=\tiny] at (2.6,5.8) {$\rho_{123}$};
\node[font=\tiny] at (2.2,7.6) {$\rho_2$};
\node[font=\tiny] at (2.3,6.5) {$\rho_{23}$};

\node[rotate=90] at (0,6) {$n-2\tau(K)$};
\end{tikzpicture}\vspace{1in}
    \caption{The unstable portion of $\alpha(K,n)$ with $\tau(K)\geq 0$ and $\epsilon(K)=1$ and $n\geq 2\tau(K).$}\label{unstabletauposepsilon1ngreater2tau}
  \end{minipage}
\end{figure}

\begin{lemma}\label{curveshape1}
    Suppose $\tau(K)\geq 0$ and $\epsilon(K)=1$. If $n<2\tau(K)$, then the essential component of the immersed curve has slope $2\tau(K)-n$ and turns down immediately after passing through $\{i+1\} \times \R$, see Figure \ref{unstabletauposepsilon1nless2tau}. If $n \geq 2\tau(K)$, then the essential component of the immersed curve has slope $2\tau(K)-n$ and turns down immediately after crossing through $\{i+1\} \times \R$, see Figure \ref{unstabletauposepsilon1ngreater2tau}. 
\end{lemma}

\begin{figure}[!tbp]
  \centering
  \begin{minipage}[b]{0.3\textwidth}
  \begin{tikzpicture}

\node[anchor=south west,inner sep=0] at (0,0)    {\includegraphics[width=1.2\textwidth]{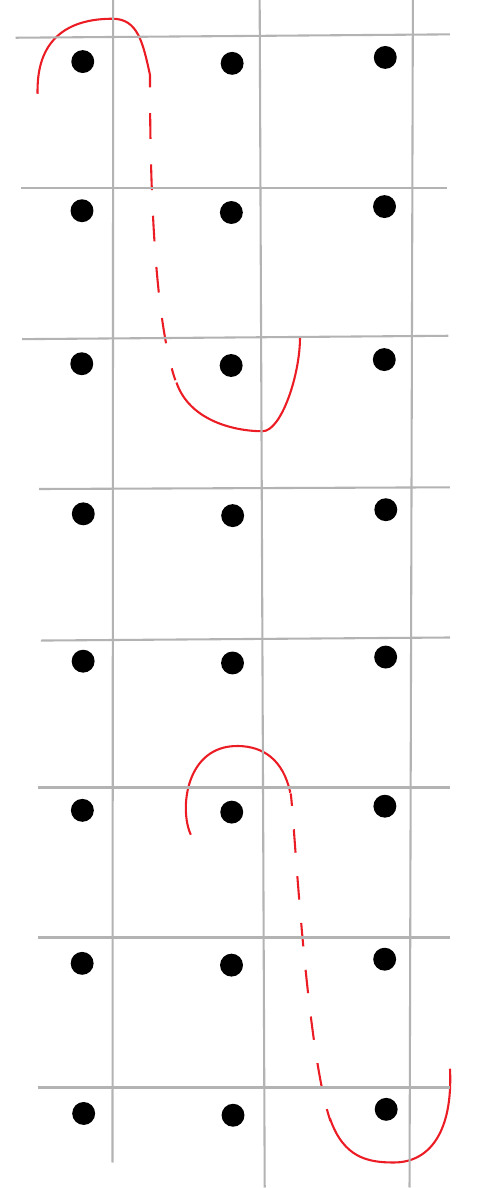}};
\node[font=\tiny] at (2.75,7.95) {$\tikzcirc{2pt}$};

\node[font=\tiny] at (2.75,4.6) {$\tikzcirc{2pt}$};
\node[font=\tiny] at (3,4.2) {$\tikzcirc{2pt}$};
\node[font=\tiny] at (3.2,2.6) {$\tikzcirc{2pt}$};
\node[font=\tiny] at (3.4,1.05) {$\tikzcirc{2pt}$};
\node[font=\tiny] at (4.3,.3) {$\tikzcirc{2pt}$};
\node[font=\tiny] at (2,4.2) {$\tikzcirc{2pt}$};
\node[font=\tiny] at (1.75,8.9) {$\tikzcirc{2pt}$};
\node[font=\tiny] at (3.2,8.9) {$\tikzcirc{2pt}$};

\node[font=\tiny] at (2.7,7.7) {$\xi_0$};
\node[font=\tiny] at (3.2,8) {$\rho_{3}$};
\node[font=\tiny] at (2.1,7.85) {$\rho_{123}$};
\node[font=\tiny] at (2.6,4.7) {$\eta_0$};
\node[font=\tiny] at (3.15,4.5) {$\rho_{3}$};
\node[font=\tiny] at (2.2,4.7) {$\rho_{1}$};
\node[font=\tiny] at (1.5,10) {$\rho_{23}$};
\node[font=\tiny] at (3.5,3.3) {$\rho_{23}$};

\node[font=\tiny] at (1.2,0) {$\{i\} \times \R$};
\node[font=\tiny] at (2.8,0) {$\{i+1\} \times \R$};

\end{tikzpicture}
    \caption{The unstable portion of $\alpha(K,n)$ with $\tau(K)\geq 0$ and $\epsilon(K)=-1$ and $n\geq2\tau(K).$}\label{unstabletauposepsilonnegngeq2tau}
  \end{minipage}
  \hspace{2in}
  \begin{minipage}[b]{0.3\textwidth}
 \begin{tikzpicture}

\node[anchor=south west,inner sep=0] at (0,0)    {\includegraphics[width=1.2\textwidth]{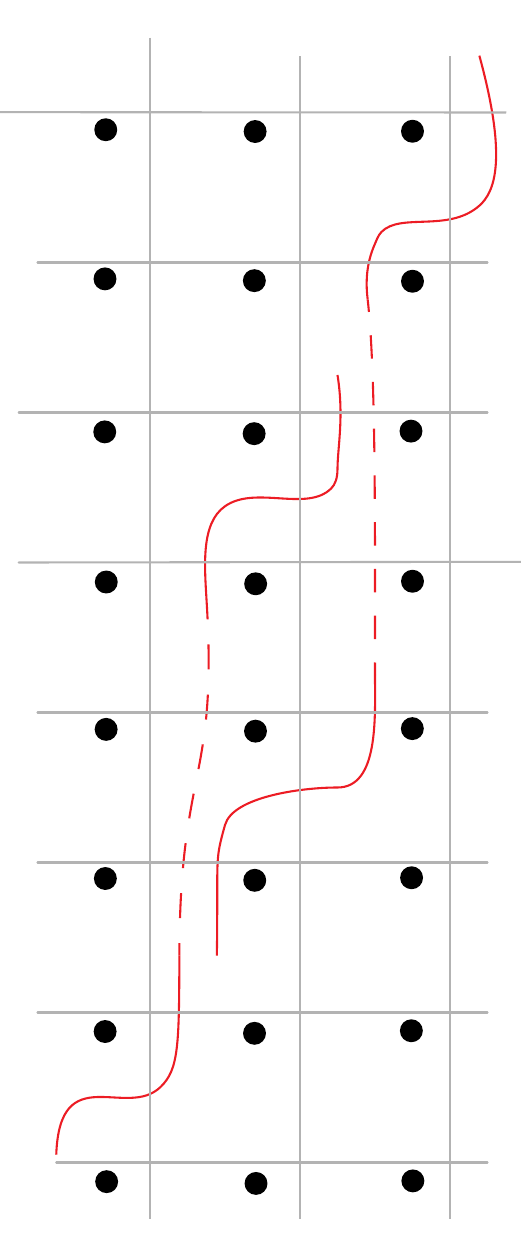}};
\node[font=\tiny] at (2.9,7.25) {$\tikzcirc{2pt}$};
\node[font=\tiny] at (3.3,8.1) {$\tikzcirc{2pt}$};
\node[font=\tiny] at (2.9,4.4) {$\tikzcirc{2pt}$};
\node[font=\tiny] at (2.15,3.75) {$\tikzcirc{2pt}$};
\node[font=\tiny] at (2,6.7) {$\tikzcirc{2pt}$};
\node[font=\tiny] at (2,5.2) {$\tikzcirc{2pt}$};
\node[font=\tiny] at (1.75,2.3) {$\tikzcirc{2pt}$};
\node[font=\tiny] at (1.8,3.75) {$\tikzcirc{2pt}$};
\node[font=\tiny] at (1.45,1.5) {$\tikzcirc{2pt}$};
\node[font=\tiny] at (3.65,5.2) {$\tikzcirc{2pt}$};
\node[font=\tiny] at (3.65,6.65) {$\tikzcirc{2pt}$};
\node[font=\tiny] at (3.65,8.1) {$\tikzcirc{2pt}$};

\node[font=\tiny] at (1.4,0) {$\{i\} \times \R$};
\node[font=\tiny] at (2.9,0) {$\{i+1\} \times \R$};

\node[font=\tiny] at (.8,1.6) {$\rho_1$};
\node[font=\tiny] at (1.8,1.7) {$\rho_3$};
\node[font=\tiny] at (1.4,3) {$\rho_{23}$};
\node[font=\tiny] at (1.55,4.5) {$\rho_{23}$};
\node[font=\tiny] at (1.6,6) {$\rho_{23}$};
\node[font=\tiny] at (2,7.2) {$\rho_1$};
\node[font=\tiny] at (3.1,7.5) {$\rho_3$};
\node[font=\tiny] at (3,7) {$\xi_0$};
\node[font=\tiny] at (3,4.2) {$\eta_0$};
\node[font=\tiny] at (1.8,6.8) {$\mu_1$};
\node[font=\tiny] at (1.8,5.3) {$\mu_i$};
\node[font=\tiny] at (1.6,3.8) {$\mu_j$};
\node[font=\tiny] at (2.3,3.8) {$\kappa$};
\node[font=\tiny] at (1.6,2.2) {$\mu_m$};
\node[font=\tiny] at (1.4,1.65) {$\eta_0$};
\node[font=\tiny] at (3.4,8.2) {$\lambda$};

\end{tikzpicture}
    \caption{The unstable portion of $\alpha(K,n)$ with $\tau(K)\geq 0$ and $\epsilon(K)=-1$ and $n\leq 0\leq 2\tau(K).$}\label{unstablechaintauposepsilonnegnneg}
  \end{minipage}
\end{figure}

\begin{proof}
    When $\epsilon(K)=1$, by \cite{Hom} there is a reduced horizontally simplified basis so that the vertically distinguished generator $\xi_0$ of $\CFK^-(K)$ is an element of this horizontally simplified basis and occurs at the end of a horizontal arrow (symmetrically the horizontally distinguished generator $\eta_0$ occurs at the end of a vertical arrow). If $\tau(K)\geq 0$ then the algorithm from \cite[Theorem 11.26]{LOT} shows that the type $D$ structure contains the portion shown in Figure \ref{typeDtauposepsilon1}, where the dotted arrow is replaced by the appropriate unstable chain. 
    
Then the algorithm in \cite[Sections 2.3-2.4]{HRW1} shows that the essential component of the immersed curve lifted to the cover $\R^2 \setminus \pi^{-1}(z)$ has the form shown. In Figure \ref{unstabletauposepsilon1nless2tau} and \ref{unstabletauposepsilon1ngreater2tau} we see the resulting curves for $n<2\tau(K)$ and $n \geq 2\tau(K)$ respectively and indicate how the curves are built from the type $D$ structure. Intersections with the vertical lines in the figure correspond to generators of $\iota_0\CFDhat(S^3- \nu(K),n)$ and the intersections with the horizontal lines correspond to generators of $\iota_1\CFDhat(S^3-\nu(K),n)$. If $\delta(x)=\rho_{I}\otimes y+\cdots$, then there is an arc $\rho_I$ from $x$ to $y$, as described in the figures. \qedhere

\end{proof}

Similarly, we can show
\begin{lemma}\label{curveshape2}
    Suppose $\tau(K)\geq 0$ and $\epsilon(K)=-1$. If $n \geq 2\tau(K)$, then the essential component of the immersed curve has slope $2\tau(K)-n$ and turns up after crossing through $\{i+1\} \times \R$, see Figure \ref{unstabletauposepsilonnegngeq2tau}. If $n<2\tau(K)$, then the essential component of the immersed curve has slope $2\tau(K)-n$ and turns up after crossing through $\{i+1\} \times \R$, see Figure \ref{unstablechaintauposepsilonnegnneg}.
\end{lemma}

The statements about the form of the essential component of the immersed curve in the case $\tau(K)\leq0$ are similar. In summary, the essential component of the immersed curve has slope $2\tau(K)-n$ and turns up, down or continues straight depending on whether $\epsilon(K)=-1,1$ or $0$.

\subsection{(1,1)-Unknot Patterns}\label{(1,1)}
In this section, we review some notation and results about $(1,1)$ unknot patterns. In the case that the $(1,1)$ pattern knot $P$ is an unknot pattern, meaning that $P(U) \sim U$, Chen showed that the $\beta$ curve for the genus $1$ doubly-pointed Heegaard diagram for $P$ can be encoded by two integers $(r,s)$, where $\gcd(2r-1,s+1)=1$ \cite[Theorem 5.1]{Chen}. In this parametrization, $r$ denotes the number of rainbows and $s$ denotes the number of stripes (see \cite[Figure 15]{Chen}). The pattern described by the pair $(r,s)$ corresponds to the two bridge link $\mathfrak{b}(2|s|+4|r|,\epsilon(r)(2|r|-1))$ \cite[Theorem 5.4]{Chen}.

For example, see Figure \ref{Mazur(0,2)} where we have drawn the doubly pointed Bordered Heegaard diagram for the unknot pattern described by the pair $(r,s)=(4,2)$, and Figure \ref{Q02withknot} where we have drawn the same genus $1$ bordered Heegaard diagram with the pattern knot that it determines. In general, the knot determined by the $(1,1)$ unknot pattern given by the pair $(r,s)$ has a presentation with $r-1$ rainbow arcs and $s+1$ stripes, see Figure \ref{genmazurknot}. 

\begin{figure}[!tbp]
  \centering
  \begin{minipage}[b]{0.3\textwidth}
  \begin{tikzpicture}
\node[anchor=south west,inner sep=0] at (0,0)    {\includegraphics[width=1.2\textwidth]{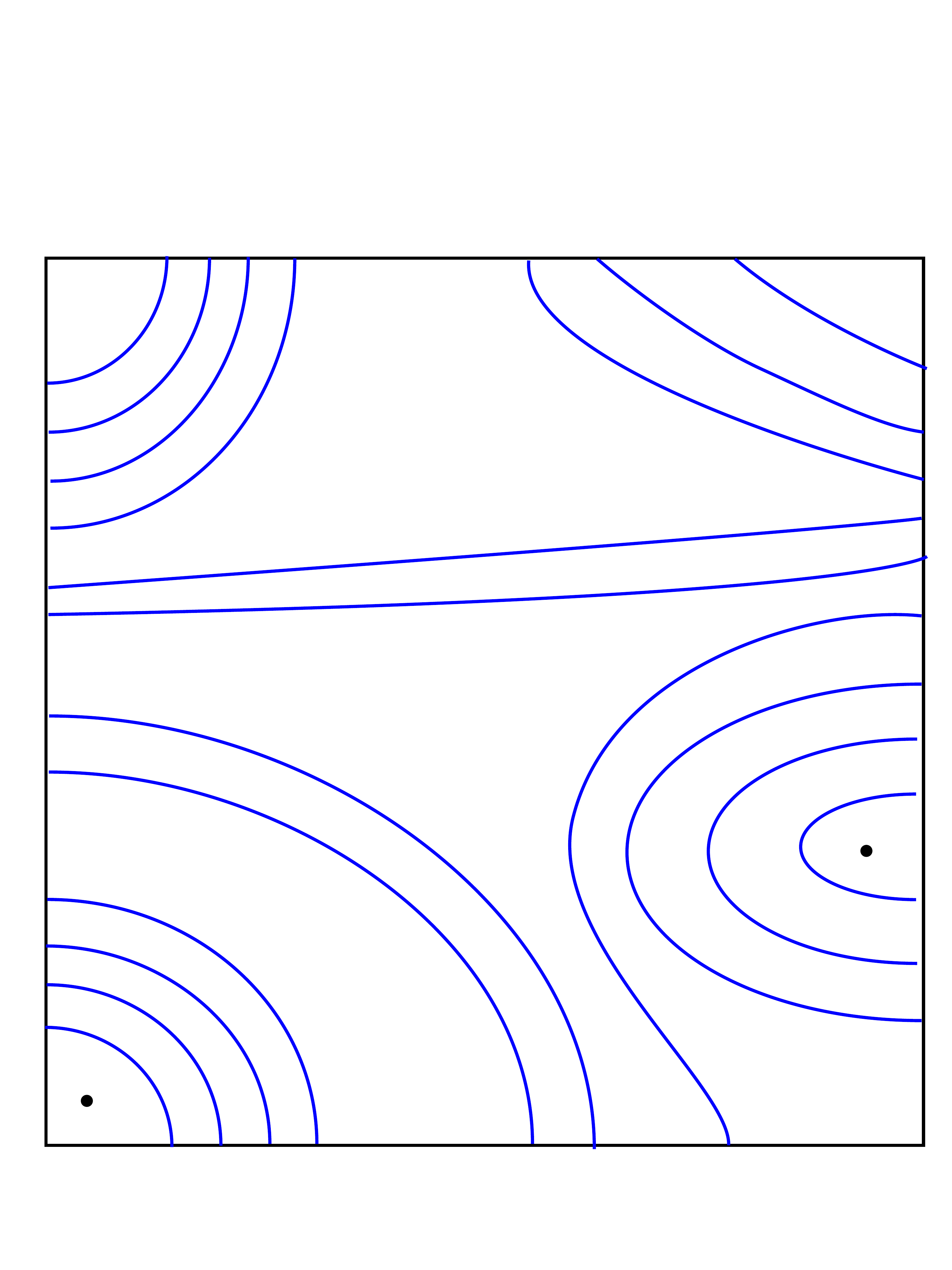}};
\node[font=\tiny] at (.5,1) {$w$};
\node[font=\tiny] at (4.7,2.25) {$z$};
\end{tikzpicture}
    \caption{The $(1,1)$ pattern determined by the pair $(r,s)=(4,2).$}\label{Mazur(0,2)}
  \end{minipage}
  \hspace{1.5in}
  \begin{minipage}[b]{0.3\textwidth}
  \begin{tikzpicture}
\node[anchor=south west,inner sep=0] at (0,0)    {\includegraphics[width=1.4\textwidth]{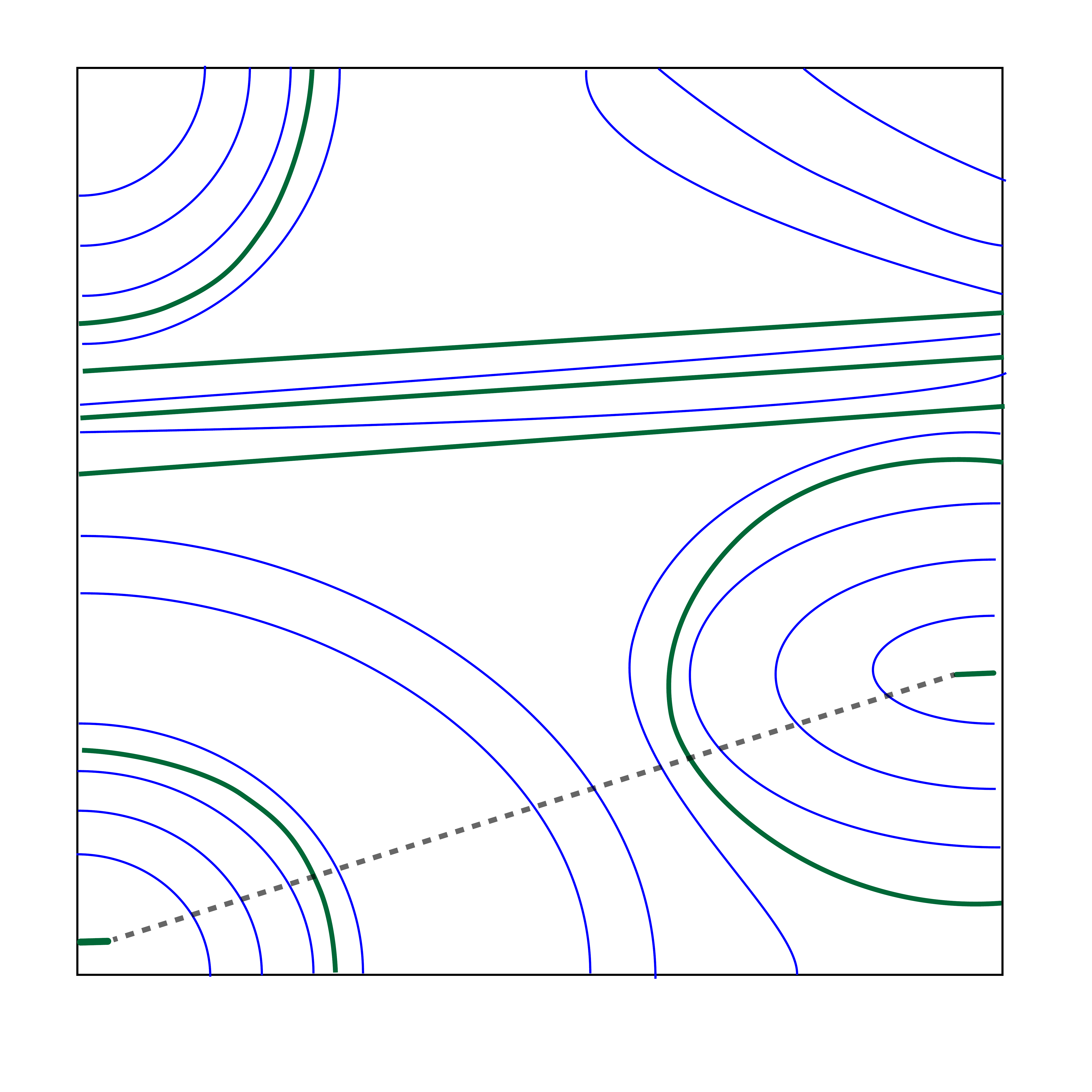}};
\end{tikzpicture}
    \caption{The knot in green in $S^1 \times D^2$ determined by the $(1,1)$ pattern from Figure \ref{Mazur(0,2)}.}\label{Q02withknot}
  \end{minipage}
\end{figure}

As above, let $\CFKhat(\alpha,\beta,z,w)$ denote the intersection Floer homology of the two curves $\alpha$ and $\beta$ in $T^2\setminus \{z,w\}$ as described in \cite[Theorem 1.2]{Chen}. The generators of $\CFKhat(\alpha,\beta,z,w)$ are the intersection points of the two curves, and the differential counts embedded bigons with left boundary on the $\beta$ curve and right boundary on the $\alpha$ curve. As proved in \cite{chenhanselman}, we can recover the $UV=0$ quotient of the full knot Floer complex by considering disks that contain either $z$ or $w$ basepoints (but not both) and label them by $V$ and $U$ respectively. The component of the differential induced by counting bigons crossing the $z$ basepoint will be called \emph{vertical} differentials and denoted $\partial^v$, and those crossing the $w$ basepoint \emph{horizontal} differentials and denoted $\partial^h$.

\begin{figure}[!tbp]
\centering
  \begin{tikzpicture}
\node[anchor=south west,inner sep=0] at (0,0)   {\includegraphics[scale=.3]{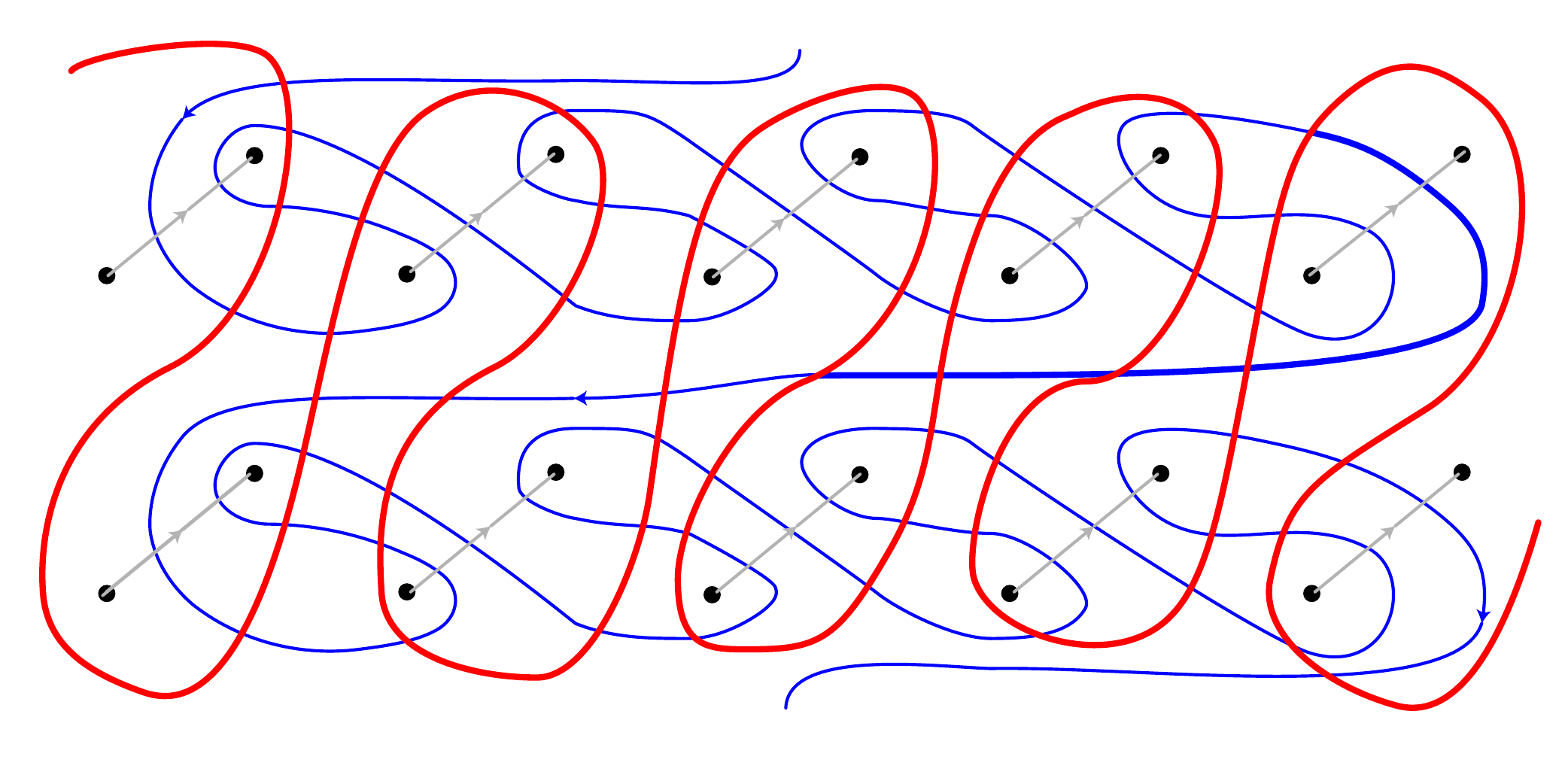}};

  \node at (5.6,2.8) {$\tikzcirc{2pt}$};
  \node at (1.55,3.2) {$\tikzcirc{2pt}$};
  \node at (1.95,4.7) {$\tikzcirc{2pt}$};
  \node at (2.23,3.05) {$\tikzcirc{2pt}$};
  \node at (3.8,3.3) {$\tikzcirc{2pt}$};
   \node at (8.85,4.4) {$\tikzcirc{2pt}$};
   
   \node at (1.6,1) {$\tikzcirc{2pt}$};
   
   \node at (2.15,2.6) {$\tikzcirc{2pt}$};
   \node at (1.9,1.75) {$\tikzcirclee{2pt}$};
   \node at (2.45,3.8) {$\tikzcirclee{2pt}$};
\node at (2.59,4.15) {$\tikzcirc{2pt}$};
    \node[font=\small, red] at (1,5.2) {$\talpha(T_{2,3},0)$};
    \node[font=\small, blue] at (5,5.2) {$\tbeta$};
     \node[font=\tiny, blue] at (3.5,2.7) {$A=0$};
     \node[font=\tiny, blue] at (3.5,4.85) {$A=3$};
     \node[font=\tiny, blue] at (7.5,.65) {$A=-3$};
   \node[font=\tiny] at (2.75,4.18) {$k$};
    \node[font=\tiny] at (5.7,2.7) {$c$};
  \node[font=\tiny] at (1.55,3) {$e$};
  \node[font=\tiny] at (2.1,4.55) {$f$};
  \node[font=\tiny] at (8.9,4.2) {$a$};
  \node[font=\tiny] at (4,3.3) {$b$};
   \node[font=\tiny] at (2.3,2.9) {$d$};
   \node[font=\tiny] at (2.3,2.5) {$g$};
\node[font=\tiny] at (1.5,.85) {$h$};
\node[font=\tiny] at (2,1.6) {$i$};
\node[font=\tiny] at (2.5,3.6) {$j$};

    \node[font=\tiny] at (.5,3.35) {$w$};
    \node[font=\tiny] at (.5,1.2) {$w$};
    
    \node[font=\tiny] at (2.9,3.35) {$w$};
    \node[font=\tiny] at (2.9,1.2) {$w$};
    
    \node[font=\tiny] at (5,3.35) {$w$};
    \node[font=\tiny] at (5,1.2) {$w$};
    
    \node[font=\tiny] at (7,3.35) {$w$};
    \node[font=\tiny] at (7,1.2) {$w$};
    \node[font=\tiny] at (9,3.35) {$w$};
    \node[font=\tiny] at (9,1.2) {$w$};
    
 \node[font=\tiny] at (1.9,2.1) {$z$};
    \node[font=\tiny] at (1.9,4.2) {$z$};
     \node[font=\tiny] at (3.9,2.1) {$z$};
    \node[font=\tiny] at (3.9,4.2) {$z$};
      \node[font=\tiny] at (6,2.1) {$z$};
    \node[font=\tiny] at (6,4.2) {$z$};
      \node[font=\tiny] at (8,2.1) {$z$};
    \node[font=\tiny] at (8,4.2) {$z$};
      \node[font=\tiny] at (10.1,2.1) {$z$};
    \node[font=\tiny] at (10.1,4.2) {$z$};

\end{tikzpicture}
    \caption{The lifted pairing diagram for $\CFKhat(\talpha(T_{2,3},0),\tbeta, w,z).$}\label{0framedpairing}
\end{figure}

Now, if $\pi: \R^2 \to T^2$ denotes the universal cover of the torus, let $\tbeta$ be a connected component of $\pi^{-1}(\beta)$ in $\R^2\setminus \{\pi^{-1}(z),\pi^{-1}(w)\}$ and let $\talpha(K,n)$ be a lift of $\alpha(K,n)$ to $\R^2$, as in Figures \ref{unstabletauposepsilon1nless2tau}-\ref{unstablechaintauposepsilonnegnneg}. Then by \cite[Proof of Theorem 1.2]{Chen} $\CFKhat(\talpha,\tbeta, \pi^{-1}(z),\pi^{-1}(w)) \cong \CFKhat(\alpha,\beta,z,w)$. Indeed, it is easy to see that there is a correspondence at the level of generators, and it is similarly straightforward to see that differentials on both sides agree. See Figures \ref{0framedpairing}, \ref{1framedtrefoilpairing}, and \ref{minus1framedtrefoilpairing}. Throughout we work with the lifted pairing diagram. We assume that the intersection between the two curves is \emph{reduced}, meaning that the only bigons contributing to the differential are the bigons that cross either the $z$ or the $w$ basepoint, this is easily obtained by an isotopy of $\alpha(K,n)$ across the Whitney disks that don't contain a basepoint. With these conventions, the following is proved in \cite[Theorem 1.2 and Lemma 4.1]{Chen} and \cite[Theorem 6.1]{chenhanselman}:

\begin{theorem}\label{Pairing}
For $P$ a $(1,1)$ pattern, $\HFKhat(S^3,P_n(K))=\CFKhat(\talpha(K,n),\tbeta(P))$ and moreover $\CFK_{\F[U,V]/UV}(S^3,P_n(K)) \cong (\CFKhat(\talpha(K,n),\tbeta,\pi^{-1}(z),\pi^{-1}(w)),\partial^v, \partial^h)$. Furthermore, given two intersection points $x$ and $y$ between $\talpha(K,n)$ and $\tbeta(P)$, $A(y)-A(x)=\ell_{x,y}\cdot\delta_{w,z}$, where $\ell_{x,y}$ is an arc on the $\beta$ curve that goes from $x$ to $y$ and $A$ denotes the Alexander grading of generators of the knot Floer homology. 
\end{theorem}

See Figure \ref{0framedpairing} for an example, where we have drawn the lifted pairing diagram for the satellite knot $Q^{0,3}_0(T_{2,3})$. In that figure, we have labelled some intersection points, and drawn the $\delta_{w,z}$ arcs. Theorem \ref{Pairing} implies that the intersection points are in bijection with the generators of $\HFKhat(S^3,Q^{0,3}_0(T_{2,3}))$. Moreover, by taking an arc along the $\beta$ curve from $c$ to $a$, for example, we see that $A(a)-A(c)=-1$. The knot Floer homology has a symmetry given by $\HFKhat(S^3,K,A) \cong \HFKhat(S^3,K,-A)$, and we can see this symmetry in the pairing diagram by rotating the whole picture by $\pi$ and exchanging the $w$ and $z$ basepoints. It follows that $A(c)=0$ and we can always upgrade the relative Alexander grading given by Theorem \ref{Pairing} to an absolute Alexander grading. In Figure \ref{0framedpairing} we find $A(b)=3$, $A(e)=A(d)=4$ and $A(f)=3$. 

\begin{figure}[!tbp]
\centering
  \begin{tikzpicture}
\begin{tikzcd}
	d & j & b \\
	& k \\
	& h & i \\
	g & l
	\arrow["{V^3}"{description}, from=1-1, to=4-1]
	\arrow["V"{description}, from=3-2, to=4-1]
	\arrow["U"', from=1-2, to=1-1]
	\arrow["U"', from=3-3, to=3-2]
	\arrow["V"{description}, from=1-2, to=2-2]
	\arrow["V"{description}, from=1-3, to=2-2]
	\arrow["V"{description}, from=3-3, to=4-2]
	\arrow["U", from=4-2, to=4-1]
\end{tikzcd}
\end{tikzpicture}
    \caption{The piece of the complex $\CFK_{\F[U,V]/UV}(S^3,Q^{0,3}_0(T_{2,3}))$ that contains the intersection point $d$ with $A(d)=\tau(Q^{0,3}_0(T_{2,3})),$ and $d+h$ generates $\HFhat(S^3)$.}\label{0framedpairingcomplex}

\end{figure}
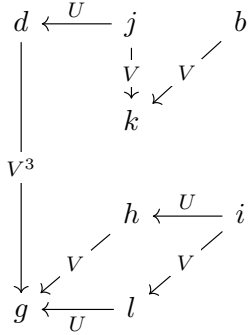
Another consequence of Theorem \ref{Pairing} is that since we can recover the $UV=0$ quotient of the full knot Floer complex, we can compute both $\tau$ and $\epsilon$ of satellite knots with $(1,1)$-patterns. We return to this in Section \ref{tausectionMazur} later, but we remark here that by counting disks that cross only the $z$ basepoint in Figure \ref{0framedpairing}, the intersection points $d$, $g$, and $h$ form a subcomplex of $\CFKhat(Q^{0,3}_0(T_{2,3}))$ such that the cycle $d+h$ generates $\HFhat(S^3)$ (obtained by setting $V=1$ in the above subcomplex). This cycle can be extended to a vertically simplified basis of $\CFK^{-}(Q^{0,3}_0(T_{2,3}))$ in the sense of \cite[Section 2]{Hom}. Moreover, the intersection points $i$ and $j$ satisfy $\partial^h(i+j)=d+h$, so the distinguished element of the vertically simplified basis is in the image of the horizontal differential and this implies \cite[Section 3]{Hom} that $\epsilon(Q^{0,3}_0(T_{2,3}))=1$. Further, it is easy to see that the intersection point $d$ satisfies $A(d)=\tau(Q^{0,3}_0(K))$. See Figure \ref{0framedpairingcomplex}, where we have indicated a portion of the complex over $\F[U,V]/UV$. Note that the above argument only involved intersection points between the unstable portion of the curve $\alpha(T_{2,3},0)$ in the first column and the $\beta$ curve. We return to this observation in section \ref{tausectionMazur}, where we see that this holds in general for the patterns given by the $\beta$ curve $\beta(i,j)$.

The pairing diagrams and their lifts become more complicated when we consider knots with non-zero framing since the unstable chain gets longer for most values of $n$, which we need for computing the knot Floer homology of satellites with $n$-twisted patterns. For example, see Figures \ref{1framedtrefoilpairing} and \ref{minus1framedtrefoilpairing} where we have the pairing diagram for $Q^{0,3}_{-1}(T_{2,3})$ and $Q^{0,3}_{1}(T_{2,3})$. In those figures, the intersection point $c$ satisfies $A(c)=0$ and we have indicated some of the Alexander gradings of intersection points.

    \begin{figure}[!tbp]
  \centering
  \begin{minipage}[b]{0.3\textwidth}
  \begin{tikzpicture}
\node[anchor=south west,inner sep=0] at (0,0)    {\includegraphics[width=1.2\textwidth]{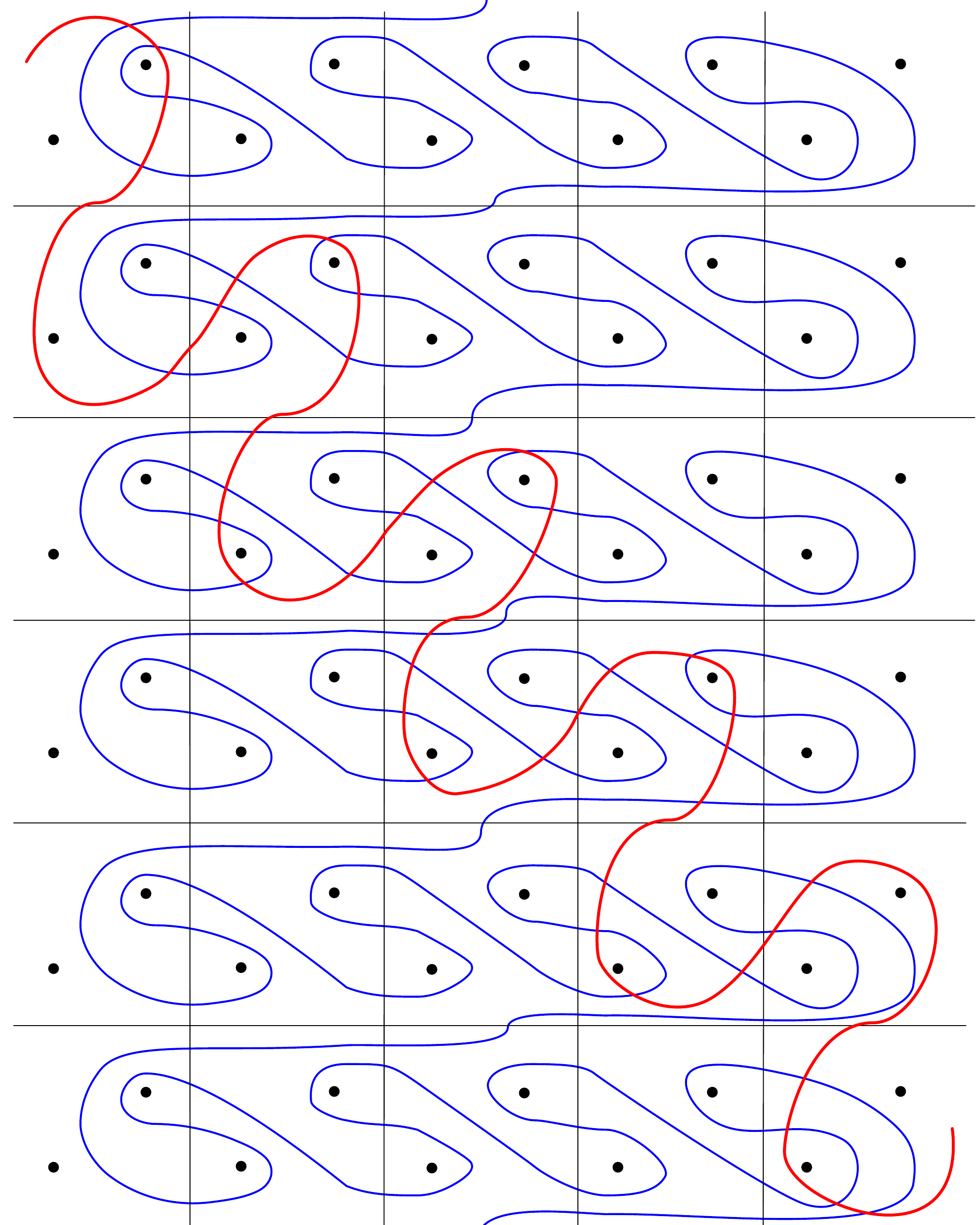}};
\node at (2.2,3.05) {$\tikzcirc{2pt}$};
\node at (.7,5.45) {$\tikzcirc{2pt}$};
\node at (.9,4.4) {$\tikzcirc{2pt}$};

\node[font=\tiny] at (1.4,6.4) {$A=9$};
\node[font=\tiny] at (1.4,5.3) {$A=6$};
\node[font=\tiny] at (1.4,4.2) {$A=3$};
\node[font=\tiny, red] at (.2,6.3) {$\talpha(T_{2,3},1)$};
\end{tikzpicture}
    \caption{The pairing diagram for $\CFKhat(\alpha(T_{2,3},1),\beta(Q^{0,3})).$}\label{1framedtrefoilpairing}
  \end{minipage}\hspace{2in}
  \begin{minipage}[b]{0.3\textwidth}
    \begin{tikzpicture}
\node[anchor=south west,inner sep=0] at (0,0)    {\includegraphics[width=1.2\textwidth]{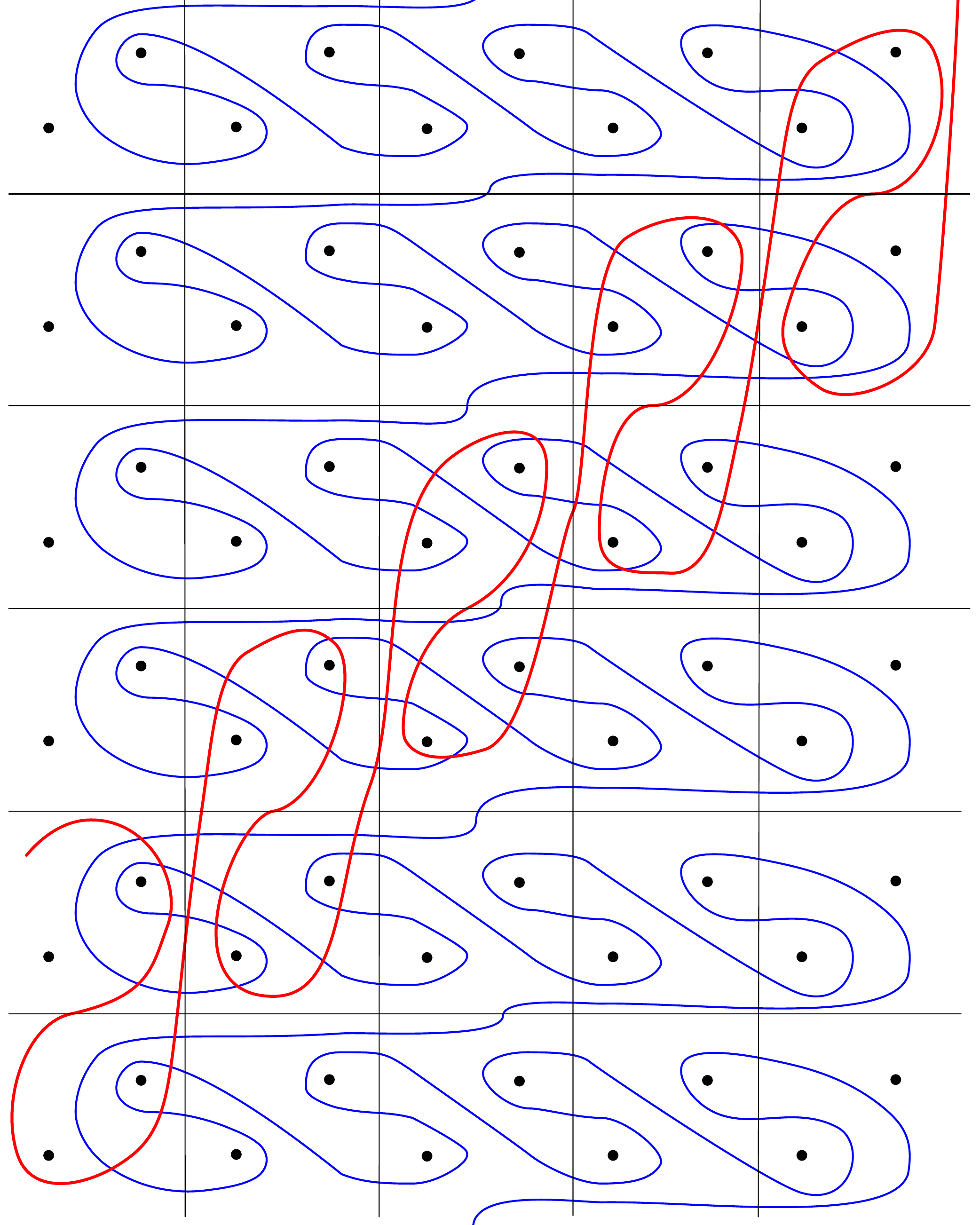}};
\node at (2.3,3.1) {$\tikzcirc{2pt}$};
\node at (1.1,2.3) {$\tikzcirc{2pt}$};
\node at (4.05,5.5) {$\tikzcirc{2pt}$};
\node[font=\tiny, red] at (4.8,6.4) {$\talpha(T_{2,3},-1)$};
\node[font=\tiny] at (1.4,6.4) {$A=9$};
\node[font=\tiny] at (1.4,5.3) {$A=6$};
\node[font=\tiny] at (1.4,4.2) {$A=3$};
\end{tikzpicture}
    \caption{The pairing diagram for $\CFKhat(\alpha(T_{2,3},-1), \beta(Q^{0,3})).$}\label{minus1framedtrefoilpairing}
  \end{minipage}
\end{figure}

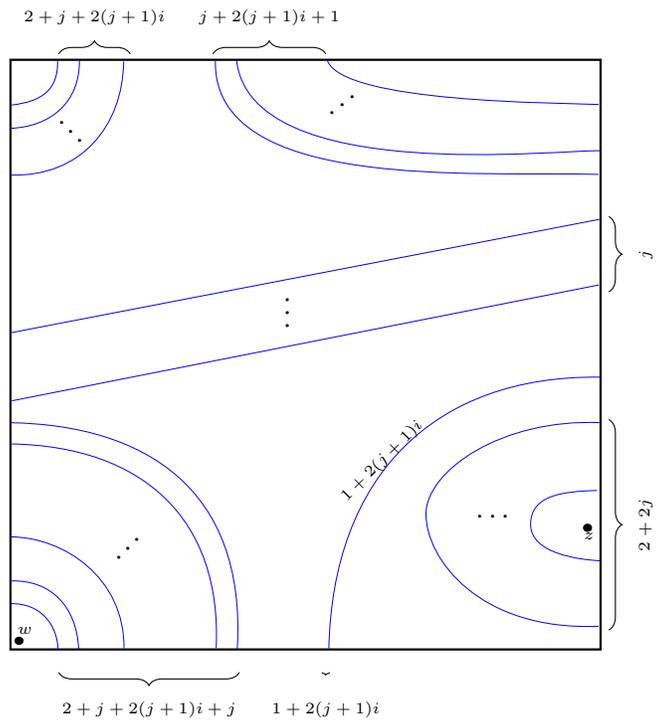
\begin{figure}[!tbp]
\begin{center}
  \begin{tikzpicture}[scale=.5]
  \draw [decorate,decoration={brace,amplitude=5pt,mirror,raise=4ex}]
  (.9,1.2) -- (5.7,1.2) node[font=\tiny, midway,yshift=-3em]{$2+j+2(j+1)i+j$};
  \draw [decorate,decoration={brace,amplitude=5pt,mirror,raise=4ex}]
  (14.2,1) -- (14.2,6.6) node[font=\tiny, rotate=90, midway,yshift=-3em]{$2+2j$};
  \draw [decorate,decoration={brace,amplitude=5pt,mirror,raise=4ex}]
  (14.2,10) -- (14.2,12) node[font=\tiny, rotate=90, midway,yshift=-3em]{$j$};
  \draw [decorate,decoration={brace,amplitude=1pt,mirror,raise=4ex}]
  (7.9,1.2) -- (8.1,1.2) node[font=\tiny, midway,yshift=-3em]{$1+2(j+1)i$};
  \draw [decorate,decoration={brace,amplitude=5pt,raise=4ex}]
  (.9,15) -- (2.8,15) node[font=\tiny, midway,yshift=3em]{$2+j+2(j+1)i$};
   \draw [decorate,decoration={brace,amplitude=5pt,raise=4ex}]
  (5,15) -- (8,15) node[font=\tiny, midway,yshift=3em]{$j+2(j+1)i+1$};
  \node at (12.5,4) {$\cdots$};
  \node[rotate=45] at (2.8,3.2) {$\cdots$};
  \node[font=\tiny, rotate=45] at (9.5,5.5) {$1+2(j+1)i$};
\node[rotate=90] at (7,9.5) {$\cdots$};
\node[rotate=-45] at (1.25,14.2) {$\cdots$};
\node[rotate =39] at (8.5, 15) {$\cdots$};
  \node[font=\tiny] at (0,1) {$w$};
  \node[font=\tiny] at (15,3.5) {$z$};
\node[anchor=south west,inner sep=0] at (-1,0)    {\includegraphics[width=.6\textwidth]{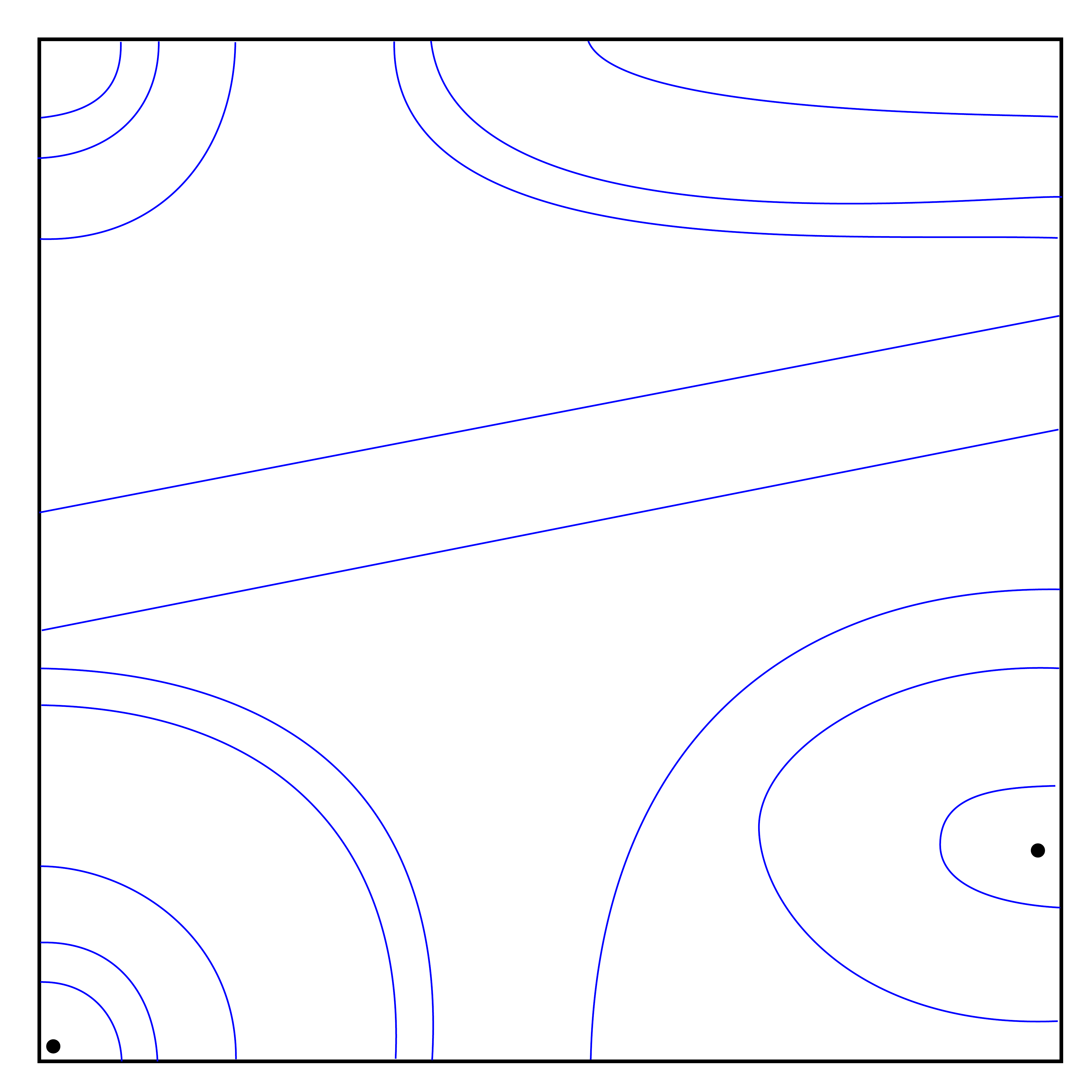}};
\end{tikzpicture}
    \caption{The $(1,1)$ pattern that determines the pattern knot $Q^{i,j}$. Figure \ref{Mazur(0,2)} shows the case $i=0$ and $j=2$.}\label{Mazur(i,j)}
    \end{center}
\end{figure}

\subsection{The Curves $\beta(i,j)$}\label{beta}

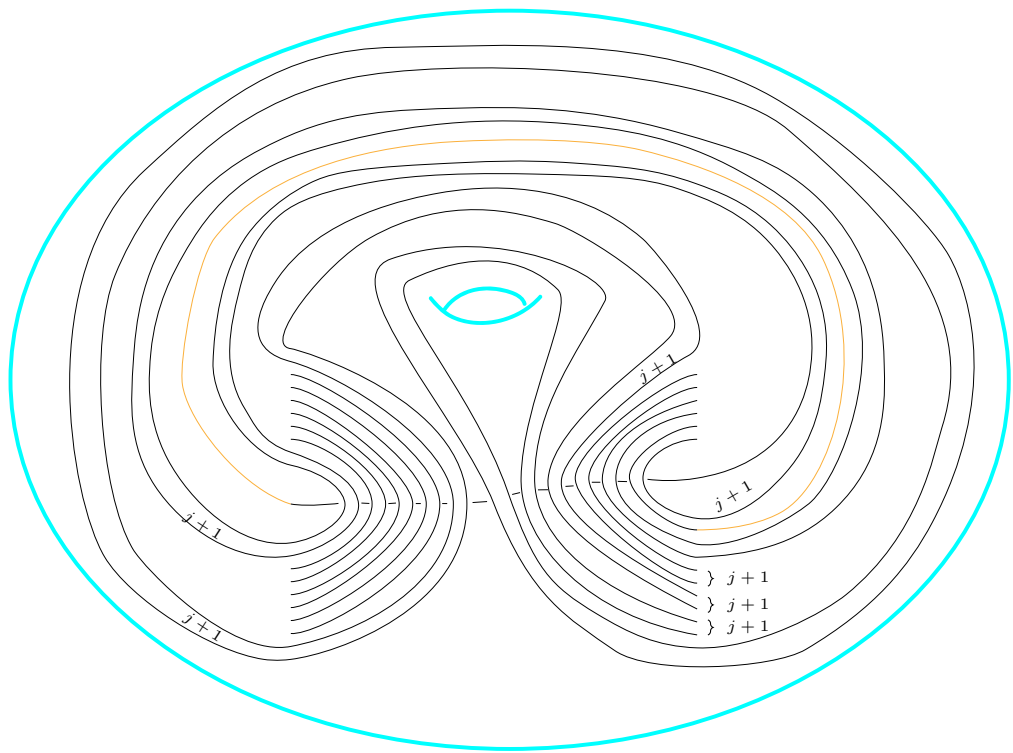
\begin{figure}[!tbp]
  \centering
  \begin{tikzpicture}

\node[anchor=south west,inner sep=0] at (0,0)    {\includegraphics[width=\textwidth]{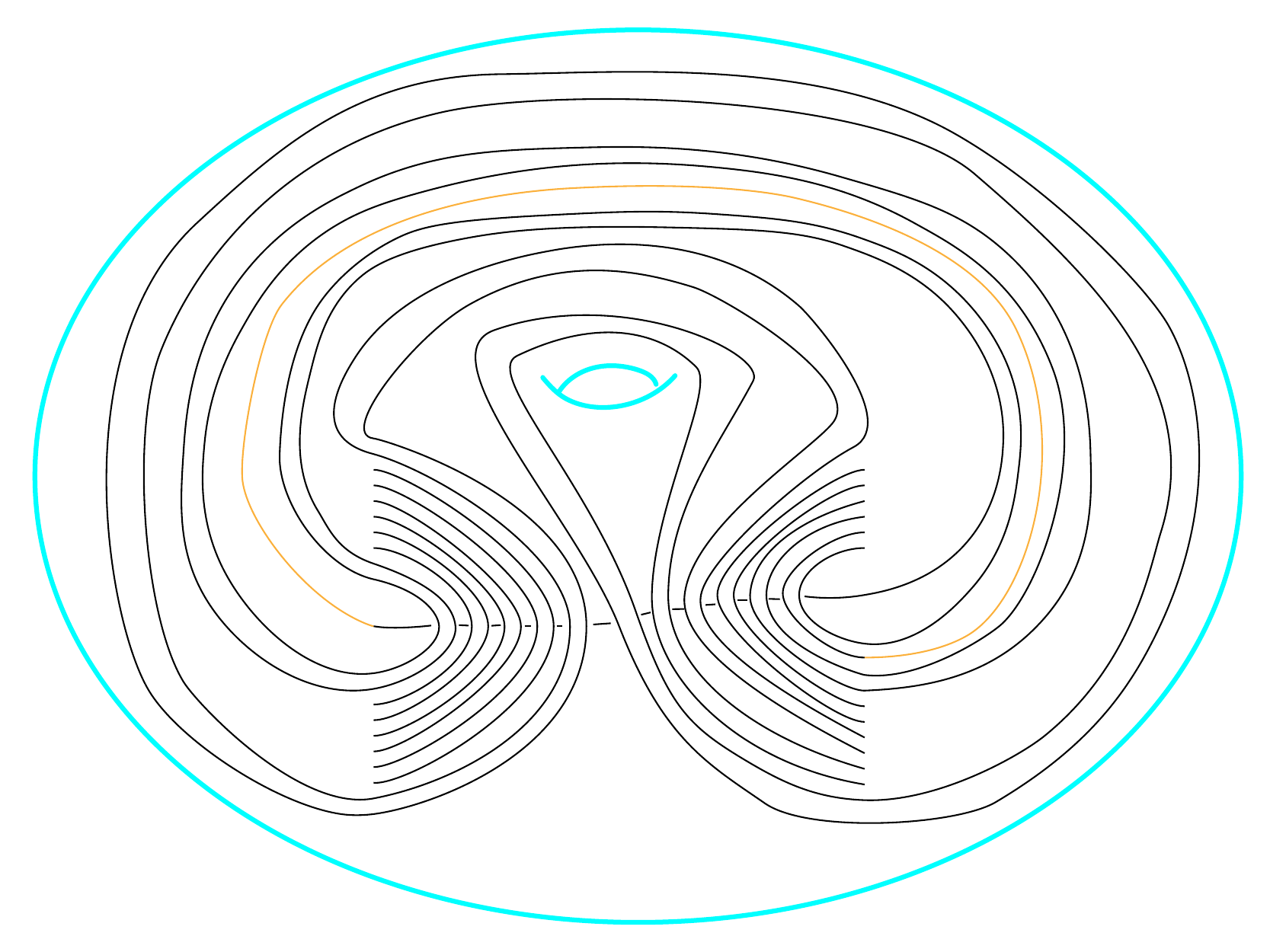}};
\node[font=\tiny, rotate=30] at (10,3.7) {$j+1$};
\draw [decorate,decoration={brace,amplitude=2pt,mirror,raise=4ex}]
  (9,1.85) -- (9,2.05) node[font=\tiny, midway]{};
  \draw [decorate,decoration={brace,amplitude=2pt,mirror,raise=4ex}]
  (9,2.5) -- (9,2.7) node[font=\tiny, midway]{};
  \draw [decorate,decoration={brace,amplitude=2pt,mirror,raise=4ex}]
  (9,2.15) -- (9,2.35) node[font=\tiny, midway, yshift=3em]{};
  
  \node[font=\tiny] at (10.2,1.95) {$j+1$};
  \node[font=\tiny] at (10.2,2.25) {$j+1$};
   \node[font=\tiny] at (10.2,2.6) {$j+1$};
\node[font=\tiny, rotate=-33] at (2.95,1.95) {$j+1$};
\node[font=\tiny, rotate=-33] at (2.95,3.3) {$j+1$};
\node[font=\tiny, rotate=30] at (9,5.4) {$j+1$};
\end{tikzpicture}
\caption{The knot in $S^1 \times D^2$ determined by the $(1,1)$ pattern with $\beta=\beta(i,j)$.}\label{genmazurknot}
\end{figure}

\begin{figure}[!tbp]
  \centering
  \begin{minipage}[b]{0.35\textwidth}
  \begin{tikzpicture}

\node[anchor=south west,inner sep=0] at (0,0)    {\includegraphics[width=1.4\textwidth]{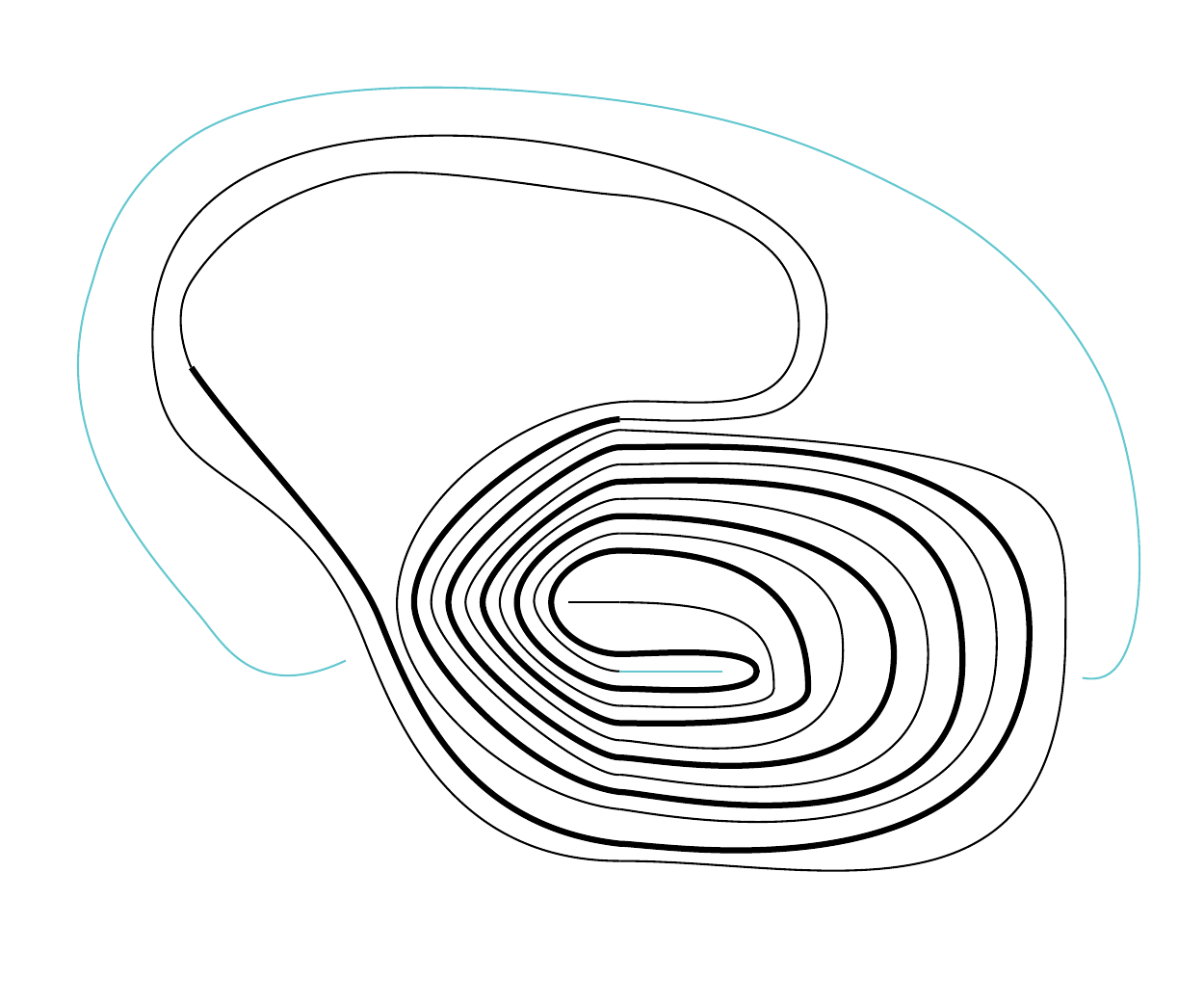}};

\end{tikzpicture}
    \caption{The knot from Figure \ref{genmazurknot} after an isotopy.}\label{isotopy.5}
  \end{minipage}
  \hspace{1in}
  \begin{minipage}[b]{0.35\textwidth}
  \begin{tikzpicture}
\node[anchor=south west,inner sep=0] at (0,0)    {\includegraphics[width=1.4\textwidth]{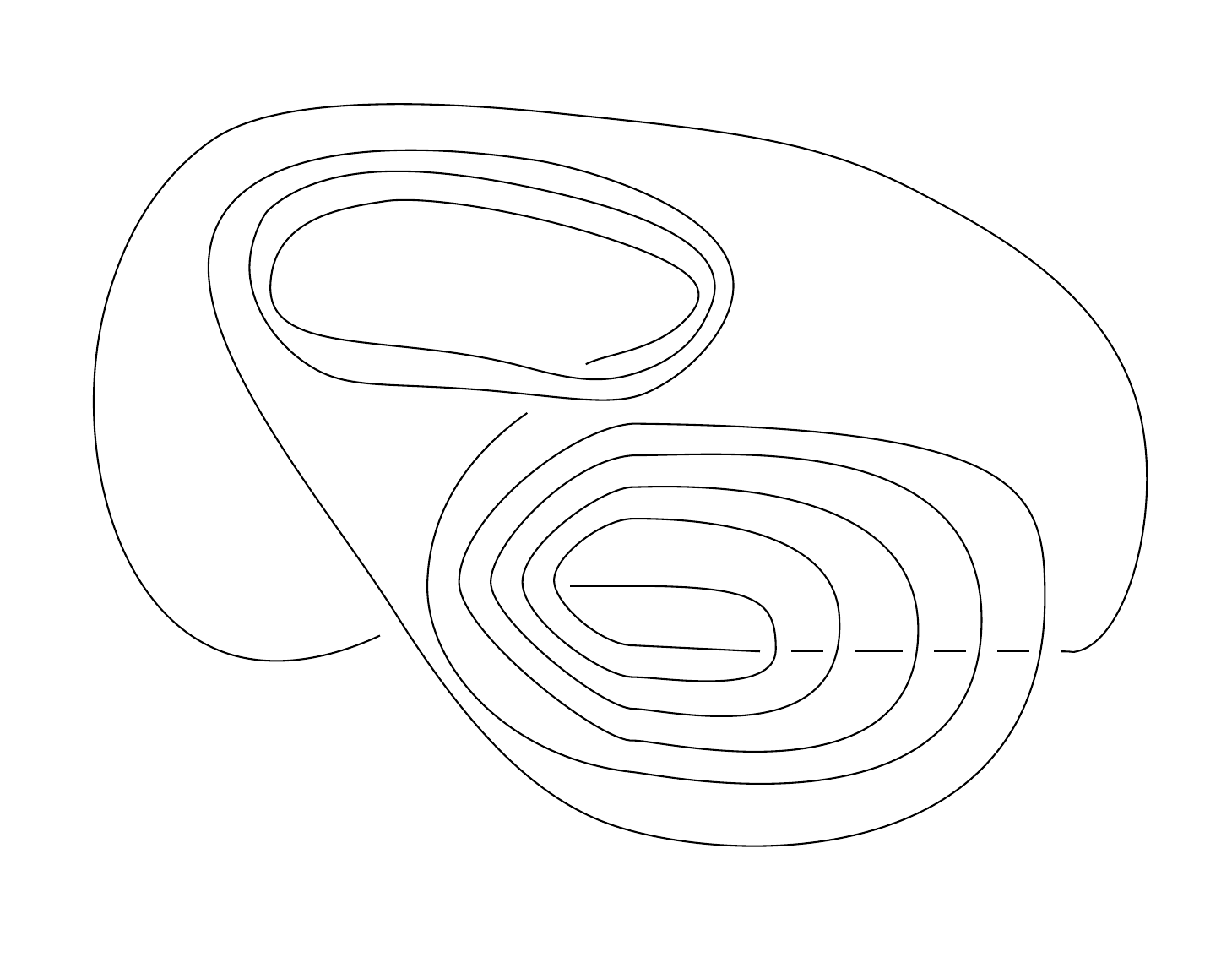}};
\end{tikzpicture}
    \caption{Isotope the $j$ consecutive strands that are bold in Figure \ref{isotopy.5} to obtain this knot, which is $Q^{i,j}_0$.}\label{isotopy1}
  \end{minipage}
\end{figure}

In this section we introduce the specific $(1,1)$-patterns that give rise to the pattern knots $Q^{i,j}$ shown in Figure \ref{Qij}. 

\begin{definition}
Let $\beta(i,j)$ denote the $\beta$ curve for the $(1,1)$ pattern which in the parameterization of \cite{Chen} is given by $(r,s)=(2+j+2(j+1)i,j)$.
\end{definition}

The doubly pointed bordered Heegaard diagram associated with $\beta(i,j)$ is shown in Figure \ref{Mazur(i,j)}, and from that description it is easy to see that knot determined by the $(1,1)$ pattern with $\beta$ curve $\beta(i,j)$ is shown in Figure \ref{genmazurknot}. In that figure there are $r-1=1+j+2i(1+j)=(2i+1)(1+j)$ rainbows and $s+1=j+1$ stripes. Each pair of strands represents $j+1$ parallel strands, as indicated, and there are $2i(j+1)$ of them. If we pull the $(2i+1)(1+j)$ rainbows from the left side of the figure around the orange arc, we end up with Figure \ref{isotopy.5}. In that figure the bold line represents $j$ consecutive strands. We isotope the $j$ strands by pulling i the bold piece of the knot, and end up at Figure \ref{isotopy1}. Here there are $j$ strands winding around the hole of the torus and $2i+1$ rainbows. It is straightforward to verify that this is the knot $Q^{i,j}_0$ shown in Figure \ref{Qij}.

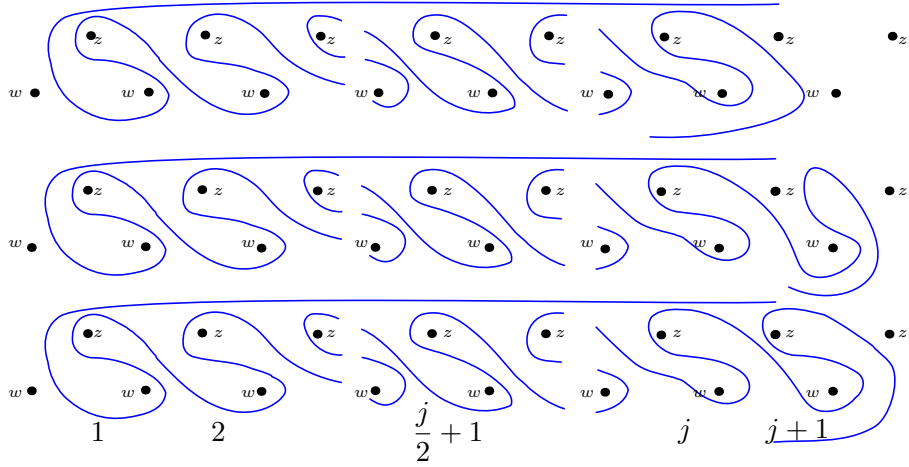
\begin{figure}[!tbp]
\begin{center}
  \begin{tikzpicture}
\node[anchor=south west,inner sep=0] at (0,0)    {\includegraphics[scale=.3]{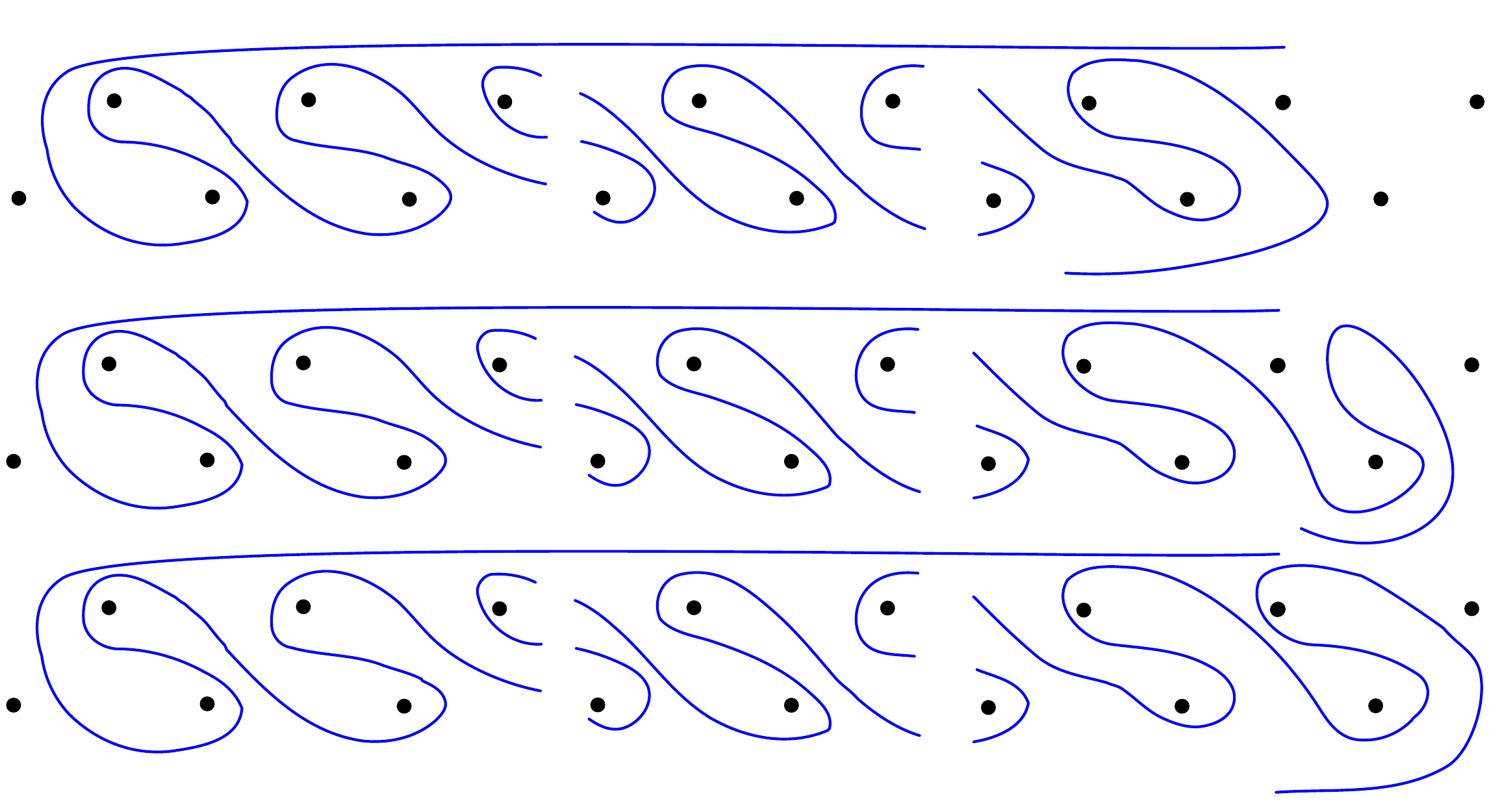}};
\node at (1,.3) {$1$};
\node at (2.6,.3) {$2$};
\node at (5.65,.3) {$\dfrac{j}{2}+1$};
\node at (8.8,.3) {$j$};
\node at (10.3,.3) {$j+1$};
\node[font=\tiny] at (-.1, 4.8) {$w$};
\node[font=\tiny] at (-.1, 2.8) {$w$};
\node[font=\tiny] at (-.1, .8) {$w$};
\node[font=\tiny] at (1.4, 4.8) {$w$};
\node[font=\tiny] at (1.4, 2.8) {$w$};
\node[font=\tiny] at (1.4, .8) {$w$};
\node[font=\tiny] at (3, 4.8) {$w$};
\node[font=\tiny] at (3, 2.8) {$w$};
\node[font=\tiny] at (3, .8) {$w$};
\node[font=\tiny] at (4.5, 4.8) {$w$};
\node[font=\tiny] at (4.5, 2.8) {$w$};
\node[font=\tiny] at (4.5, .8) {$w$};
\node[font=\tiny] at (6, 4.8) {$w$};
\node[font=\tiny] at (6, 2.8) {$w$};
\node[font=\tiny] at (6, .8) {$w$};
\node[font=\tiny] at (7.5, 4.8) {$w$};
\node[font=\tiny] at (7.5, 2.8) {$w$};
\node[font=\tiny] at (7.5,.8) {$w$};
\node[font=\tiny] at (9, 4.8) {$w$};
\node[font=\tiny] at (9, 2.8) {$w$};
\node[font=\tiny] at (9, .8) {$w$};
\node[font=\tiny] at (10.5, 4.8) {$w$};
\node[font=\tiny] at (10.5, 2.8) {$w$};
\node[font=\tiny] at (10.5, .8) {$w$};

\node[font=\tiny] at (1, 1.6) {$z$};
\node[font=\tiny] at (1, 3.5) {$z$};
\node[font=\tiny] at (1, 5.5) {$z$};
\node[font=\tiny] at (2.6, 1.6) {$z$};
\node[font=\tiny] at (2.6, 3.5) {$z$};
\node[font=\tiny] at (2.6, 5.5) {$z$};
\node[font=\tiny] at (4.1, 1.6) {$z$};
\node[font=\tiny] at (4.1, 3.5) {$z$};
\node[font=\tiny] at (4.1, 5.5) {$z$};
\node[font=\tiny] at (5.65, 1.6) {$z$};
\node[font=\tiny] at (5.65, 3.5) {$z$};
\node[font=\tiny] at (5.65, 5.5) {$z$};
\node[font=\tiny] at (7.15, 1.6) {$z$};
\node[font=\tiny] at (7.15, 3.5) {$z$};
\node[font=\tiny] at (7.15, 5.5) {$z$};
\node[font=\tiny] at (8.7, 1.6) {$z$};
\node[font=\tiny] at (8.7, 3.5) {$z$};
\node[font=\tiny] at (8.7, 5.5) {$z$};
\node[font=\tiny] at (10.2, 1.6) {$z$};
\node[font=\tiny] at (10.2, 3.5) {$z$};
\node[font=\tiny] at (10.2, 5.5) {$z$};
\node[font=\tiny] at (11.7, 1.6) {$z$};
\node[font=\tiny] at (11.7, 3.5) {$z$};
\node[font=\tiny] at (11.7, 5.5) {$z$};
\end{tikzpicture}
    \caption{The isotopy that produces $\beta(0,j+1)$ from $\beta(0,j)$.}\label{fingermove}
    \end{center}
\end{figure}

In order to understand the pairing diagram for $\CFKhat(\talpha(K,n),\tbeta(i,j))$ we make some observations about the lifted $\beta$ curve $\tbeta(i,j)$. When $i=0$ the curve $\beta(0,j)$ is determined by the pair $(r,s)=(2+j,j)$. In this case, it is easy to see that the lift $\tbeta(0,j)$ has the form shown in Figure \ref{fingermove} top row. Indeed, each ``wave'' contributes one to the count of rainbows, and there are $j+1$ ``waves'', and there is one extra rainbow at the left end. Said another way, the lifted $\beta$ curve $\tbeta(0,j)$ is obtained from $\tbeta(0,j-1)$ by the finger move shown in Figure \ref{fingermove} and this isotopy introduces one more rainbow and one more stripe to $\tbeta(0,j-1)$.

\begin{figure}[!tbp]
  \centering
  \begin{minipage}[b]{0.3\textwidth}
  \begin{tikzpicture}
  \node[anchor=south west,inner sep=0] at (0,0)    {\includegraphics[width=1.2\textwidth]{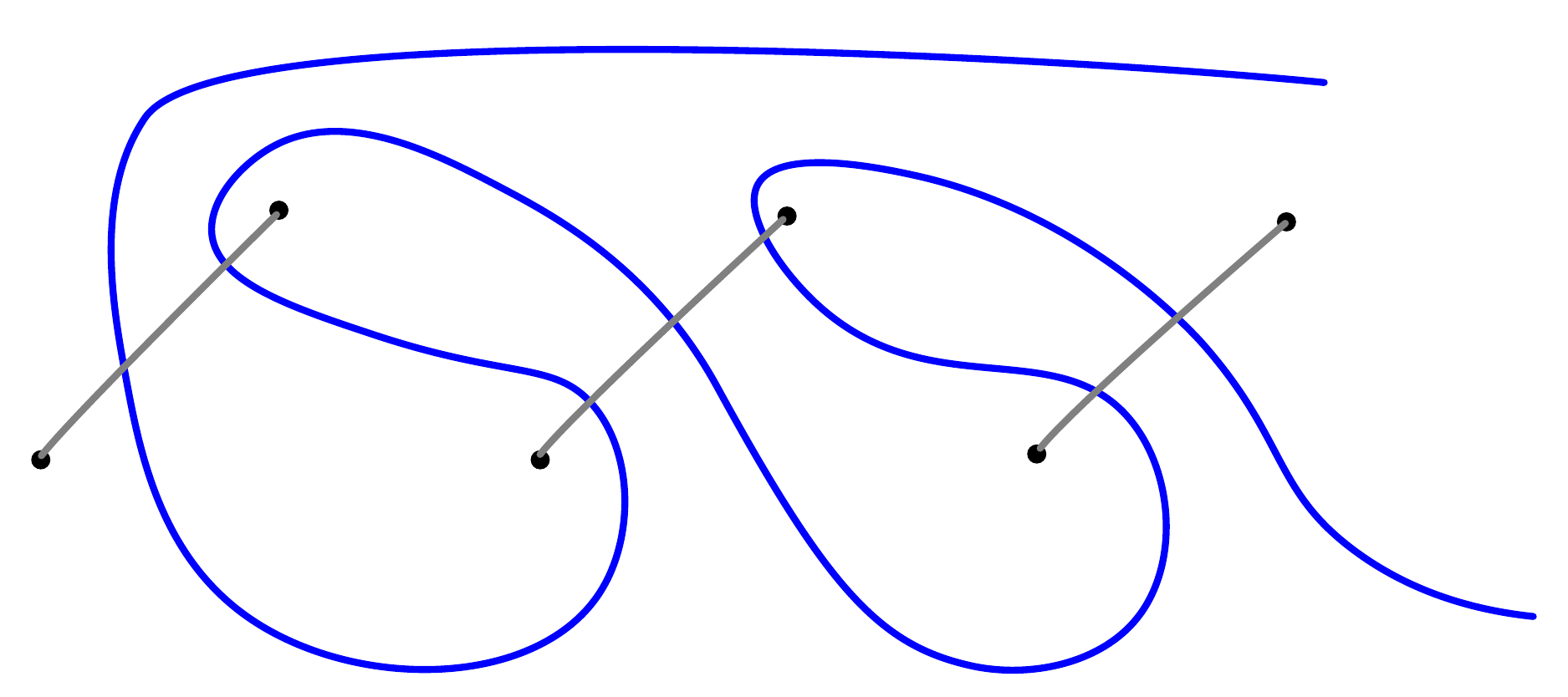}};
  \node[font=\tiny] at (-.1, .8) {$w$};
  \node[font=\tiny] at (1.6, .8) {$w$};
  \node[font=\tiny] at (3.1, .8) {$w$};
   \node[font=\tiny] at (1,1.5) {$z$};
  \node[font=\tiny] at (2.65,1.5) {$z$};
  \node[font=\tiny] at (4.25, 1.5) {$z$};
\node at (5,2) {$\cdots$};
\node at (5.3,.3) {$\cdots$};
\node[font=\tiny] at (1,.3) {$A$};
\node[font=\tiny] at (2,2.2) {$A-1$};
\node[font=\tiny, rotate=-20] at (1.2,1) {$A-1$};
\node[font=\tiny, rotate=-20] at (1.6,1.75) {$A-2$};

\node[font=\tiny, rotate=-50] at (2.5,.8) {$A-1$};
\node[font=\tiny, rotate=-20] at (3,1.2) {$A-2$};
\node[font=\tiny, rotate=-15] at (3,1.75) {$A-3$};
\node[font=\tiny, rotate=-30] at (4.5,.5) {$A-2$};
\end{tikzpicture}
    \caption{The curve $\tbeta(0,j)$ for the knot $Q^{0,j}$.}\label{pretwistbeta}
  \end{minipage}
  \hspace{1in}
  \begin{minipage}[b]{0.3\textwidth}
  \begin{tikzpicture}

\node[anchor=south west,inner sep=0] at (0,0)    {\includegraphics[width=1.4\textwidth]{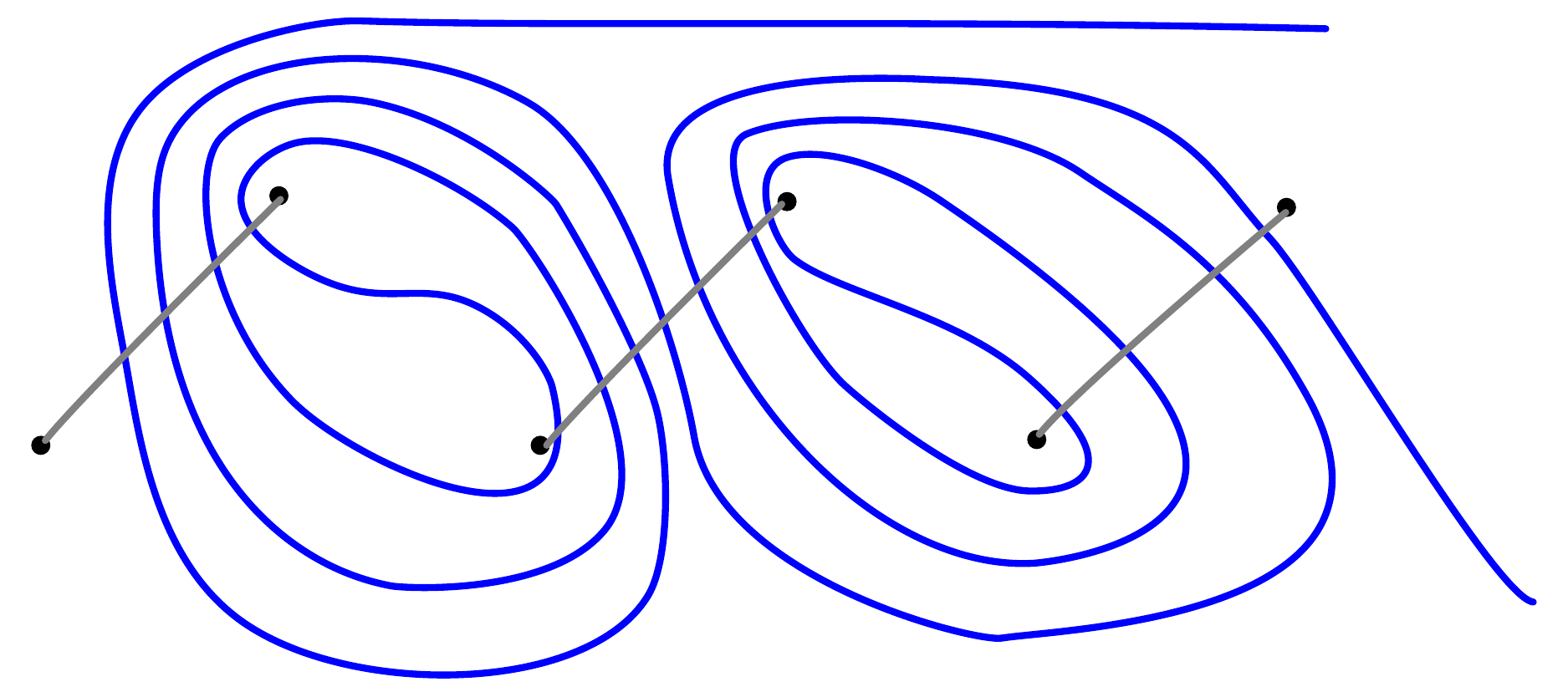}};
\node[font=\tiny] at (.1,1.2) {$\delta_{w,z}$};
\node at (5.3,2.5) {$\cdots$};
\node at (6.1,.4) {$\cdots$};
\node[font=\tiny] at (1,.3) {$A$};
\node[font=\tiny] at (2,2.7) {$A-1$};
\node[font=\tiny] at (1.4,1) {$A$};
\node[font=\tiny] at (1.4,.57) {$A-1$};
\node[font=\tiny] at (1.4,1.58) {$A-1$};
  \node[font=\tiny] at (-.1, .9) {$w$};
  \node[font=\tiny] at (1.75, .9) {$w$};
  \node[font=\tiny] at (3.8, .9) {$w$};
   \node[font=\tiny] at (1.15,1.9) {$z$};
  \node[font=\tiny] at (3.2,1.9) {$z$};
  \node[font=\tiny] at (4.9, 1.9) {$z$};
\end{tikzpicture}
    \caption{twist up the curve $\tbeta(0,j)$ to get the curve $\tbeta(1,j)$ for the knot $Q^{1,j}$.}\label{posttwistbeta}
  \end{minipage}
\end{figure}

\begin{figure}[!tbp]
  \centering
  \begin{minipage}[b]{0.3\textwidth}
  \begin{tikzpicture}
\node[anchor=south west,inner sep=0] at (0,0)    {\includegraphics[width=1.2\textwidth]{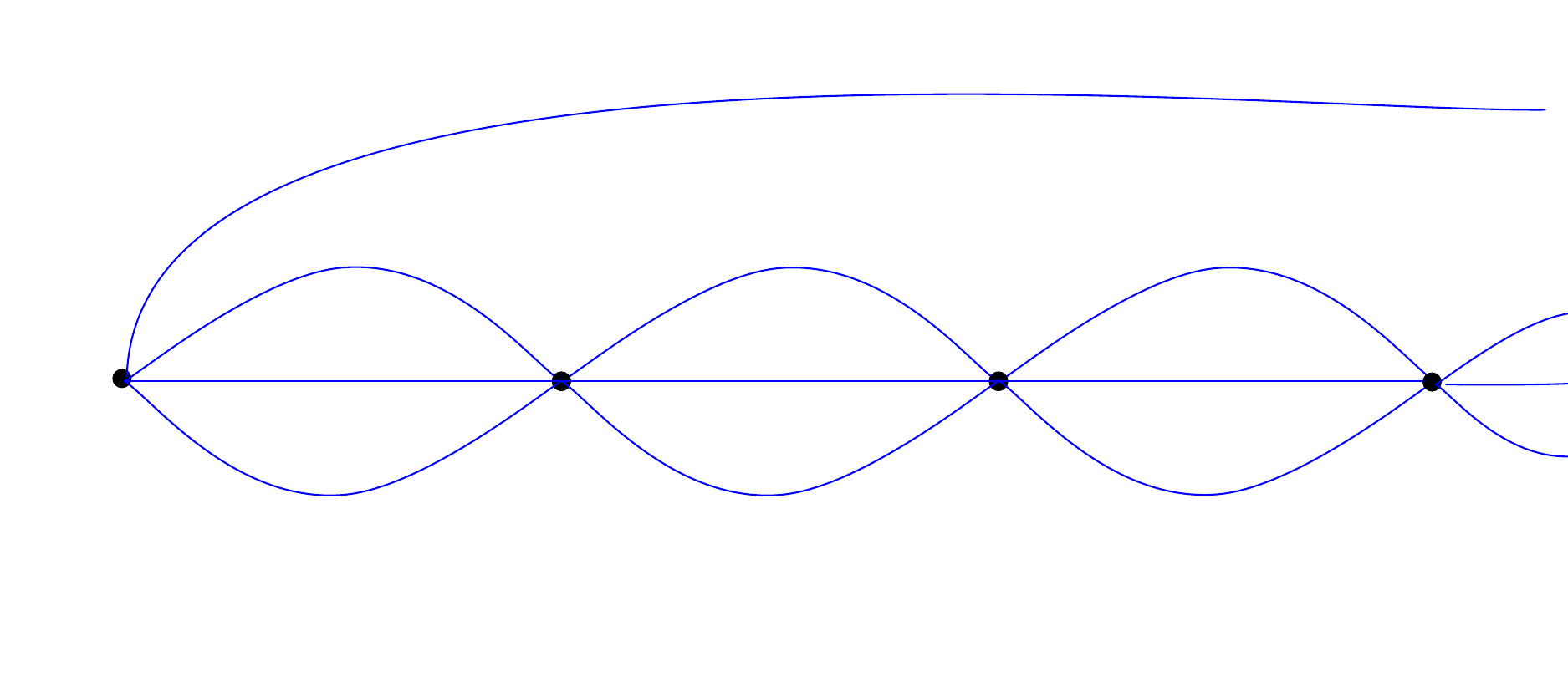}};
\node at (5,2) {$\cdots$};
\node at (5.3,.3) {$\cdots$};
\node[font=\tiny] at (1,.6) {$A$};
\node[font=\tiny] at (2,2.2) {$A-1$};
\node[font=\tiny] at (1.2,1.1) {$A-1$};
\node[font=\tiny] at (1.2,1.45) {$A-2$};
\node[font=\tiny] at (2.5,.6) {$A-1$};
\node[font=\tiny] at (4,.6) {$A-2$};
\node[font=\tiny] at (2.6,1.1) {$A-2$};
\node[font=\tiny] at (2.6,1.46) {$A-3$};
\end{tikzpicture}
    \caption{The collapsed $\tbeta(0,j)$ curve for the knot $Q^{0,j}$.}\label{collapsedbeta}
  \end{minipage}\hspace{1in}
  \begin{minipage}[b]{0.3\textwidth}
  \begin{tikzpicture}
\node[anchor=south west,inner sep=0] at (0,0)    {\includegraphics[width=1.4\textwidth]{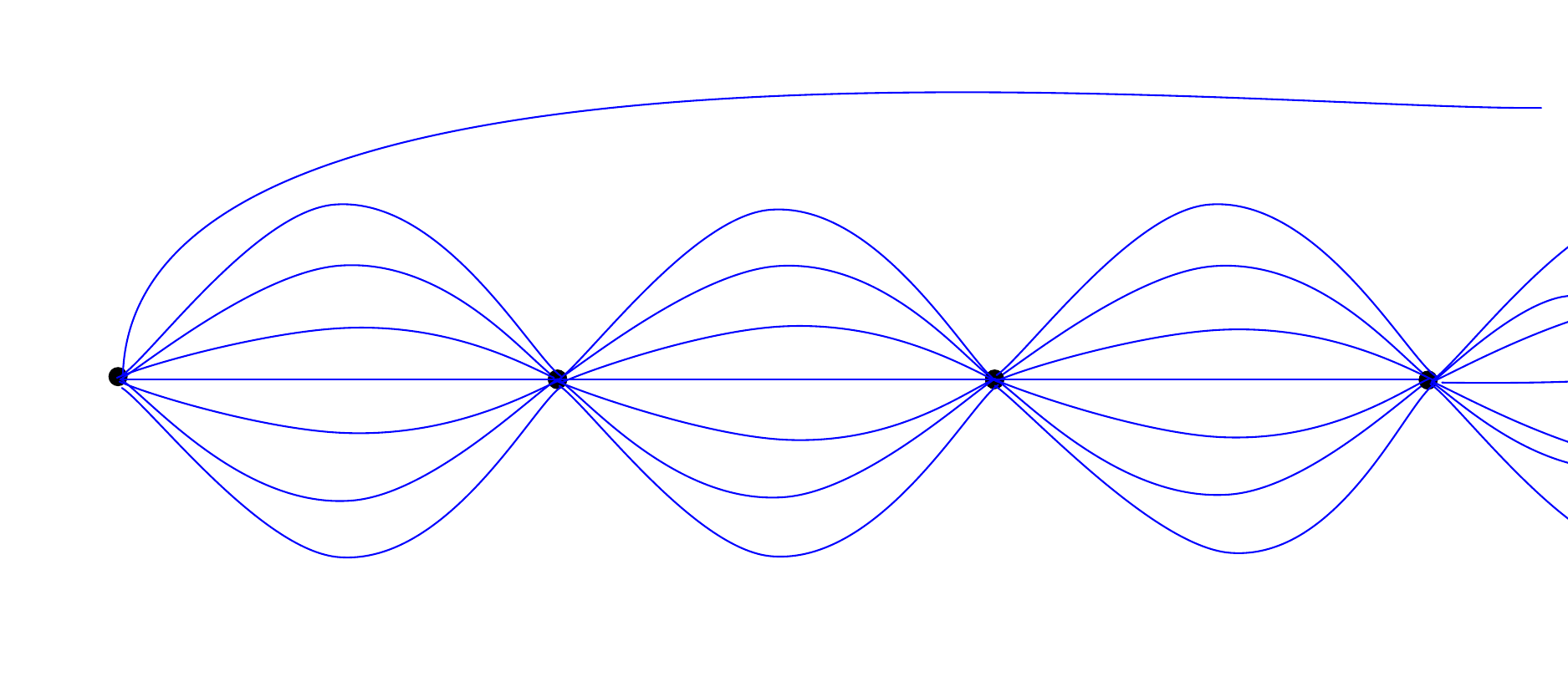}};
\node at (5.3,2.5) {$\cdots$};
\node at (6.1,.4) {$\cdots$};
\node[font=\tiny] at (1.4,1) {$A$};
\node[font=\tiny] at (1.3,.75) {$A-1$};
\node[font=\tiny] at (1.45,1.2) {$A-1$};
\node[font=\tiny] at (1.4, 1.4) {$A-2$};
\node[font=\tiny] at (1.35,1.6) {$A-1$};
\node[font=\tiny] at (1.3,1.8) {$A-2$};
\node[font=\tiny] at (2,2.5) {$A-1$};
\node[font=\tiny] at (1.35,.55) {$A$};
\end{tikzpicture}
    \caption{The collapsed $\tbeta(1,j)$ curve for the knot $Q^{1,j}$.}\label{collapsedtwistedbeta}
  \end{minipage}
\end{figure}

Next we claim that the transition from $\tbeta(0,j)$ to $\tbeta(1,j)$ corresponds to ``twisting up'' each wave, which is shown in Figure \ref{posttwistbeta}. Indeed, here we see that twisting up adds an extra $2$ rainbows for each wave region, and thus $2(j+1)$ new rainbows in total. In general, $\tbeta(i,j)$ is obtained from $\tbeta(0,j)$ by twisting up each wave region $i$ times, and we see that this corresponds to adding $2(j+1)i$ new rainbows, and no new stripes, to the lifted $\beta$ curve.

For convenience we label the arcs of the $\beta$ curves between lifts of the $\delta_{w,z}$ arcs by relative Alexander gradings that an intersection between $\alpha(K,n)$ and $\beta(i,j)$ on that arc would carry if there were intersections on that arc. It is straightforward to see that these Alexander grading labels increase as we move from right to left and bottom to top along the lift $\tbeta(i,j)$. Moreover, from the description of twisting up and \cite[Lemma 4.1]{Chen} the following lemma is immediate (see Figures \ref{pretwistbeta}-\ref{posttwistbeta}).

\begin{lemma}\label{twistupnochange}
For any knot $K$ and for any $i>0$, we have
$$\{A: \HFKhat(S^3,Q^{i,j}_n(K), A) \neq 0\}=\{A:\HFKhat(S^3,Q^{0,j}_n(K), A) \neq 0\}.$$
\end{lemma}

In order to simplify arguments and pictures in the next section, we introduce a modified version of the lifted $\beta$ curve, called the collapsed $\beta$ curve. 
\begin{definition}
    Let $B(i,j)$ denote the curve $\tbeta(i,j)$ after collapsing the lifts of the arcs $\delta_{w,z}$ to a single point.
\end{definition}

See Figure \ref{collapsedbeta} and \ref{collapsedtwistedbeta} where we draw $B(0,j)$ and $B(1,j)$ together with the Alexander gradings of arcs. The following lemma is immediate. 
\begin{lemma}
  As an $\F$-vector space, $\CFKhat(\talpha(K,n),B(i,j))\cong \CFKhat(\talpha(K,n),\tbeta(i,j))$ and moreover, we can recover the Alexander grading of any intersection point in the collapsed pairing diagram. 
\end{lemma}

Although twisting up does not change the set of Alexander gradings labelling arcs of the $\beta$ curves by Lemma \ref{twistupnochange}, twisting up does change the number of arcs of the collapsed $\beta$ curve that are labelled with a fixed Alexander grading. We will return to this observation in section \ref{fiberedsection} (see Lemma \ref{rankintop}).

\section{Three-Dimensional Invariants}

In this section we compute the genus of the patterns $Q^{i,j}_n$, determine the set of triples $(i,j,n)$ so that the pattern $Q^{i,j}_n$ is fibered in the solid torus, and show that whenever $K$ is a non-trivial companion the satellite knots $Q^{i,j}_n(K)$ are not Floer thin.

\subsection{Three-Genus and $n$-twisted Satellites}\label{Three-genus}

In this section we use Theorem \ref{Pairing} and the collapsed pairing diagram for $n$-framed satellite knots to prove Theorems \ref{genusnontrivial} and \ref{genusunknotsatellite} from the introduction. Recall that our strategy is to determine $g(Q^{i,j}_n(T_{2,3}))$ directly from the pairing diagram and deduce the forumla for a general non-trivial companion from Equation \ref{schuberteq}. An immediate Corollary of Theorem \ref{genusnontrivial} is a computation of the genus of the $n$-twisted pattern knot $Q^{i,j}_n$ in $S^1 \times D^2$.

\begin{corollary}
    For $n \in \Z$, $i \in \Z_{\geq 0}$ and $j \in \Z_{>0}$, the pattern knot $Q^{i,j}_n$ in $S^1\times D^2$ has genus $$g(Q^{i,j}_n)=\dfrac{j(j+1)}{2}|n|+\begin{cases} 1 & n \geq 0\\ 1-j & n <0\,.\end{cases}$$
\end{corollary}
\begin{proof}
    Equation \ref{schuberteq} shows that $$g(Q^{i,j}_n)=g(Q^{i,j}_n(T_{2,3}))-jg(T_{2,3})=g(Q^{i,j}_n(T_{2,3}))-j.$$\qedhere
\end{proof}

To prove Theorems \ref{genusnontrivial} and \ref{genusunknotsatellite}, we will make use of the collapsed pairing diagram. Note first that since $g(K)=\max\{A: \HFKhat(S^3,K,A)\neq0\}$, Lemma \ref{twistupnochange} implies that $g(Q^{0,j}_n(K))=g(Q^{i,j}_n(K))$, so it is enough to consider the case $i=0$. 

In Figures \ref{genuspositivejodd}-\ref{genuspositiveeven}, we see half of the lifted pairing diagram $\CFKhat(\talpha(T_{2,3},n),\tbeta(0,j))$. The other half is determined by the symmetry of the pairing diagram coming from the symmetry of knot Floer homology. We work with the collapsed pairing diagram to simplify the pictures, since we are not interested in any of the differentials and only in the Alexander gradings in this section. Note that by \cite[Lemma 6.3]{Chen}, the Alexander grading of intersection points of $\alpha(T_{2,3})$ and $\beta(0,j)$ increase by $-w(Q^{0,j}_n)=j$ as we go up one row in the pairing diagram, so to determine the largest Alexander grading of an intersection point in the pairing diagram, it is enough to determine the number of rows between the central intersection point $c$ (with $A(c)=0$) and the top of the pairing diagram.

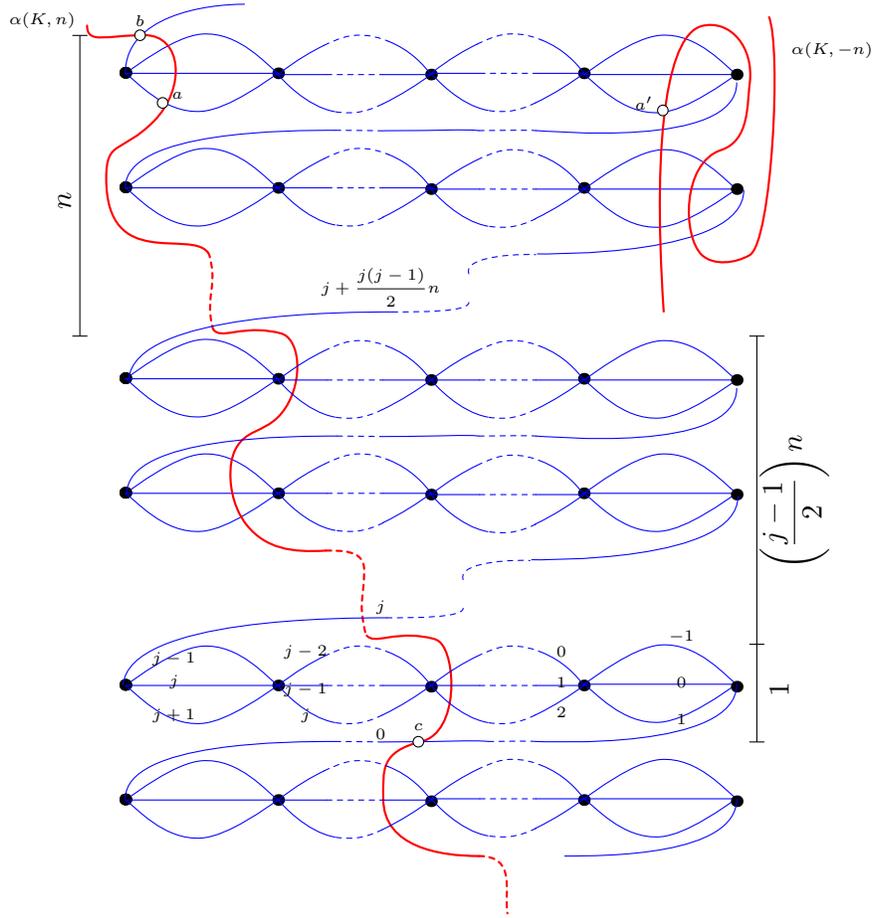
\begin{figure}[!tbp]
\begin{center}
  \begin{tikzpicture}

\node[anchor=south west,inner sep=0] at (0,0)    {\includegraphics[scale=.4]{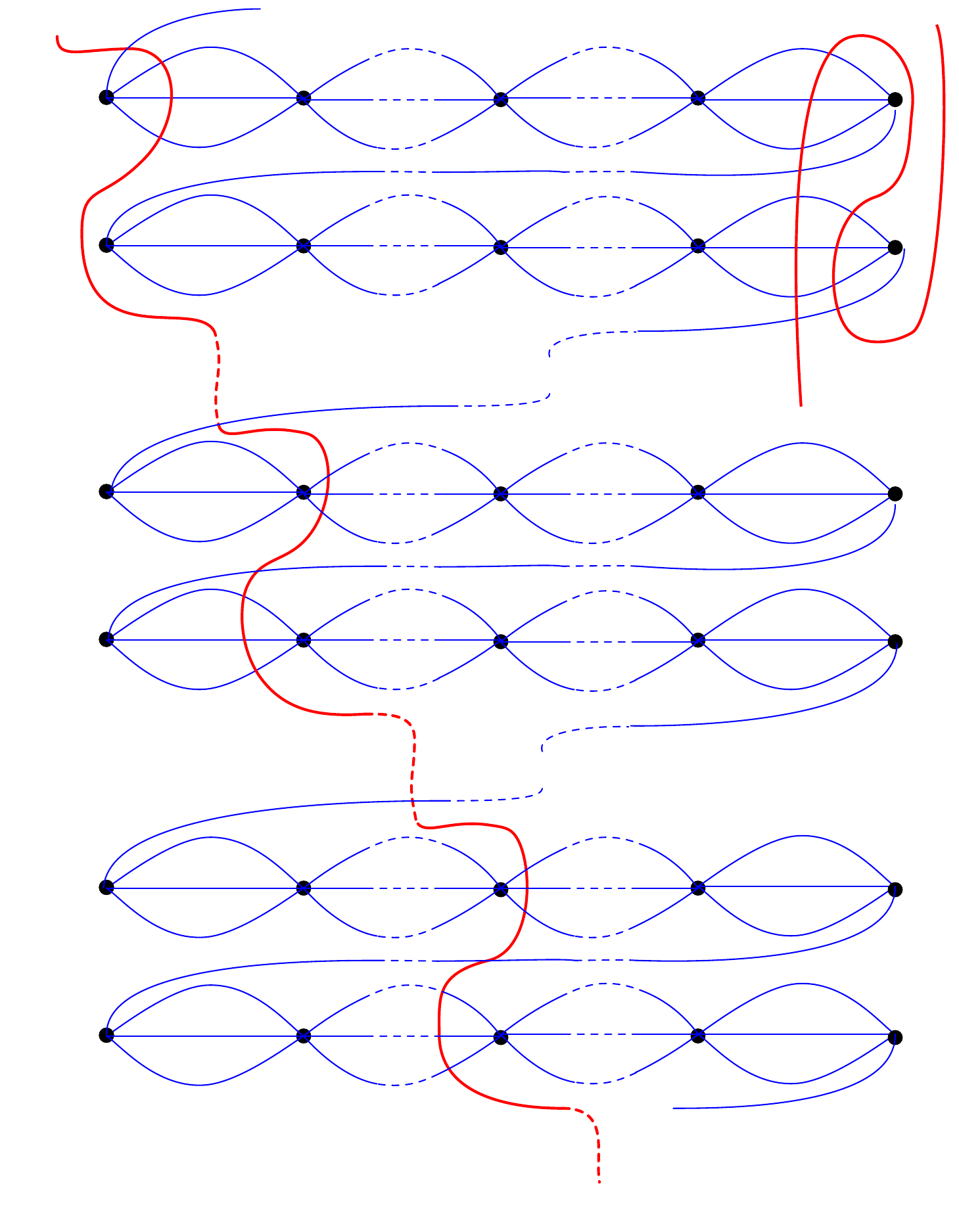}};
\node at (5,2.8) {$\tikzcirc{2pt}$};
\node at (1.6,11.3) {$\tikzcirc{2pt}$};
\node at (1.3,12.2) {$\tikzcirc{2pt}$};
\node at (8.25,11.2) {$\tikzcirc{2pt}$};
\node[font=\tiny] at (8,11.3) {$a'$};
\node[font=\tiny] at (10.5,12) {$\alpha(K,-n)$};
\node[font=\tiny] at (5,3) {$c$};
\node[font=\tiny] at (1.8,11.4) {$a$};
\node[font=\tiny] at (1.3,12.4) {$b$};
\node[font=\tiny] at (0,12.4) {$\alpha(K,n)$};
\node[font=\tiny] at (4.5,2.9) {$0$};
\node[font=\tiny] at (4.5,4.6) {$j$};
\node[font=\tiny] at (8.5,4.2) {$-1$};
\node[font=\tiny] at (8.5,3.6) {$0$};
\node[font=\tiny] at (8.5,3.1) {$1$};

\node[font=\tiny] at (6.9,4) {$0$};
\node[font=\tiny] at (6.9,3.6) {$1$};
\node[font=\tiny] at (6.9,3.2) {$2$};

\node[font=\tiny] at (1.75,3.9) {$j-1$};
\node[font=\tiny] at (1.75,3.6) {$j$};
\node[font=\tiny] at (1.75,3.15) {$j+1$};

\node[font=\tiny] at (3.5,4) {$j-2$};
\node[font=\tiny] at (3.5,3.5) {$j-1$};
q\node[font=\tiny] at (3.5,3.15) {$j$};

\node[font=\tiny] at (4.5,8.85) {$j+\dfrac{j(j-1)}{2}n$};

\draw (9.5,2.8) to (9.5,4.1);
\draw (9.4,2.8) to (9.6,2.8);
\draw (9.4,4.1) to (9.6,4.1);
\draw (9.5,4.1) to (9.5,8.2);
\draw (9.4,8.2) to (9.6,8.2);
\draw (.5,8.2) to (.5,12.2);
\draw (.4,8.2) to (.6,8.2);
\draw (.4,12.2) to (.6,12.2);
\node[rotate=90] at (.3,10) {$n$};
\node[rotate=90] at (10,6) {$\left(\dfrac{j-1}{2}\right)n$};
\node[rotate=90] at (9.8,3.5) {$1$};
\end{tikzpicture}
    \caption{The pairing diagram for $Q^{0,j}_n$ when $j$ is odd and $n>0$.}\label{genuspositivejodd}
    \end{center}
\end{figure}

\begin{proof}[Proof of Theorem \ref{genusnontrivial}]

As mentioned, by Lemma \ref{twistupnochange}, it is enough to determine the genus in the case $i=0$, and by Equation \ref{schuberteq} it is enough to compute $g(Q^{0,j}_n(T_{2,3}))$. To this end, consider first the case $n\geq 0$. It is easy to see that the intersection point labelled $a$ in Figures \ref{genuspositivejodd}-\ref{genuspositiveeven} has the largest Alexander grading. Indeed, the Alexander gradings increase by $j$ for each row we go up in the pairing diagram and they also increase on each strand by one for each rwo we go over. To determine $A(a)$, note that there are a total of $2(j+2)+(n-2)(j+1)$ lifts of the curve $\beta(0,j)$ needed to account for all the intersections between $\alpha(T_{2,3},n)$ and $\beta(0,j)$. Indeed there are $(j+2)$ lifts of the $\CFK^{\infty}(T_{2,3})$ region (which occupies $2$ rows) and there are $(j+1)$ lifts of the unstable region, which spans $n-2$ rows. There are then three cases to distinguish. If $j$ is odd, then there are an even number of rows and the central intersection point occurs between these rows. Moreover, since there are an odd number of $\CFK^{\infty}(T_{2,3})$ regions, by symmetry of the pairing diagram the central intersection point occurs in the middle of the central $\CFK^{\infty}$ region of the curve. See Figure \ref{genuspositivejodd}. If $j$ is even, then there are an odd number of lifts of the unstable region and so the central intersection point occurs somewhere along the unstable region. If $n$ is even or odd, then the number of rows is either even or odd. If $n$ is even, we are in the situation pictured in Figure \ref{genuspositiveeven} and if $n$ is odd, we are in the situation pictured in Figure \ref{genuspositiveodd}. In any case, to determine $A(a)$ it is enough to count the number of rows between the central intersection point $c$ and the intersection point labelled $b$. In all the cases the number of rows are indicated in the figure. We describe the case $j$ even and $n$ odd in detail and leave the rest to the reader. 
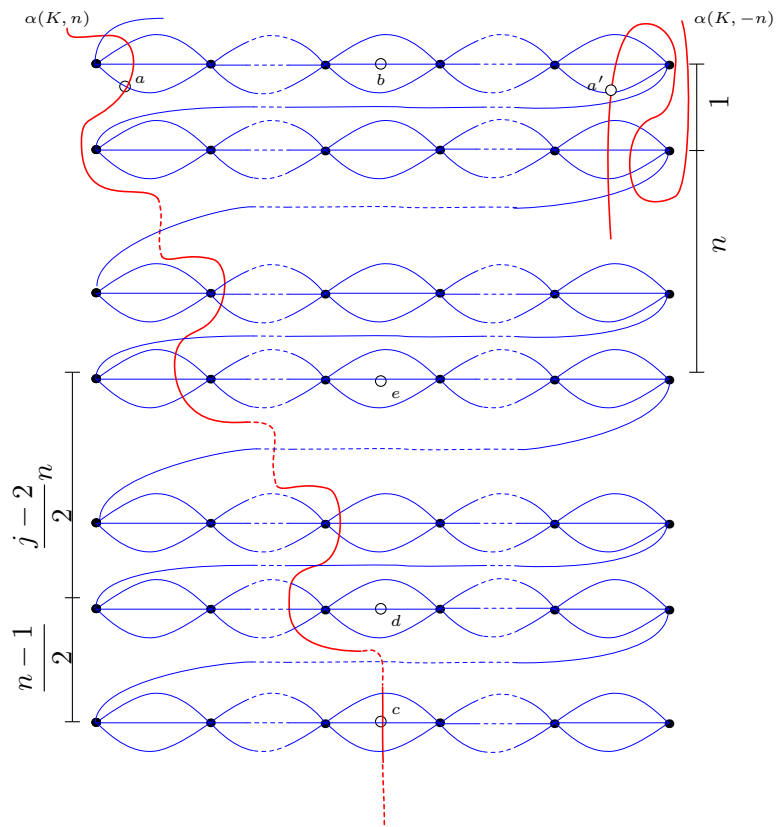
\begin{figure}[!tbp]
\begin{center}
  \begin{tikzpicture}
  \node at (4.3,2.15) {$\tikzcirc{2pt}$};
  \node at (.9,10.6) {$\tikzcirc{2pt}$};
  \node at (4.3,10.9) {$\tikzcirc{2pt}$};
  \node at (4.3,3.65) {$\tikzcirc{2pt}$};
  \node at (4.3,6.68) {$\tikzcirc{2pt}$};
  \draw (.2,2.15) to (.2,3.8);
  \draw (.1,2.15) to (.3,2.15);
  \draw(.1,3.8) to (.3,3.8);
  \draw(.1,6.8) to (.3,6.8);
  \draw (.2,3.8) to (.2,6.8);
  \draw (8.5,6.8) to (8.5,9.75);
  \draw (8.5,9.75) to (8.5,10.9);
  \draw (8.4,6.8) to (8.6,6.8);
  \draw (8.4,9.75) to (8.6,9.75);
  \draw (8.4,10.9) to (8.6,10.9);
\node[anchor=south west,inner sep=0] at (0,0)    {\includegraphics[scale=.3]{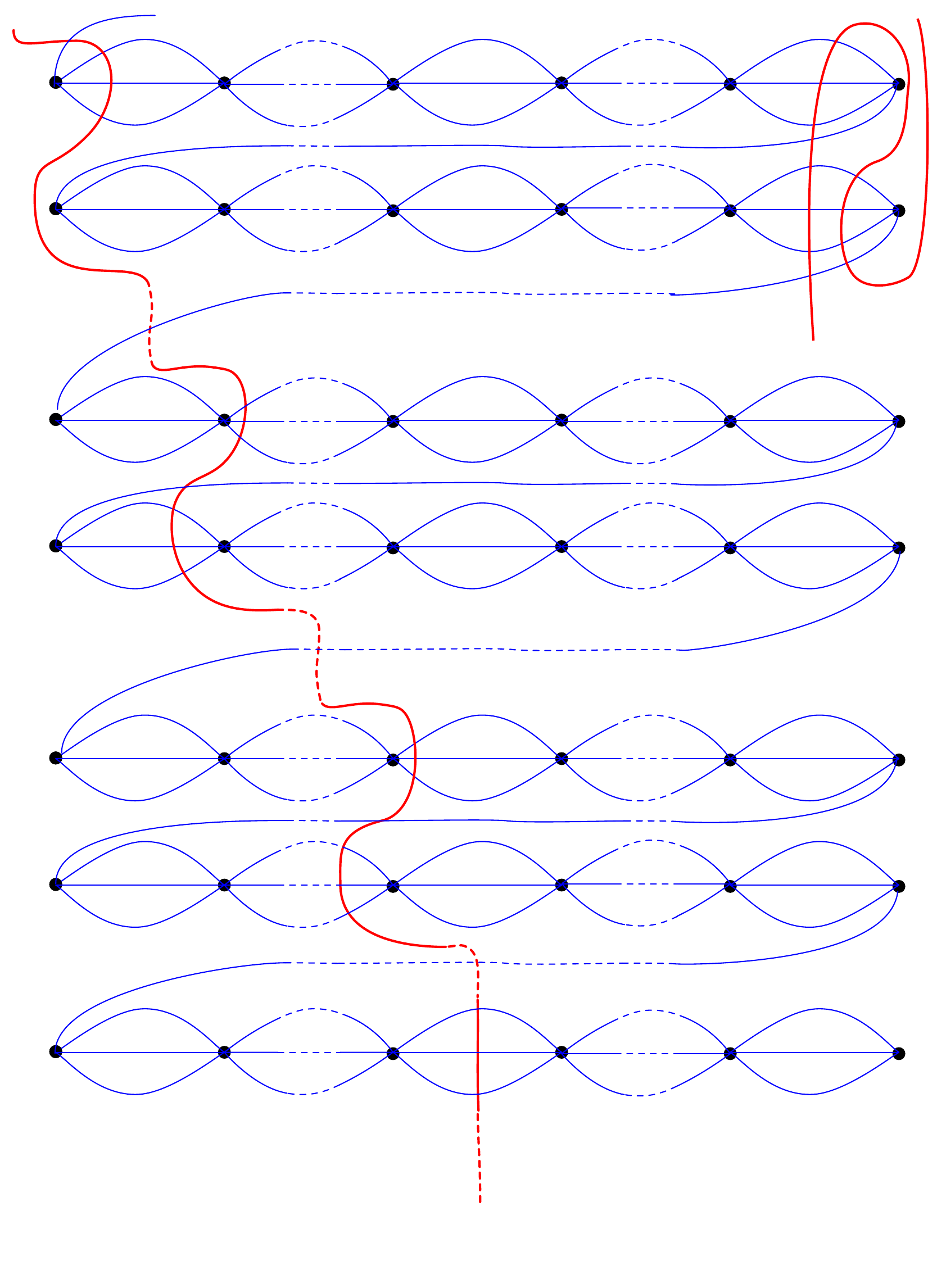}};
\node[font=\tiny] at (0,11.5) {$\alpha(K,n)$};
\node[font=\tiny] at (9,11.5) {$\alpha(K,-n)$};
\node[rotate=90] at (-.2,3) {$\dfrac{n-1}{2}$};
\node[rotate=90] at (-.2,5) {$\dfrac{j-2}{2}n$};
\node[rotate=90] at (8.8,8.5) {$n$};
\node[rotate=90] at (8.8,10.4) {$1$};
\node[font=\tiny] at (4.5, 2.3) {$c$};
\node[font=\tiny] at (1.1, 10.7) {$a$};
\node[font=\tiny] at (4.3,10.7) {$b$};
\node[font=\tiny] at (4.5,3.5) {$d$};
\node[font=\tiny] at (4.5,6.5) {$e$};
\node at (7.36,10.55) {$\tikzcirc{2pt}$};
\node[font=\tiny] at (7.16,10.65) {$a'$};
\end{tikzpicture}
    \caption{The general pairing diagram for $j$ even and $n$ odd.}\label{genuspositiveodd}
    \end{center}
\end{figure}

\begin{figure}[!tbp]
\begin{center}
  \begin{tikzpicture}
\node[font=\tiny] at (0,12.2) {$\alpha(K,n)$};
\node[font=\tiny] at (9.7,12.2) {$\alpha(K,-n)$};
\node[font=\tiny] at (.85,11.2) {$a$};
\node[font=\tiny] at (7.95,11.22) {$\tikzcirc{2pt}$};
\node[font=\tiny] at (7.8,11.4) {$a'$};
  \node at (4.7,1.7) {$\tikzcirc{2pt}$};
  \node at (1.05,11.3) {$\tikzcirc{2pt}$};
  \node at (.76,12.05) {$\tikzcirc{2pt}$};
\node[anchor=south west,inner sep=0] at (0,0)    {\includegraphics[scale=.32]{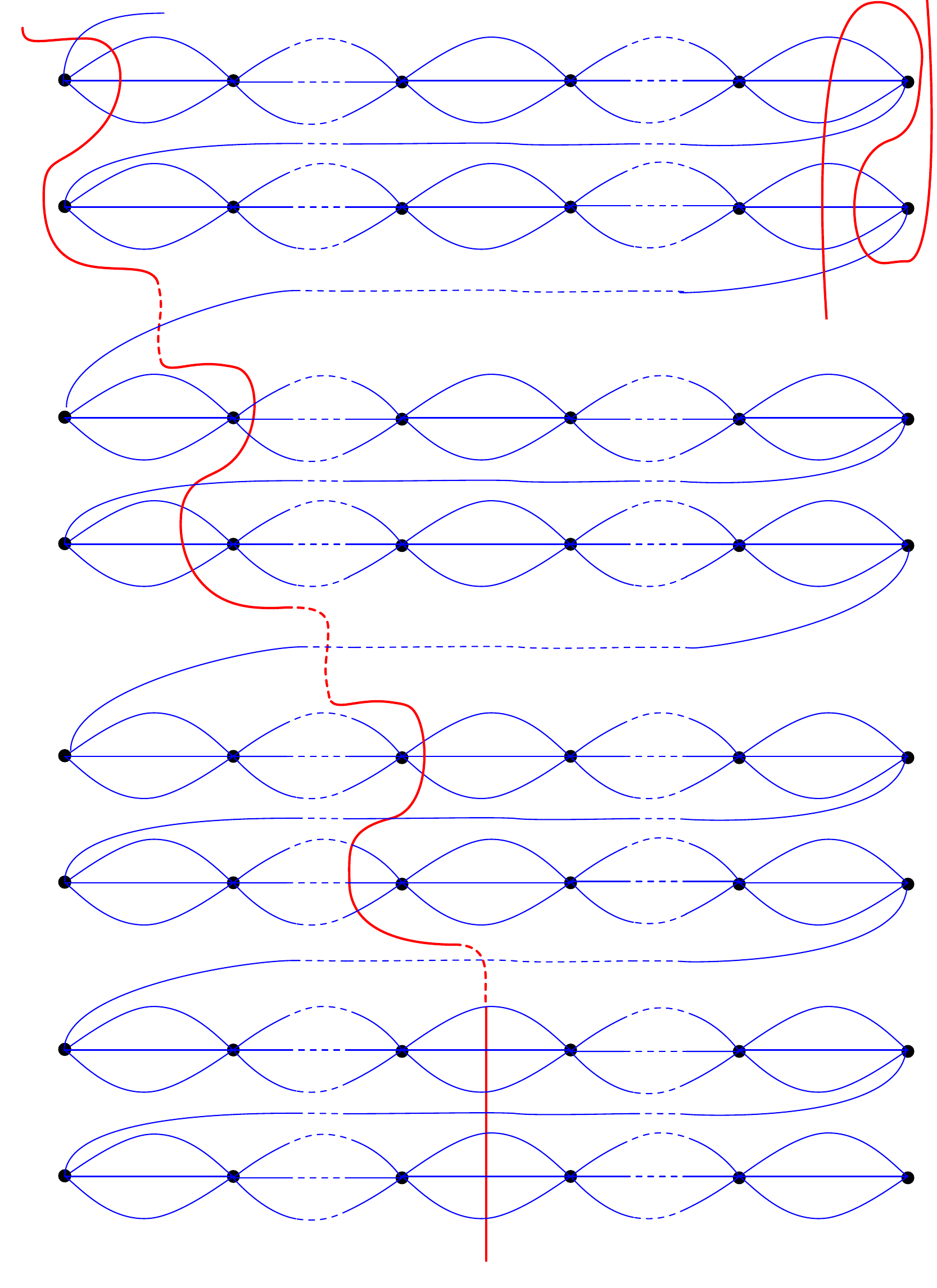}};
\draw (9,1.6) to (9,3.5);
\draw (8.9, 1.6) to (9.1,1.6);
\draw (8.9,3.5) to (9.1,3.5);
\draw (8.9,6.8) to (9.1,6.8);
\draw (9,3.5) to (9,6.8);
\draw (.4,6.8) to (.4,9.9);
\draw (.3,6.8) to (.5,6.8);
\draw (.3,9.9) to (.5,9.9);
\node[rotate=90] at (9.5,2.3) {$\dfrac{n-2}{2}$};
\node[rotate=90] at (9.5,5) {$\left(\dfrac{j-2}{2}\right)n$};
\node[rotate=90] at (.2, 8.5) {$n$};
\end{tikzpicture}
    \caption{The general pairing diagram for $j$ even and $n$ even.}\label{genuspositiveeven}
    \end{center}
\end{figure}
In Figure \ref{genuspositiveodd} the central intersection point $c$ occurs on the central lift of the unstable chain. There are $\dfrac{n-1}{2}$ rows between that intersection point and the intersection point labelled $d$ in Figure \ref{genuspositiveodd}. Then, there are $\dfrac{j-2}{2}n$ rows between the interesction point $d$ and the intersection point $e$, and $n+1$ rows between $e$ and $b$. Therefore,
$$A(b)-A(c)=j\left(\dfrac{n-1}{2}+\dfrac{j-2}{2}n+n+1\right).$$ Then it is straightforward to verify that $A(a)-A(b)=\dfrac{j}{2}+1$. Simplifying, we see that when $n\geq 0$

$$g(Q^{i,j}_n(T_{2,3}))=A(a)=j+\dfrac{j(j+1)}{2}n+1$$ This finished the proof in the case $n \geq 0$. When $n<0$, Figures \ref{genuspositivejodd}-\ref{genuspositiveeven} show both the curve $\alpha(K,n)$ and the curve $\alpha(K,-n)$ and we can see that the difference $g(Q^{i,j}_{-n}(T_{2,3}))-g(Q^{i,j}_{n}(T_{2,3}))=j$. For example in Figure \ref{genuspositiveodd}, the intersection point with the largest Alexander grading in the pairing with $\alpha(K,-n)$ is labelled $a'$: $A(a')=g(Q^{0,j}_{-n}(K))$. Then $A(a)-A(a')=\ell_{a',a}\cdot\delta_{w,z}=j$. Therefore when $n<0$, we have $g(Q^{i,j}_{n}(T_{2,3}))=j+\dfrac{j(j+1)}{2}|n|+1-j$\qedhere

\end{proof}

\begin{proof}[Proof of Theorem \ref{genusunknotsatellite}]
The computation of the genus when the companion knot is the unknot is similar to the proof of Theorem \ref{genusnontrivial} and left to the reader (see Case $0$ in the proof of Theorem \ref{tau} for the relevant pairing diagram).
\end{proof}

\subsection{Fiberedness}\label{fiberedsection}

In this section we prove Theorem \ref{fiberedness}. By \cite{HMS} a necessary condition for a satellite to be fibered is for the companion to be fibered and to determine the fiberedness of any satellite $Q^{i,j}_n(K)$ with fibered companion, it is enough to determine if the satellite knot $Q^{i,j}_n(T_{2,3})$ is fibered. 

Now, recall that a knot $K$ in $S^3$ is fibered if and only if $\HFKhat(S^3,K,g(K))$ has rank one \cites{Nifibered, Juhaszfibered}. In the previous section we determined that largest Alexander grading, so the genus, of any satellite knot with pattern $Q^{i,j}_n$ and companion $K$. In this section, we will determine the rank of the knot Floer homology of $Q^{i,j}_n(T_{2,3})$ in Alexander grading $g(Q^{i,j}_n(T_{2,3}))$. Using this, we will show when this has rank one. In the following, let $g=g(Q^{i,j}_n(T_{2,3}))$

\begin{lemma}\label{rankintop}
$$\dim(\HFKhat(S^3,Q^{i,j}_n(T_{2,3}),g))=\begin{cases} 2(i+1) & \text{if}\,\, n=0\,\, \text{and}\,\, j \geq 1 \\ 2(i+1) & \text{if}\,\, n=-1\,\, \text{and}\,\, j=1 \\ (i+1) & \text{else} \end{cases}$$
\end{lemma}

    \begin{figure}[!tbp]
  \centering
  \begin{minipage}[b]{0.3\textwidth}
  \begin{tikzpicture}
\node[anchor=south west,inner sep=0] at (0,0)    {\includegraphics[width=1.2\textwidth]{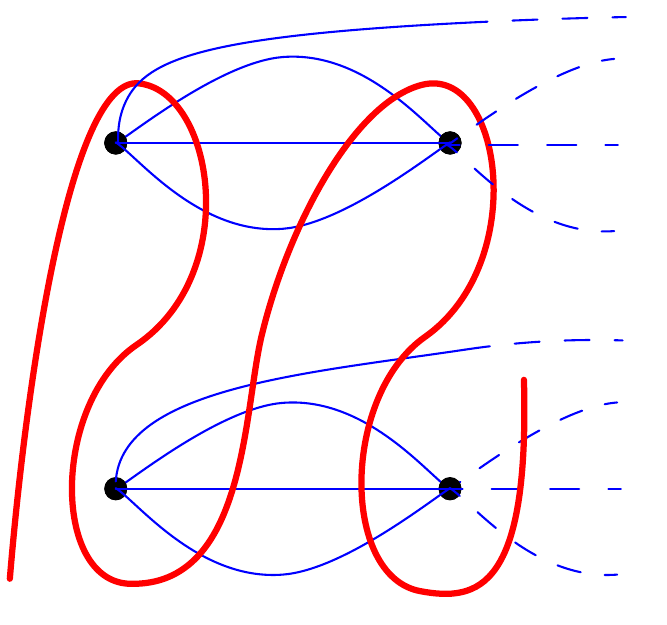}};
\node at (1.55,3.1) {$\tikzcirc{2pt}$};
\node[font=\tiny] at (2,3.1) {$g$};
\node[font=\tiny] at (2,3.7) {$g-1$};
\node[font=\tiny] at (2.3,4.3) {$g-2$};
\node[font=\tiny] at (2.5,2.1) {$0$};
\node[font=\tiny] at (1.2,3) {$a_1$};
\node[font=\tiny] at (2.6,3) {$a_2$};
   \node at (2.25,3) {$\tikzcirc{2pt}$};
\end{tikzpicture}
    \caption{The lifted pairing diagram for $\HFKhat(Q^{0,j}_0(T_{2,3})).$}\label{fibered0}
  \end{minipage}
  \hspace{1in}
  \begin{minipage}[b]{0.3\textwidth}
    \begin{tikzpicture}
    \node[font=\tiny] at (3.3,3.7) {$a$};
    
\node[anchor=south west,inner sep=0] at (0,0)    {\includegraphics[width=1\textwidth]{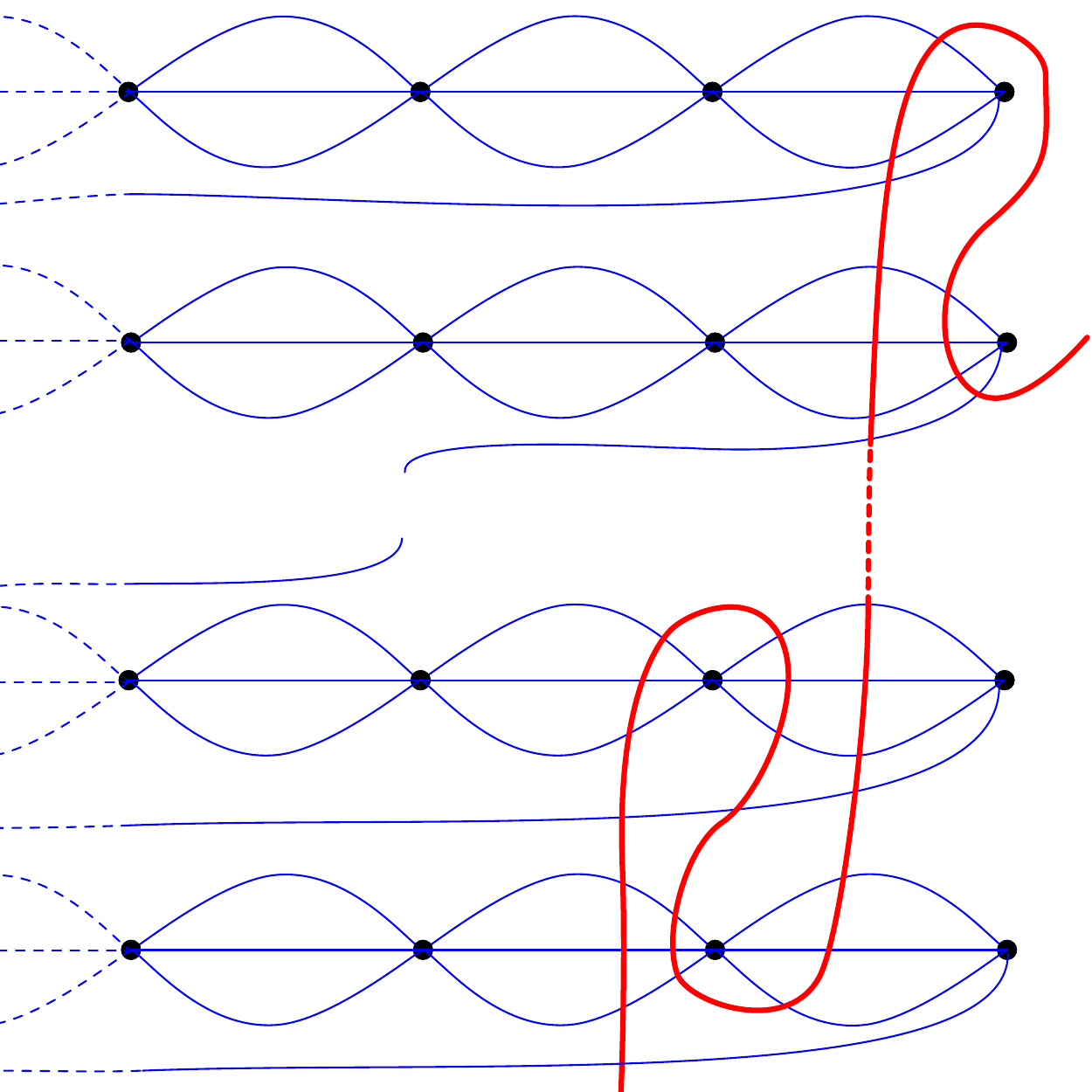}};
 \node[font=\tiny] at (3.4,3.6) {$\tikzcirc{1.5pt}$};
\end{tikzpicture}
    \caption{The pairing diagram for $Q_n^{0,j}(T_{2,3})$ when $n<-1$. The Alexander grading labels of the $\beta$ arcs are as in Figure \ref{fiberednegative}.}\label{fiberednegnlessneg2}
  \end{minipage}
\end{figure}
    
\begin{figure}[!tbp]
  \begin{tikzpicture}
\node[anchor=south west,inner sep=0] at (0,0)    {\includegraphics[width=1\textwidth]{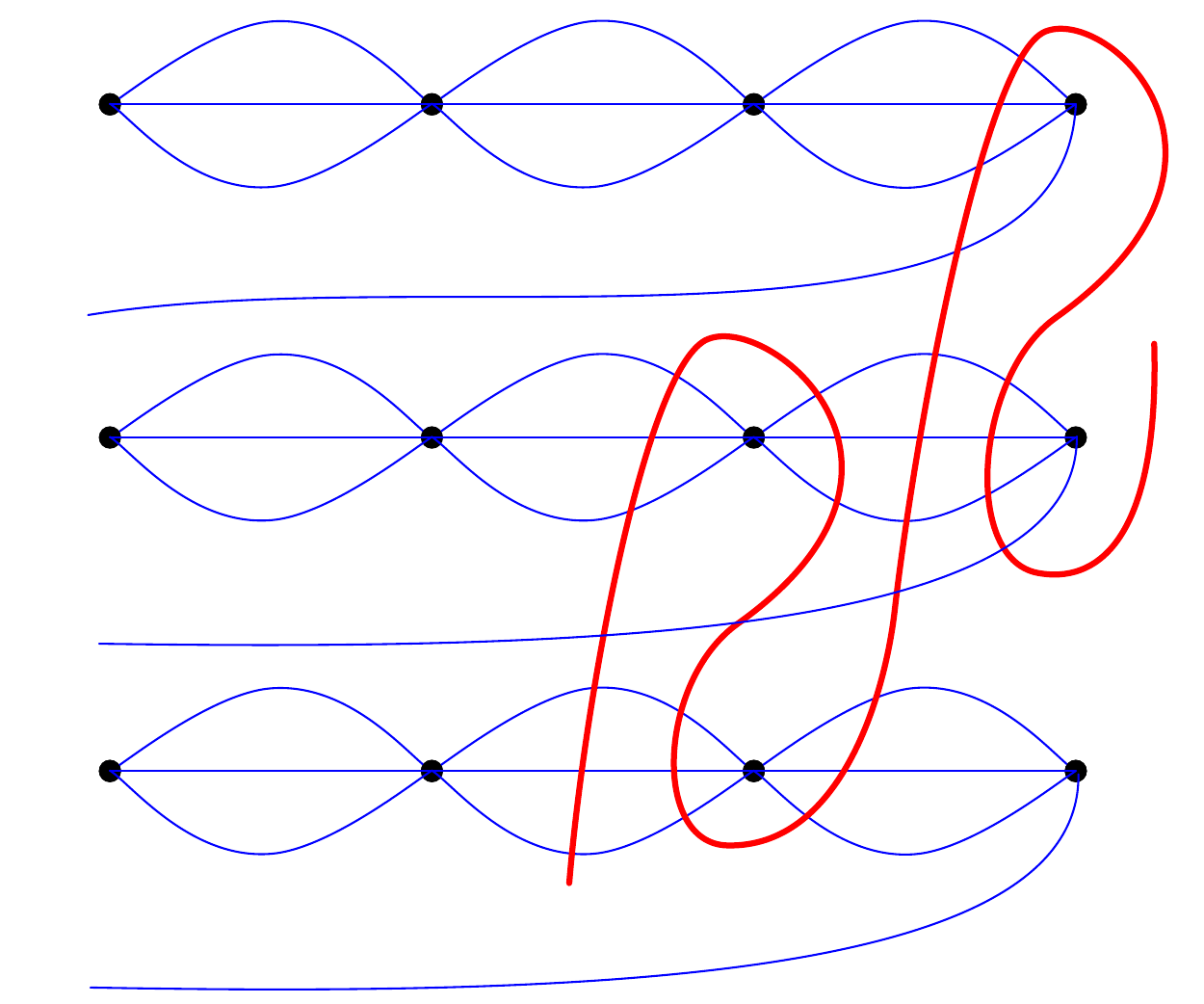}};
\node at (11.35,9.7) {$\tikzcirc{2pt}$};
\node at (7.3,5.7) {$\tikzcirc{2pt}$};
\node at (7.45,5.55) {$b$};
\node at (11.1,9.9) {$a$};
\node at (10.6,5.65) {$g-j$};
\node at (10.6,7.55) {$g-2-j$};
\node at (10.6,6.6) {$g-1-j$};
\node at (10.6,1.9) {$g-2j$};
\node at (10.6,3.7) {$g-2-2j$};
\node at (10.6,2.8) {$g-1-2j$};
\node at (7,8.4) {$g-1$};
\node at (10.6,9.6) {$g$};
\node at (10.6,11.5) {$g-2$};
\node at (10.6,10.55) {$g-1$};
\node[font=\small] at (6.6,5.85) {$g-j+1$};
\node at (6.7,7.6) {$g-1-j$};
\node at (6.7,6.65) {$g-j$};
\node at (6.6,3.7) {$g-1-2j$};
\node at (6.7,2.8) {$g-2j$};
\node at (6.7,1.9) {$g-2j+1$};
\node at (6,4.5) {$g-1-j$};
\end{tikzpicture}
    \caption{The general pairing diagram showing intersection points with largest possible Alexander grading when $i=0$ and $n=-1$. For each increase in $i$, there is one more arc in the top right with Alexander grading label $g$, and one more arc in the second to top row and second to right column with label $g-1+j$.}\label{fiberednegative}
\end{figure}
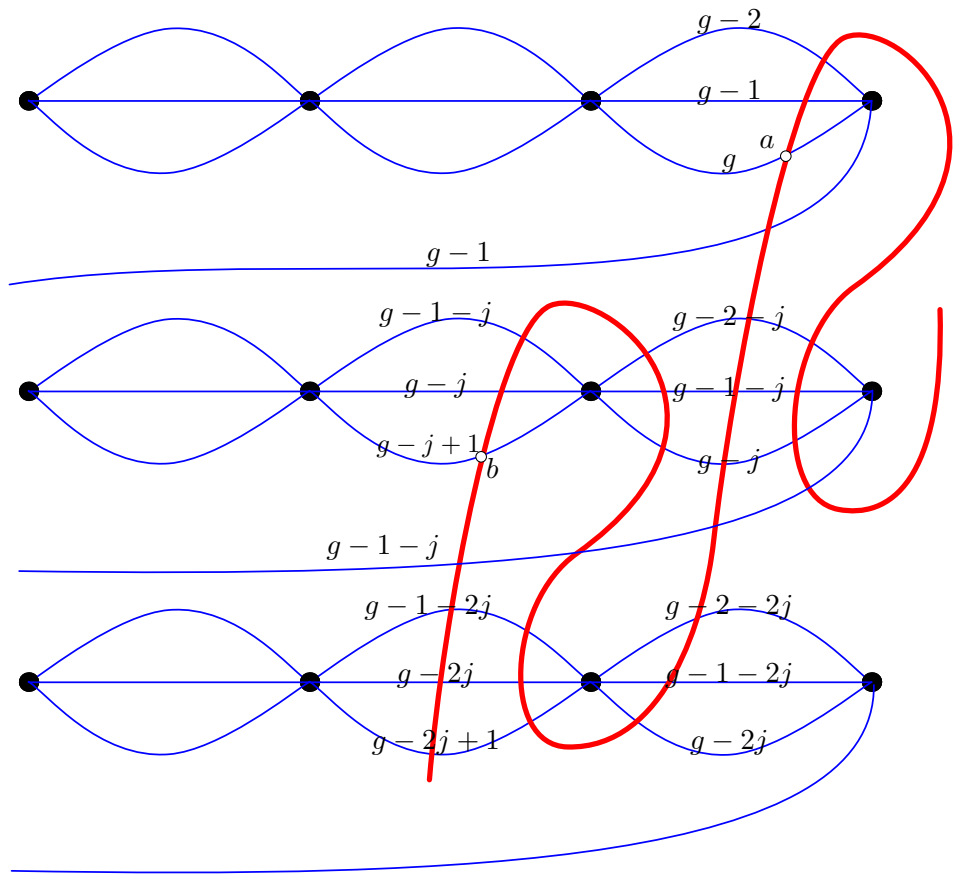

\begin{proof}
 We will first determine the rank of $\HFKhat(S^3,Q^{0,j}_n(T_{2,3}),g)$ then we will see how the rank changes when we increase $i$ by twisting up the $\beta$ curve. 
    
    Suppose first that $n=0$, In this case there are two intersection points in the top row of the pairing diagram that contribute to $\HFKhat(S^3,Q^{0,j}_0(T_{2,3}),g)$, shown in Figure \ref{fibered0} and labelled $a_1$ and $a_2$. Direct inspection shows that there are no other intersection points in this Alexander grading. Therefore $\dim(\HFKhat(S^3,Q^{0,j}_0(T_{2,3})))=2$.

    Suppose now that $n<0$. We first deal with the case $n<-1$. The pairing diagram for this case is shown in Figure \ref{fiberednegnlessneg2}. In that figure we see that there is one intersection point with Alexander grading $g(Q^{0,j}_n(T_{2,3}))$ labelled $a$ in the top row of that figure. Inspection of the pairing diagram shows that all the intersection points in the lower rows of the pairing diagram all carry Alexander gradings $<g$ regardless of the value of $j$. Hence $\dim(\HFKhat(S^3,Q^{0,j}_n(T_{2,3}),g))=1$ when $n<-1$ and $j\geq 1$. 
    
    The pairing diagram for the case $n=-1$ is shown in \ref{fiberednegative}. In that figure, we see that there is one intersection point in Alexander grading $g$ in the top row of the pairing diagram, labelled $a$. All other arcs of the $\beta$ curve in this row (and thus all other intersection points in this row) carry an Alexander grading label $<g$. Consider the next to top row of the pairing diagram. The largest possible Alexander grading is the Alexander grading of the intersection point labelled $b$, which is $g-j+1$. This is always strictly less than $g$ unless $j=1$. Further, regardless of the value of $j$, all other intersection points carry an Alexander grading $\leq g-1$. So in the case that $n=-1$, we see that $\dim(\HFKhat(S^3,Q^{0,1}_{-1}(T_{2,3}),g))=2$ and $\dim(\HFKhat(S^3,Q^{0,j}_{-1}(T_{2,3}),g))=1$ when $j>1$.

    The case that $n\geq 1$ is similar. In that case we see that $\rk(\HFKhat(S^3,Q^{0,j}_n(T_{2,3}),g))=1$ for all $j \geq 1$ and $n \geq 1$. 
    
\begin{figure}[!tbp]
  \centering
  \begin{minipage}[b]{0.3\textwidth}
  \begin{tikzpicture}
\node[anchor=south west,inner sep=0] at (0,0)    {\includegraphics[width=1.2\textwidth]{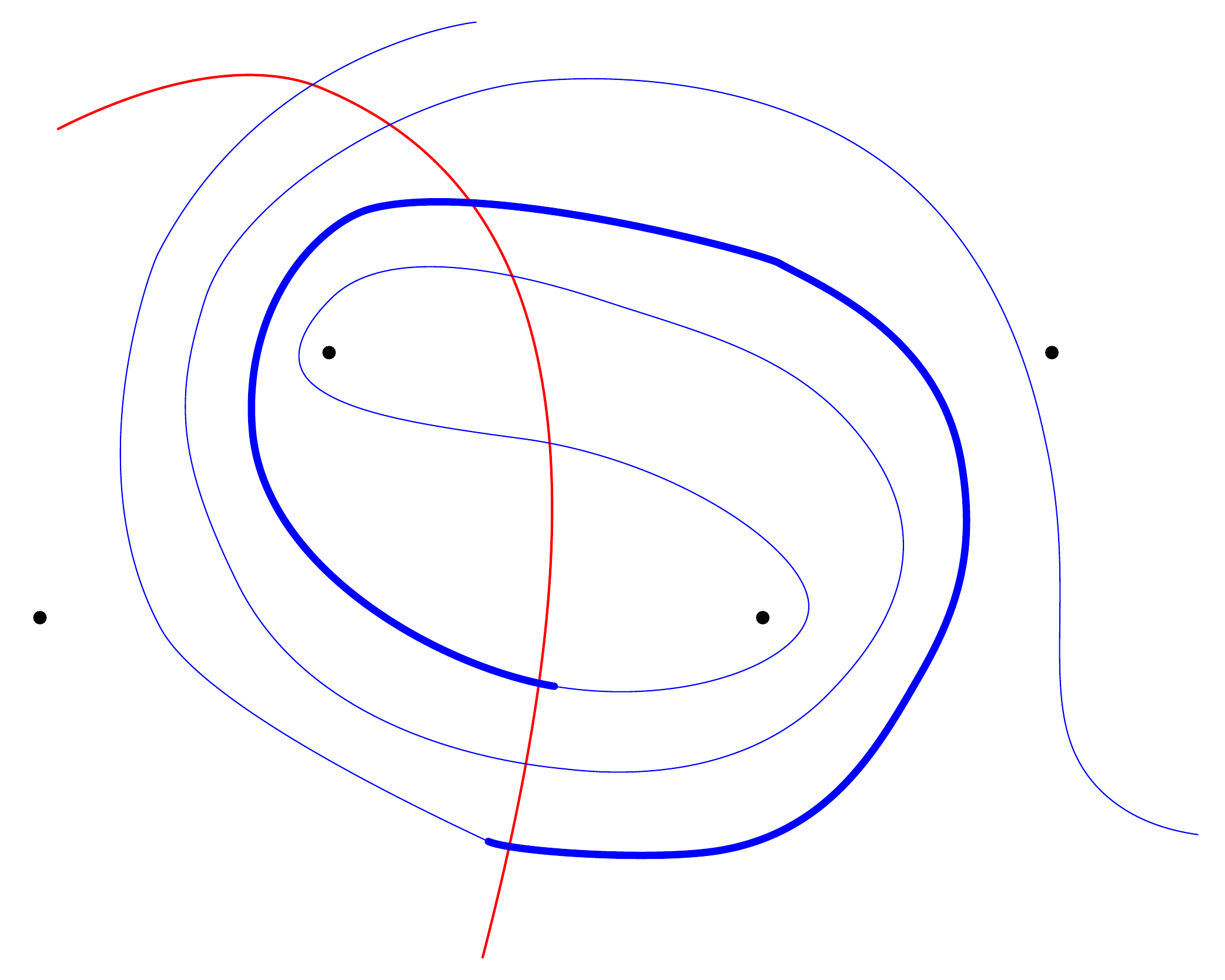}};
\node at (2.13,.55) {$\tikzcirc{2pt}$};
\node at (2.3,1.2) {$\tikzcirc{2pt}$};
\end{tikzpicture}
    \caption{The top left of the pairing diagram when $n>0$ and $i=2$. The intersection points connected by a spiral are in Alexander grading $g=g(Q^{i,j}_n(K)).$}\label{beta(2,j)pairing}
  \end{minipage}\hspace{1in}
  \begin{minipage}[b]{0.3\textwidth}
  \begin{tikzpicture}
\node[anchor=south west,inner sep=0] at (0,0)    {\includegraphics[width=1.2\textwidth]{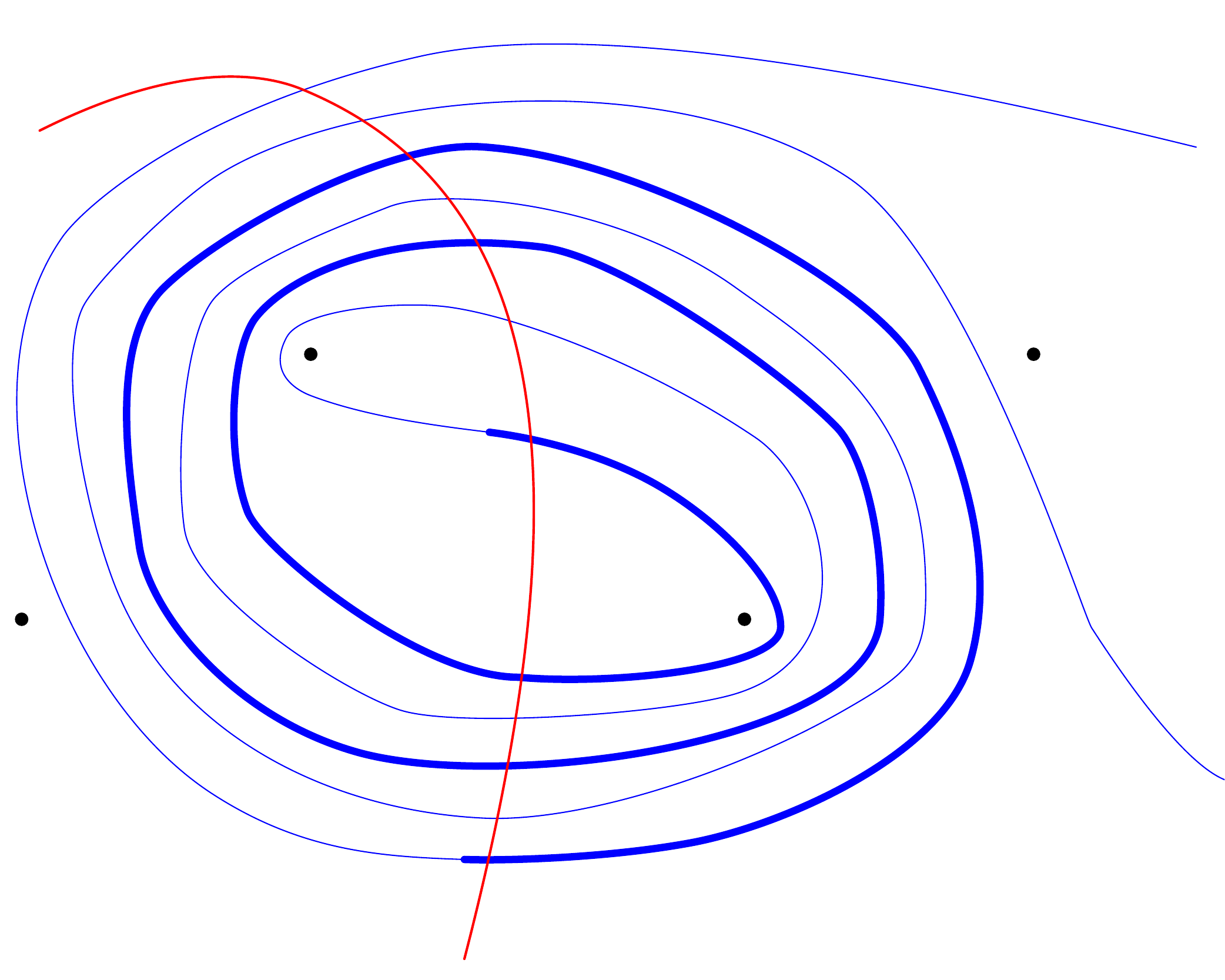}};
\node at (2,.5) {$\tikzcirc{2pt}$};
\node at (2.1,.86) {$\tikzcirc{2pt}$};
\node at (2.15,1.25) {$\tikzcirc{2pt}$};

\end{tikzpicture}
    \caption{The top left of the pairing diagram when $n>0$ and $i=3$. The intersection points connected by a spiral are in the same Alexander grading $g=g(Q^{i,j}_n(K))$.}\label{beta3jpairing}
  \end{minipage}
\end{figure}
    This proves the theorem in the case $i=0$. To deal with the cases $i>0$, recall that the lifted curve $\tbeta(i,j)$ is obtained from the lifted curve $\tbeta(i-1,j)$ by twisting up, as shown in Figure \ref{posttwistbeta}. We see in Figures \ref{beta(2,j)pairing} and \ref{beta3jpairing} that for each intersection point of $\talpha(T_{2,3},n)$ with $\tbeta(i-1,j)$ in Alexander grading $g$, there is one more intersection point of $\talpha(T_{2,3},n)$ with $\tbeta(i,j)$ in that same Alexander grading. The theorem follows. 
\end{proof}

With Lemma \ref{rankintop} in hand, we can prove Theorem \ref{fiberedness} from the introduction.

\begin{proof}[Proof of Theorem \ref{fiberedness}]

By \cite{HMS}, the pattern knot $Q^{i,j}_n$ is fibered in $S^1 \times D^2$ if and only if the satellite knot $Q^{i,j}_n(T_{2,3})$ is fibered in $S^3$. By the computation in lemma \ref{rankintop} and the fact that a knot in $S^3$ is fibered if and only if $\rank(\HFKhat(S^3,K,g(K)))=1$ \cite{Nifibered}, we see that the pattern knot $Q^{i,j}_n$ is fibered for $j \geq 2$ if and only if $i=0$ and $n \neq 0$ and when $j=1$ $Q^{i,1}_n$ is fibred if and only if $i=0$ and $n \neq 0, -1$. 
\end{proof}

\section{Thickness and unknotting number of generalized Mazur satellites with non-trivial companions}

In this section we prove Theorems \ref{theoremthinnontrivial} and \ref{theoremthintrivial} from the introduction, which give lower bounds on the thickness and torsion order for $n$-twisted satellites with patterns $Q^{i,j}$ and arbitrary non-trivial companions. Recall that a knot $K$ is called Floer thin if for all pairs of generators $x$ and $y$ of $\HFKhat(S^3,K)$ $M(x)-A(x)=M(y)-A(y)$. Equivalently, if we define the $\delta$-grading as $\delta(x)=M(x)-A(x)$ a knot is thin if the $\delta$ grading is constant for all generators. This concept was introduced in \cite{quasialt}, where they showed that all quasi-alternating knots have thin knot Floer homology. 

Suppose that there is a length $k$ vertical arrow between two distinct generators $x$ and $y$ of the knot Floer homology. Then $A(y)=A(x)-k$ and $M(y)=M(x)-1$. In this case, if we consider the collapsed $\delta$ grading we see that $\delta(x)=M(x)-A(x)$ and $\delta(y)=M(x)-1-(A(x)-k)=\delta(x)+k-1$. So if $k>1$, these two generators are supported in distinct $\delta$ gradings, and so the knot $K$ is not Floer homologically thin.

The proof of Theorem \ref{theoremthinnontrivial} relies on the observation that, since knot Floer homology detects the genus of knots, if a knot $K$ is non-trivial there is always a portion of the immersed curve in each column that exhibits this. We are only interested in the portion of the immersed curve in the second column of the pairing diagram which is shown in Figure \ref{nonthinnontrivial}.

\begin{figure}[!tbp]
  \begin{tikzpicture}
  \node at (2.8,1.3) {$\tikzcirc{2pt}$};
  \node[font=\tiny] at (2.6,1.2) {$x$};
  \node[font=\tiny] at (3.5,.3) {$y$};
  \node at (3.8,.4) {$\tikzcirc{2pt}$};
\node[anchor=south west,inner sep=0] at (0,0)    {\includegraphics[width=1\textwidth]{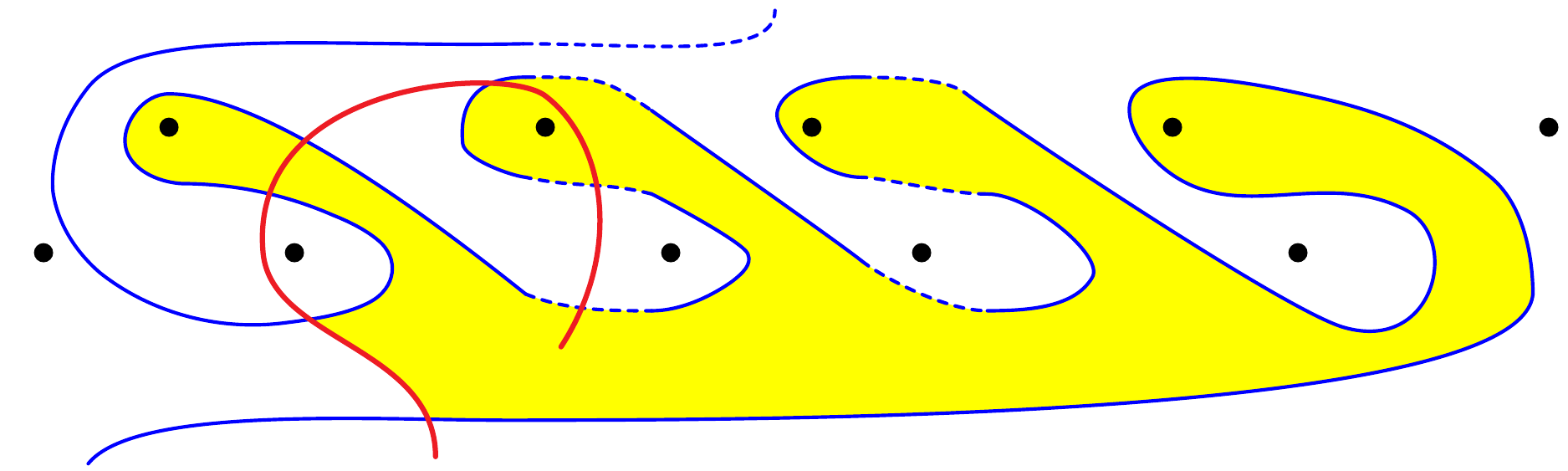}};
\end{tikzpicture}
    \caption{Illustration of two intersection points in the pairing diagram with a length $j+1$ vertical differential between them. The red arc is a portion of $\alpha(K,n)$ that exhibits the genus detection of knot Floer homology.}\label{nonthinnontrivial}
\end{figure}
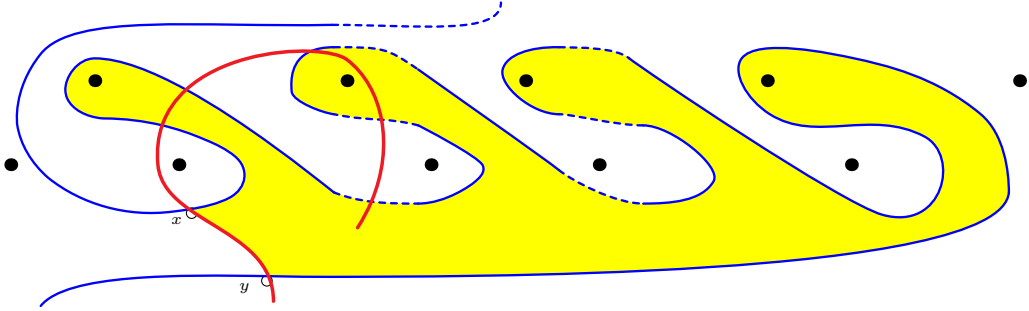
\begin{proof}[Proof of Theorem \ref{theoremthinnontrivial}]
    Suppose $K$ is a non-trivial companion knot. Then the curve $\alpha(K,n)$ contains a portion as shown in Figure \ref{nonthinnontrivial} by the genus detection of knot Floer homology. Indeed the curve in each column, when paired with the meridian (a vertical line) recovers the knot Floer homology of $K$. The portion of the curve $\alpha(K,n)$ passes at height $g(K)$ (relative to height zero in this column) and must turn down on either side. We see that there are two intersection points, denoted $x$ and $y$, that are connected by a length $j+1$ vertical differential. Hence the knot $Q^{i,j}_n(K)$ is not Floer thin. 
\end{proof}
\begin{figure}[!tbp]
  \begin{tikzpicture}
\node[anchor=south west,inner sep=0] at (0,0)    {\includegraphics[width=1\textwidth]{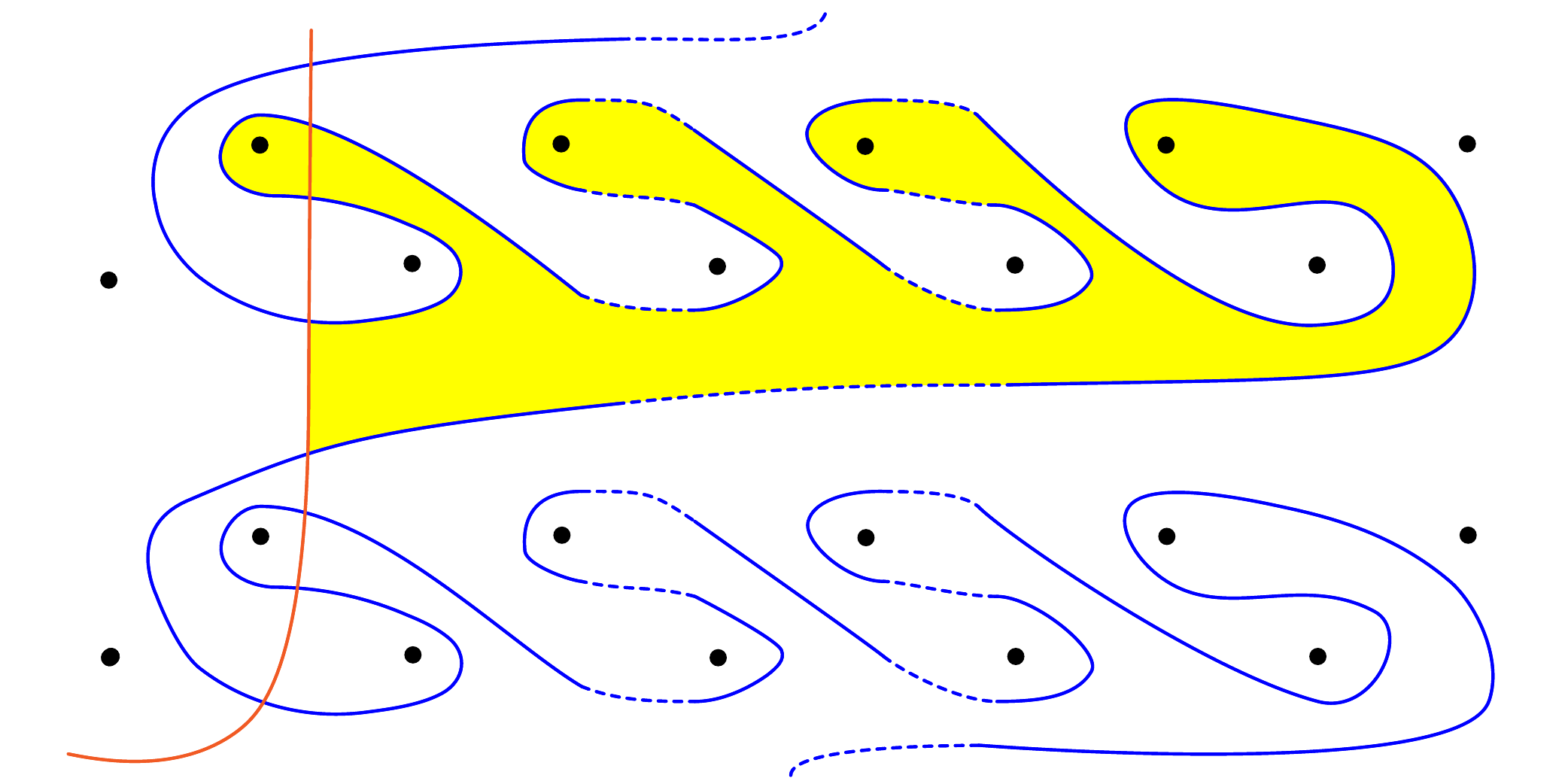}};
\node at (2.75,2.95) {$\tikzcirc{2pt}$};
\node[font=\tiny] at (2.6,3) {$y$};
\node[font=\tiny] at (2.6,4) {$x$};
\node at (2.75,4.1) {$\tikzcirc{2pt}$};
\end{tikzpicture}
    \caption{The pairing $\CFKhat(\alpha(U,n),\beta(i,j))$ when $n<-1.$}\label{nonthintrivialnegative2}
\end{figure}
\begin{figure}[!tbp]
  \begin{tikzpicture}
\node[anchor=south west,inner sep=0] at (0,0)    {\includegraphics[width=1\textwidth]{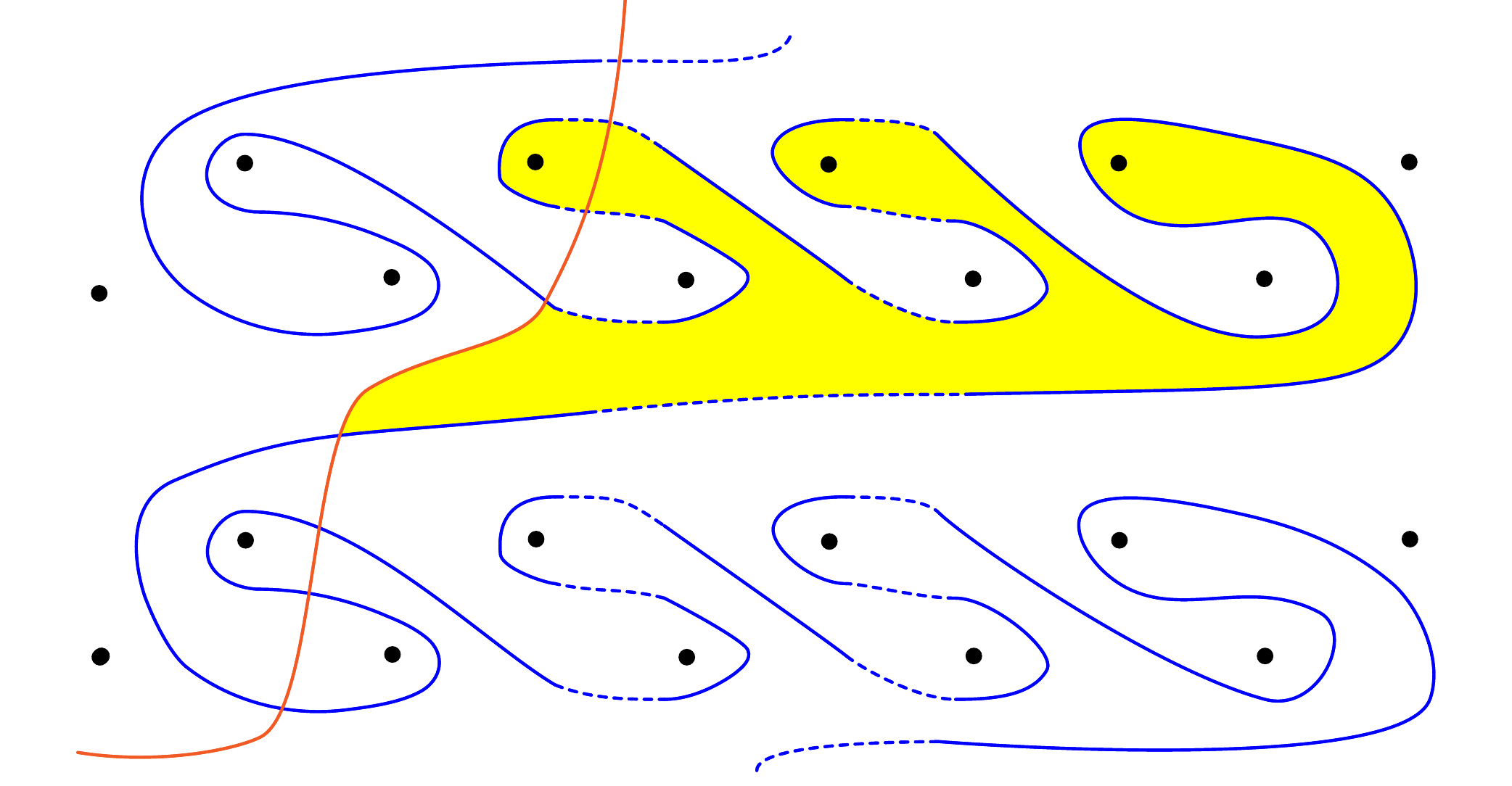}};
\node[font=\tiny] at (3.2, 3.25) {$\tikzcirc{2pt}$};
\node[font=\tiny] at (3,3.4) {$y$};
\node[font=\tiny] at (4.8,4.55) {$x$};
\node at (5,4.5) {$\tikzcirc{2pt}$};
\end{tikzpicture}
    \caption{The pairing $\CFKhat(\alpha(U,-1), \beta(i,j))$.}\label{nonthintrivialnegative}
\end{figure}
Next, we investigate what happens when the companion knot $K$ is the unknot. In that case, since $Q^{i,j}_0(U)\sim U$, it is clear that the $0$-twisted satellite is Floer thin. In all other cases, we prove the following
\begin{proof}[Proof of Theorem \ref{theoremthintrivial}]
   We prove the case $n<0$, the case $n>0$ is similar and left to the reader. Since $K=U$, the pairing diagram $\CFKhat(\alpha(U,n),\beta(i,j))$ has the form shown in Figures \ref{nonthintrivialnegative2} and \ref{nonthintrivialnegative}. Figure \ref{nonthintrivialnegative2} shows the case when $n<-1$ and Figure \ref{nonthintrivialnegative} shows the case when $n=-1$. The main difference is that $\alpha(U,-1)$ only crosses through one row before veering into the next column, and $\alpha(U,n)$ for $n<-1$ crosses through more than one row. In the case $n<-1$ inspecting Figure \ref{nonthintrivialnegative2} we see that there is a length $j+1$ vertical differential between the intersection points labelled $x$ and $y$. In the case $n=-1$, Figure \ref{nonthintrivialnegative} shows that there is a length $j$ vertical differential between the intersection points labelled $x$ and $y$. Inspection of the pairing diagram shows that these are the longest possible vertical differentials in the complex $\CFKhat_{\F[U,V]/UV}(S^3,Q^{i,j}_n(U))$. Therefore, when $n<-1$, the satellite knot $Q^{i,j}_n(U)$ is never thin and when $n=-1$, the satellite knot $Q^{i,j}_n(U)$ is thin if and only if $j=1$. \qedhere
\end{proof}

The \emph{$V$-torsion order} of a knot, $\Ord_V(K)$ is the smallest integer $k$ with the property that $V^k(Tors(\HFK^{-}_{\F[V]}(S^3,K)))=0$. The proofs of Theorem \ref{theoremthinnontrivial} and \ref{theoremthintrivial}, in addition to determining when the satellite knots $Q^{i,j}_n(K)$ are not thin, also gives a lower bound on the torsion order of $Q^{i,j}_n(K)$: 

\begin{corollary}
    When $K$ is non-trivial, or when $K=U$ and $n \neq -1,$ $\Ord_V(Q^{i,j}_n(K)) \geq j+1$. When $K=U$ and $n=-1$, $\Ord_V(Q^{i,j}_{-1}(U)) \geq j$.
\end{corollary}
\begin{proof}
    The proofs of Theorems \ref{theoremthinnontrivial} and \ref{theoremthintrivial} shows that the complex $g\CFK^{-}(Q^{i,j}_n(K))$ has a length $j+1$ vertical differential in the case $K$ is non-trivial or $K=U$ and $n \neq -1$, or a length $j$ vertical differential in the case $K=U$ and $n=-1$. 
\end{proof}
Since the torsion order is a lower bound for the unknotting number of a knot \cite{unknotting} the following Corollary is immediate. This verifies a conjecture of Hom, Lidman and Park in the case that the pattern knot is an $n$-twisted generalized Mazur pattern \cite[Conjecture 1.10]{hom2022unknotting}. 

\begin{corollary}
    The satellite knots $Q^{i,j}_n(K)$ with non-trivial companions have unknotting number at least $j+1=w(Q^{i,j}_n)+1$. 
\end{corollary}

\section{Heegaard Floer Concordance Invariants and Twisting}\label{tausectionMazur}

In this section we determine the dependence of the invariants $\tau$ and $\epsilon$ on the parameters $i$, $j$ and the twisting parameter $n$. First, we will determine the invariants $\tau(Q^{0,j}_n(K))$ and $\epsilon(Q^{0,j}_n(K))$ in terms of $\tau(K)$, $\epsilon(K)$, $j$ and $n$, and then we will show that $\tau(Q^{i,j}_n(K))$ and $\epsilon(Q^{i,j}_n(K))$ are independent of $i \in \Z_{\geq 0}$.

Recall that by Theorem \ref{Pairing}, the complex $\CFK_{\F[U,V]/UV}(S^3,Q^{0,j}_n(K))$ can be extracted from the pairing diagram by considering disks that cover either the $z$ or $w$ basepoint and do not cover both. Let $\CFK_{\F[V]}(S^3,Q^{0,j}_n(K))$ denote the complex obtained by only counting disks that cross the $z$-basepoint (so the $U=0$ quotient of $\CFK_{\F[U,V]/UV}(S^3,Q^{0,j}_n(K))$). Theorem \ref{Pairing} shows that this complex is isomorphic to $g\CFK^{-}(S^3,Q^{0,j}_n(K))$ and so has homology isomorphic to $\HFK^{-}(S^3,Q^{0,j}_n(K))$ as an $\F[V]$ module. The structure theorem for $\HFK^{-}$ implies that it has a single free $\F[V]$ summand, and the generator of this summand has Alexander grading $\tau(Q^{0,j}_n(K))$ by \cite[Appendix A]{taucombinatorialFLoer}. Therefore, to determine the value of $\tau$ of satellites with arbitrary companions, arbitrary framings and patterns $Q^{0,j}$, we will use Theorem \ref{Pairing} to identify a collection of intersection points, so generators of the complex $\CFK_{\F[U,V]/UV}(S^3,Q^{0,j}_n(K))$, that form a subcomplex with respect to the vertical $z$-basepoint differentials (when we set $U=0$) and generate the $\F[V]$ free part of the homology of $\HFK^{-}(S^3,Q^{0,j}_n(K))$. Setting $V=1$ in this complex gives $\HFhat(S^3)$ and so, said another way, we identify a cycle in $\HFKhat(S^3,Q^{0,j}_n(K))$ that, in the $V$-filtration, survives in $\HFhat(S^3)$. 

We will see in the pairing diagram that the form of this subcomplex is completely determined by the piece of the essential component of $\talpha(K,n)$ in the first column of the lifted pairing diagram corresponding to Lemmas \ref{curveshape1} and \ref{curveshape2}.
Once we identify the cycle that generates the $\F[V]$-free part of the homology (so survives in the spectral sequence to $\HFhat(S^3)$), it can be extended to be the distinguished element of some vertically simplified basis, as in \cite[Section2.3]{Hom}. Then it is possible to determine $\epsilon(Q^{0,j}_n(K))$ from the horizontal ($w$-basepoint crossing) differentials. By \cite[Definition 3.4 and Lemma 3.2]{Hom} $\epsilon(K)=1,-1$ or $0$ depending on whether the distinguished element of the vertically simplified basis has a horizontal differential into it, out of it, or neither respectively.

As in Lemmas \ref{curveshape1} and \ref{curveshape2}, we distinguish multiples cases for the essential component of $\talpha(K,n)$ depending on $\tau(K)$, $\epsilon(K)$, and $n$. In each case the form of the pairing diagram, and thus the subcomplex carrying the $\F[V]$ free part of the homology, changes. Moreover, the Alexander grading labels of the arcs of the $\beta$ curve relative to the central intersection point of the pairing diagram also change. 
As in the proof of Theorem \ref{genusnontrivial}, there are also multiple sub-cases depending on whether $j$ and $n$ are even or odd. We mostly draw the pairing diagram in the case $j$ is odd, since the pictures are slightly simpler. We analyze the case $j$ even and $n$ odd in Figure \ref{taunegepsilonnegnpos}, and leave the rest of the cases where $j$ is even to the reader.

\begin{figure}[!tbp]
  \centering
  \begin{minipage}[b]{0.32\textwidth}
  \begin{tikzpicture}\hspace{-.6in}
\node[anchor=south west,inner sep=0] at (0,0)    {\includegraphics[width=1.3\textwidth]{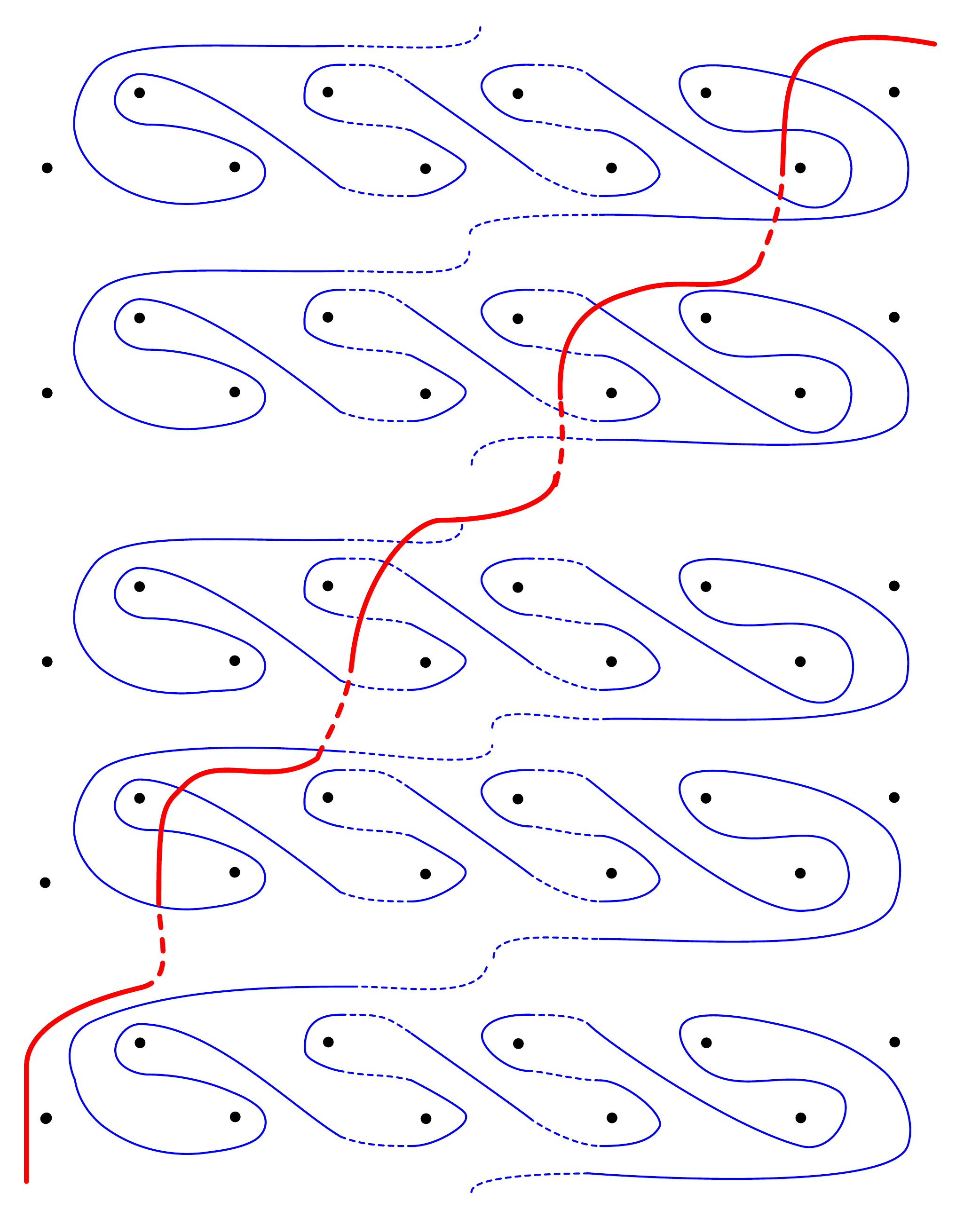}};
\draw (0,2.75) to (0,4.1);
\draw (-.1,2.75) to (.1,2.75);
\draw (-.1,4.1) to (.1,4.1);
\node[font=\tiny, rotate=90] at (-.3,3.3) {$\dfrac{j-1}{2}n$};
\node at (1,1.9) {$\tikzcirc{2pt}$};
\node at (1.95,2.85) {$\tikzcirc{2pt}$};
\node at (2.4,4.05) {$\tikzcirc[]{2pt}$};
\node at (2.1,3.2) {$\tikzcircle{2pt}$};
\node at (2.15,3.6) {$\tikzcirclee{2pt}$};
\node at (1,2.35) {$\tikzcirclee{2pt}$};
\node[font=\tiny] at (.6,1.85) {$a_{3}$};
\node[font=\tiny] at (1.7,2.95) {$a_{2}$};
\node[font=\tiny] at (2.2,3.05) {$a_{1}$};
\node[font=\tiny] at (2.3,3.45) {$b_1$};
\node[font=\tiny] at (.8,2.35) {$b_3$};
\node[font=\tiny] at (2.4,4.2) {$c$};
\end{tikzpicture}
    \caption{$\epsilon(K)=\tau(K)=0$ and $n<0$.}\label{tauepsilon0nneg}
  \end{minipage}\hspace{1in}
  \begin{minipage}[b]{0.3\textwidth}
  \begin{tikzpicture}\hspace{-.3in}
\node[anchor=south west,inner sep=0] at (0,0)    {\includegraphics[width=1.3\textwidth]{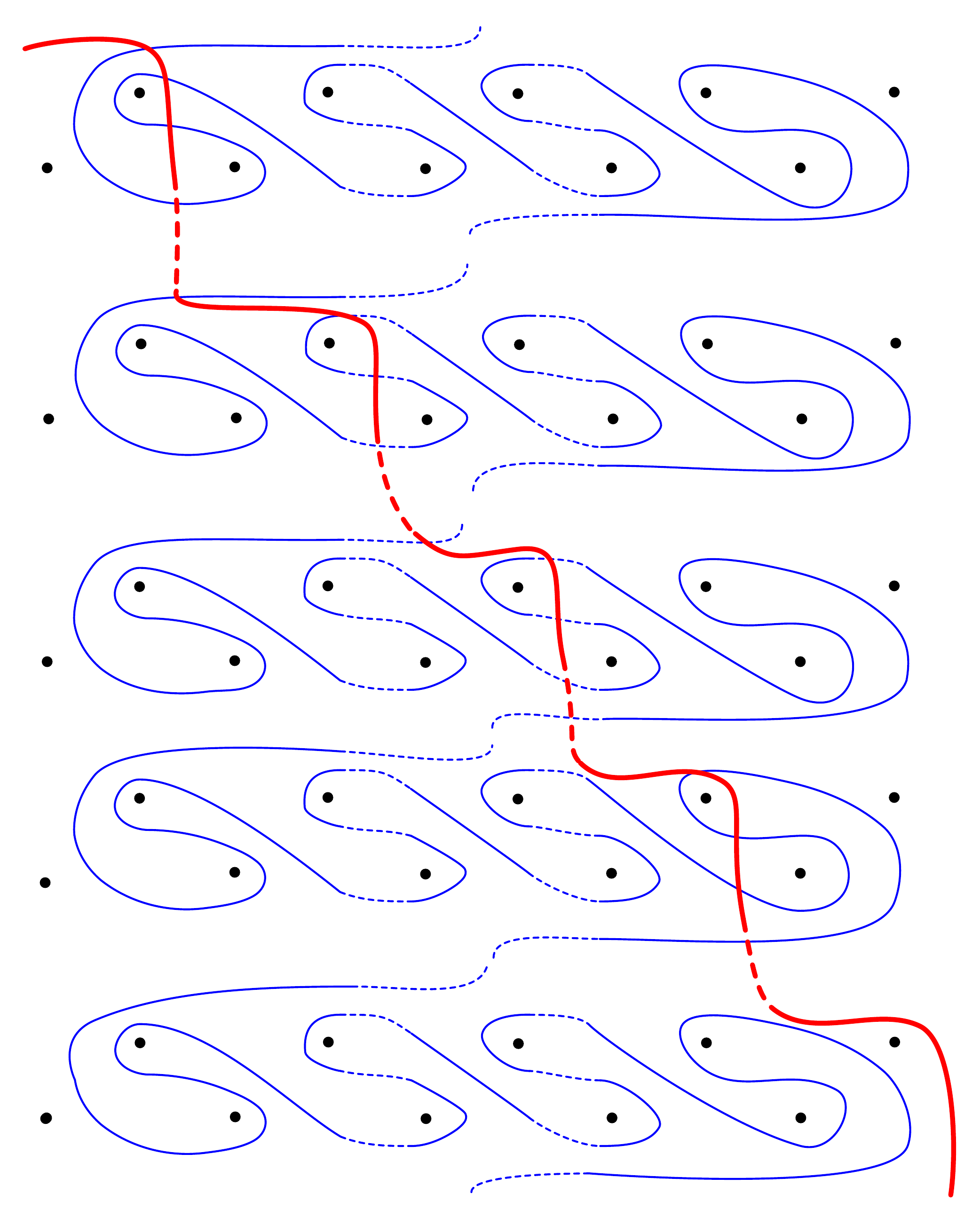}};
\draw (0,3.75) to (0,5.2);
\draw (-.1,3.75) to (.1,3.75);
\draw (-.1,5.2) to (.1,5.2);
\node[font=\tiny, rotate=90] at (-.3,4.5) {$\dfrac{j-1}{2}n$};
\node at (1,5.2) {$\tikzcircle{2pt}$};
\node at (1,5.7) {$\tikzcirc{2pt}$};
\node at (2.4,3.8) {$\tikzcirc[]{2pt}$};
\node at (.9,6.5) {$\tikzcirc{2pt}$};
\node at (2,5) {$\tikzcirclee{2pt}$};
\node[font=\tiny] at (.8,5.3) {$a_{1}$};
\node[font=\tiny] at (.8,5.6) {$a_{2m}$};
\node[font=\tiny] at (1,6.7) {$a_{2m+1}$};
\node[font=\tiny] at (2.2,5) {$b$};
\node[font=\tiny] at (2.4,3.95) {$c$};
\end{tikzpicture}\hspace{1in}
    \caption{$\epsilon(K)=\tau(K)=0$ and $n>0$.}\label{tauepsilon0npos}
  \end{minipage}
\end{figure}

\begin{proof}[Proof of Theorem \ref{tau}]
The proof is divided into many cases, first by the value of $\epsilon(K)$, then into whether $\tau(K)$ is positive or negative, and then into various cases of whether or not $n \geq 2\tau(K)$ or $n <2\tau(K)$. The pictures look slightly different when, for example $\tau(K) \geq 0$ and either $n\leq0\leq 2\tau(K)$ or $0\leq n\leq 2\tau(K)$, so we separately analyze those cases as well.

\textbf{Case 0 $\epsilon(K)=0$}: In this case it follows that $\tau(K)=0$ \cite{Hom}, and the essential component of the immersed curve $\alpha(K,n)$ is the same as the immersed curve for the $n$-framed unknot complement, and so $\tau(Q^{0,j}_n(K))=\tau(Q^{0,j}_n(U))$. The case $n=0$ is clear, since $Q^{i,j}_0(U) \sim U$. We indicate the pairing diagrams for the cases $n<0$ and $n>0$ in Figures \ref{tauepsilon0nneg} and \ref{tauepsilon0npos}. \\
In the case that $n<0$, the intersection points labelled $\{a_k\}_{k=1}^{2m+1}$ form a subcomplex of $\CFK_{\F[V]}(S^3,Q^{0,j}_n(U))$ with respect to the vertical differentials that contains an $\F[V]$ free part. Setting $V=1$ in the above subcomplex, we see that the cycle $\sum a_{2i+1}$ generates $\HFhat(S^3)$. Note that the intersection points $b_{2i+1}$ satisfy $\partial^h(\sum b_{2i+1})=U \sum a_{2i+1}$, so that $\epsilon(Q^{0,j}_n(U))=1$ by \cite[Section 3]{Hom}. Recall that $A(\sum a_{2i+1})=\max\{A(a_{2i+1})\}$, and from this it is easy to see that $\tau(Q^{0,j}_n(U))=A(a_1)$. Then, $A(a_1)=\ell_{c,a_1} \cdot \delta_{w,z}$, where $\ell_{c,a}$ is the arc of the lifted $\beta$ curve from $c$ to $a$ by \cite[Lemma 4.1]{Chen}. Now as remarked the Alexander grading labels of arcs of the $\beta$ curve change by $-j$ for each row we go down in the pairing diagram, so we see that $A(a)=\tau(Q^{0,j}_n(U))=-j\left(\dfrac{j-1}{2}|n|-1\right)=\dfrac{j(j-1)}{2}n+j$\\

\begin{figure}[!tbp]
  \centering
  \begin{minipage}[b]{0.3\textwidth}
  \begin{tikzpicture}\hspace{-.4in}
\node[anchor=south west,inner sep=0] at (0,0)    {\includegraphics[width=1.2\textwidth]{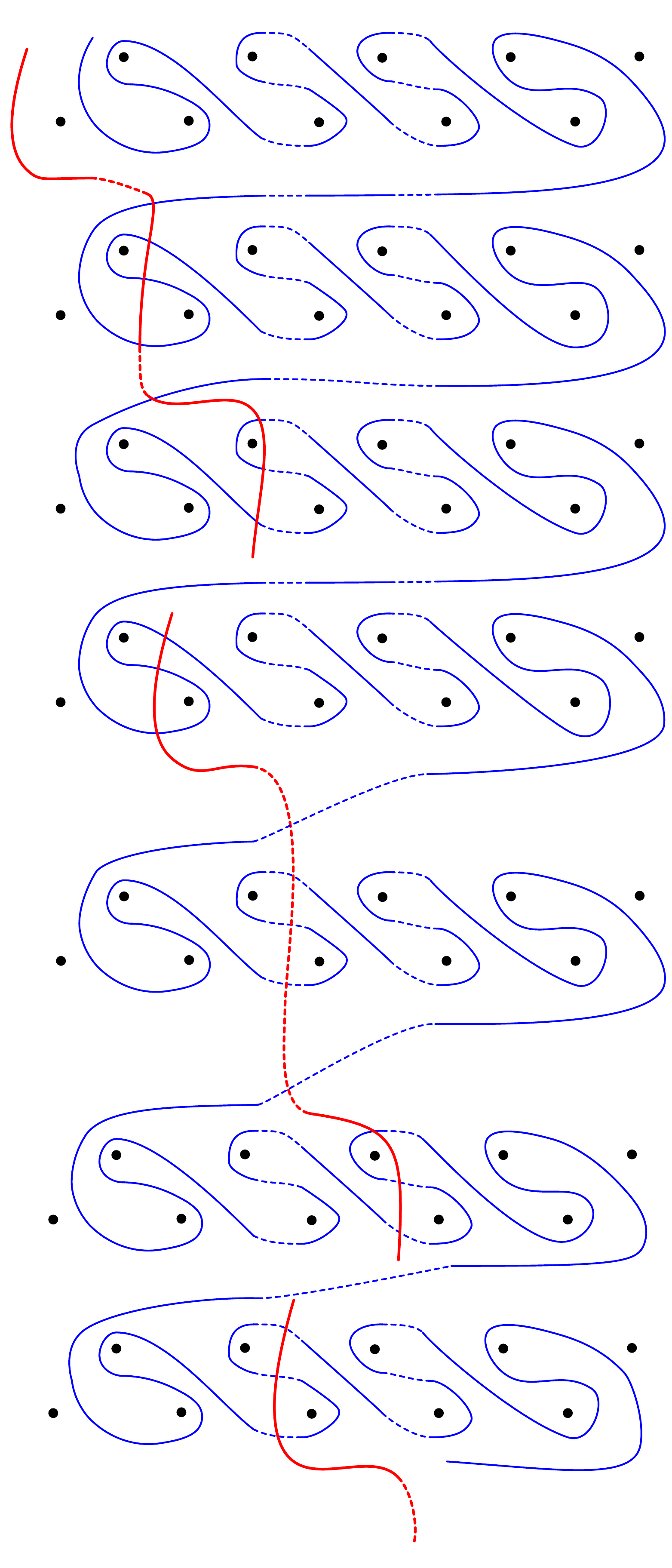}};
\node at (1.15,10.2) {$\tikzcirc{2pt}$};
\node at (1.05,9.15) {$\tikzcirc{2pt}$};
\node at (1.2,8.75) {$\tikzcircle{2pt}$};
\node at (1.9,8.6) {$\tikzcirclee{2pt}$};
\node at (2.5,2.1) {$\tikzcirc[]{2pt}$};
\node[font=\tiny] at (1.2,10.4) {$a_{2i+1}$};
\node[font=\tiny] at (1.15,9.3) {$a_{2i+2}$};
\node[font=\tiny] at (.7,8.8) {$a_{1}$};
\node[font=\tiny] at (2,8.8) {$b$};
\node[font=\tiny] at (2.5,2.3) {$c$};

\draw[very thin] (0,2.1) to (0,3.3);
\draw[very thin] (-.1,2.1) to (.1,2.1);
\draw[very thin] (-.1,3.3) to (.1,3.3);
\draw[very thin] (0,3.3) to (0,6);
\draw[very thin] (-.1,3.3) to (.1,3.3);
\draw[very thin] (-.1,6) to (.1,6);
\draw[very thin] (0,6) to (0,8.8);
\draw[very thin] (-.1,8.8) to (.1,8.8);
\node[font=\tiny, rotate=90] at (-.2,2.6) {$\tau(K)$};
\node[font=\tiny, rotate=90] at (-.2,7.5) {$2\tau(K)$};
\node[font=\tiny, rotate=90] at (-.4,4.6) {$\left(\dfrac{j-1}{2}\right)n-2\tau(K)$};
\draw[very thin] (5.1,10.2) to (5.1,8.8);
\draw[very thin] (5,10.2) to (5.2,10.2);
\draw[very thin] (5,8.8) to (5.2,8.8);
\node[font=\tiny, rotate=90] at (5.3,9.5) {$n-2\tau(K)$};
\end{tikzpicture}
    \caption{The pairing diagram when $\tau(K)\geq 0$, $\epsilon(K)=1$ and $n\geq 2\tau(K)$ and $j$ odd.}\label{taupositiveepsilononenbig2tau}
  \end{minipage}\hspace{2in}
  \begin{minipage}[b]{0.3\textwidth}
  \begin{tikzpicture}\hspace{-1in}
\node[anchor=south west,inner sep=0] at (0,0)    {\includegraphics[width=1.2\textwidth]{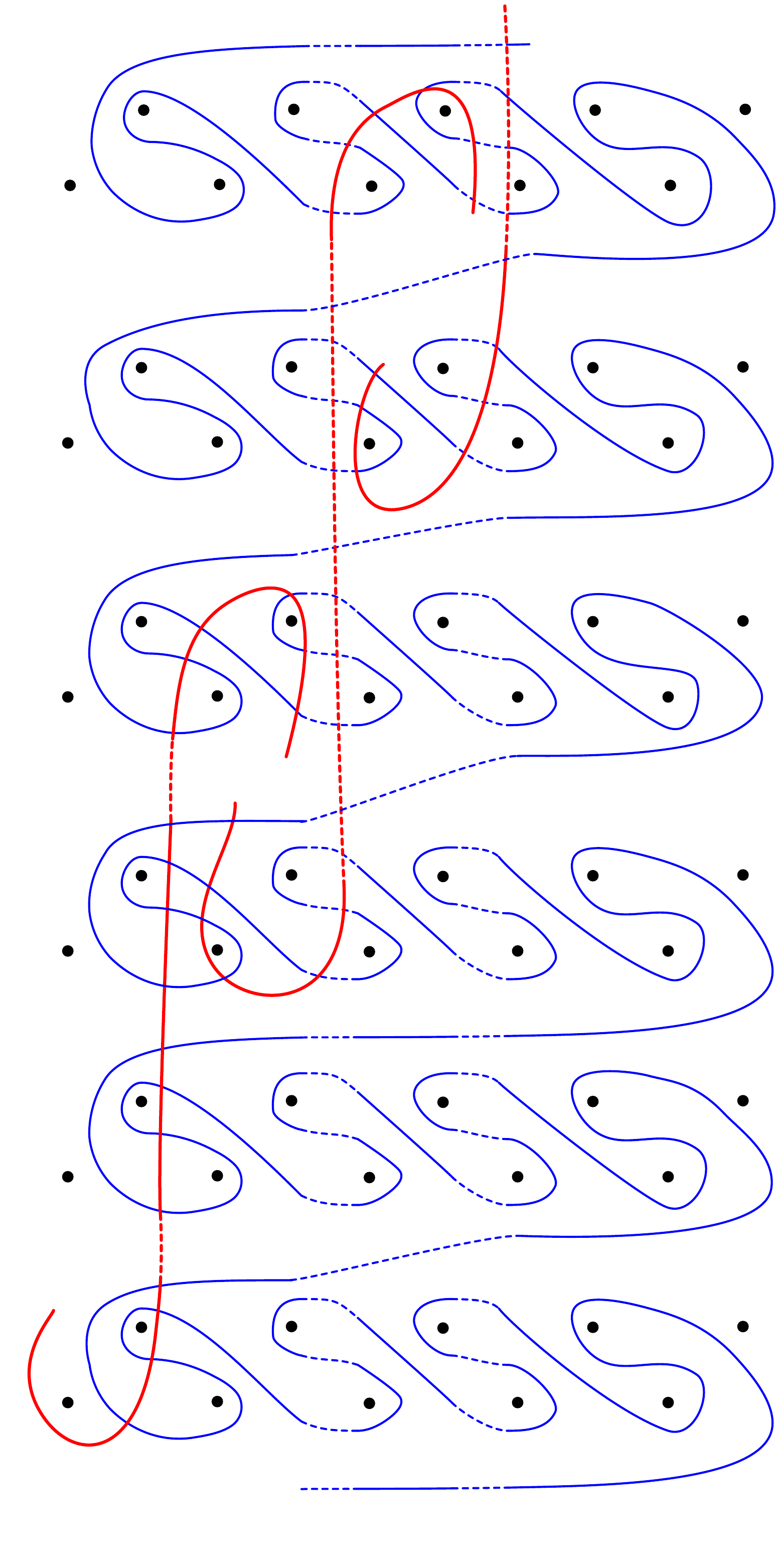}};
\node at (.85,.9) {$\tikzcirc{2pt}$};
\node at (1.05,1.8) {$\tikzcirc{2pt}$};
\node at (1.05,2.3) {$\tikzcirc{2pt}$};
\node at (1.05,3.3) {$\tikzcirc{2pt}$};
\node at (1.07,3.75) {$\tikzcirc{2pt}$};
\node at (1.1,4.75) {$\tikzcirc{2pt}$};
\node at (1.13,5.35) {$\tikzcircle{2pt}$};
\node at (1.2,5.8) {$\tikzcirclee{2pt}$};
\node at (1.12,4.2) {$\tikzcirclee{2pt}$};
\node at (1.05,2.8) {$\tikzcirclee{2pt}$};
\node at (1,1.35) {$\tikzcirclee{2pt}$};
\node at (2.6,8.2) {$\tikzcirc[]{2pt}$};

\node[font=\tiny] at (.85,.7) {$a_{2m+1}$};
\node[font=\tiny] at (1.1,1.9) {$a_{2m}$};
\node[font=\tiny] at (1.2,2.2) {$a_{2i+1}$};
\node[font=\tiny] at (1.2,3.45) {$a_{2i}$};
\node[font=\tiny] at (.6,3.75) {$a_{2i-1}$};
\node[font=\tiny] at (1.35,4.9) {$a_{2i-2}$};
\node[font=\tiny] at (1.35,5.2) {$a_1$};
\node[font=\tiny] at (1.2,6) {$b_1$};
\node[font=\tiny] at (1.5,4.35) {$b_{2i-1}$};
\node[font=\tiny] at (1.5,2.9) {$b_{2i+1}$};
\node[font=\tiny] at (1.5,1.4) {$b_{2m+1}$};

\node[font=\tiny] at (2.6,8.4) {$c$};
\draw[very thin] (5.1,8.2) to (5.1,6.9);
\draw[very thin] (5,8.2) to (5.2,8.2);
\draw[very thin] (5,6.9) to (5.2,6.9);
\draw[very thin] (0,6.9) to (0,6.1);
\draw[very thin] (-.1,6.9) to (0.1,6.9);
\draw[very thin] (-.1,6.1) to (0.1,6.1);
\draw[very thin] (0,6.1) to (0,.9);
\draw[very thin] (-.1,6.1) to (0.1,6.1);
\draw[very thin] (-.1,.9) to (0.1,.9);
\node[font=\tiny, rotate=90] at (-.2,3.1) {$2\tau(K)-n$};
\node[font=\tiny, rotate=90] at (5.3,7.6) {$\tau(K)$};
\node[font=\tiny] at (-1,6.4) {$\dfrac{j-1}{2}|n|-2\tau(K)$};
\end{tikzpicture}\hspace{1in}
    \caption{The pairing diagram when $\tau(K)>0$ $\epsilon(K)=1$ and $n\leq 0<2\tau(K)$ and $j$ odd.}\label{tauposepsilononenneg}
  \end{minipage}\hspace{1in}
\end{figure}
In the case that $n>0$, The intersection points $\{a_k\}$ form a subcomplex of $\CFK_{\F[V]}(S^3,Q^{0,j}_n(U))$ that contains an $\F[V]$-free part, and we see that the cycle $a_1$ generates $\HFhat(S^3)$. Directly from Figure \ref{tauepsilon0npos} we see that the intersection point $b$ satisfies $\partial^h(b)=U^2a_1$ so $\epsilon(Q^{0,j}_n(U))=1$. Furthermore, we have that $A(a_1)=\tau(Q^{0,j}_n(U))=\dfrac{j(j-1)}{2}n$.

\textbf{Case 1 $\epsilon(K)=1$}: In the case $\epsilon(K)=1$, we first distinguish between the cases $\tau(K)$ positive and negative and then distinguish further between various sub-cases depending on the value of $n$ relative to $\tau(K)$.

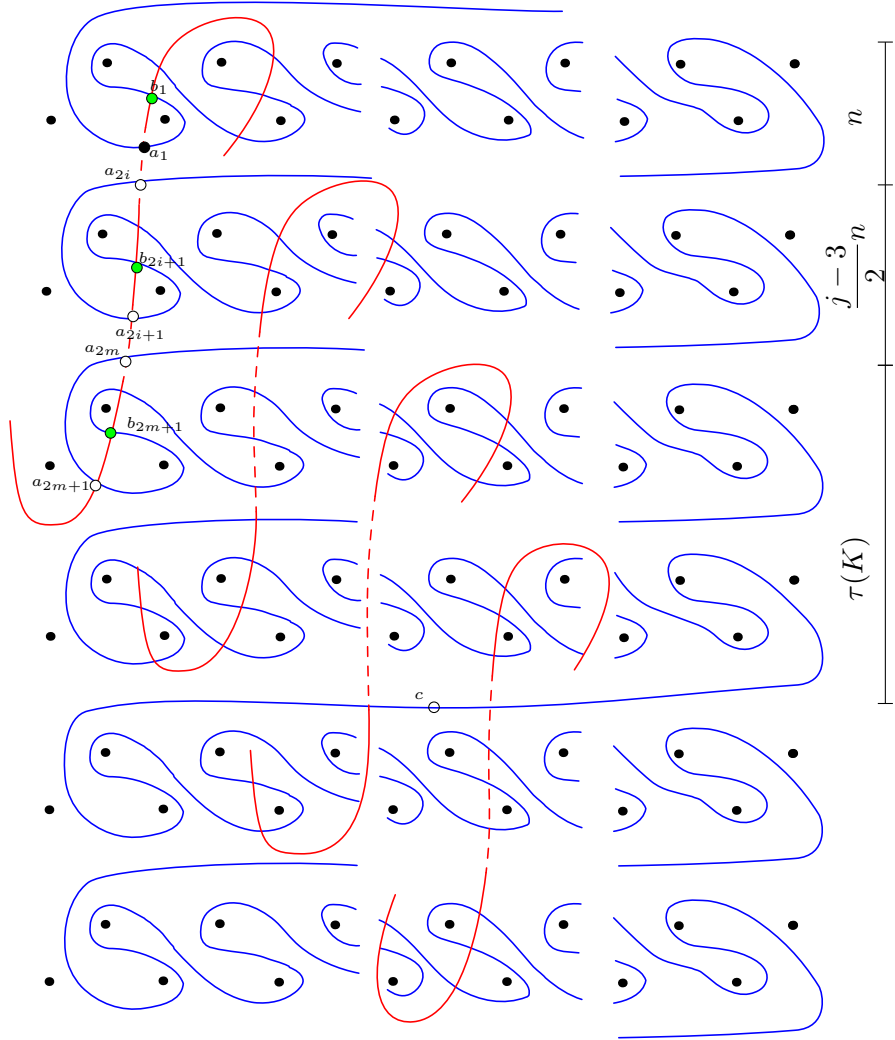
\begin{figure}[!tbp]
  \begin{tikzpicture}
\centering\hspace{.6in}
\draw (12,6.7) to (12,11.2);
\draw (11.9,6.7) to (12.1,6.7);
\draw (11.9,11.2) to (12.1,11.2);
\draw (12,11.2) to (12,13.6);
\draw (11.9,11.2) to (12.1,11.2);
\draw (11.9,13.6) to (12.1,13.6);
\draw (12,15.5) to (12,13.6);
\draw (11.9,15.5) to (12.1,15.5);
\node[anchor=south west,inner sep=0] at (0,2)    {\includegraphics[scale=.3]{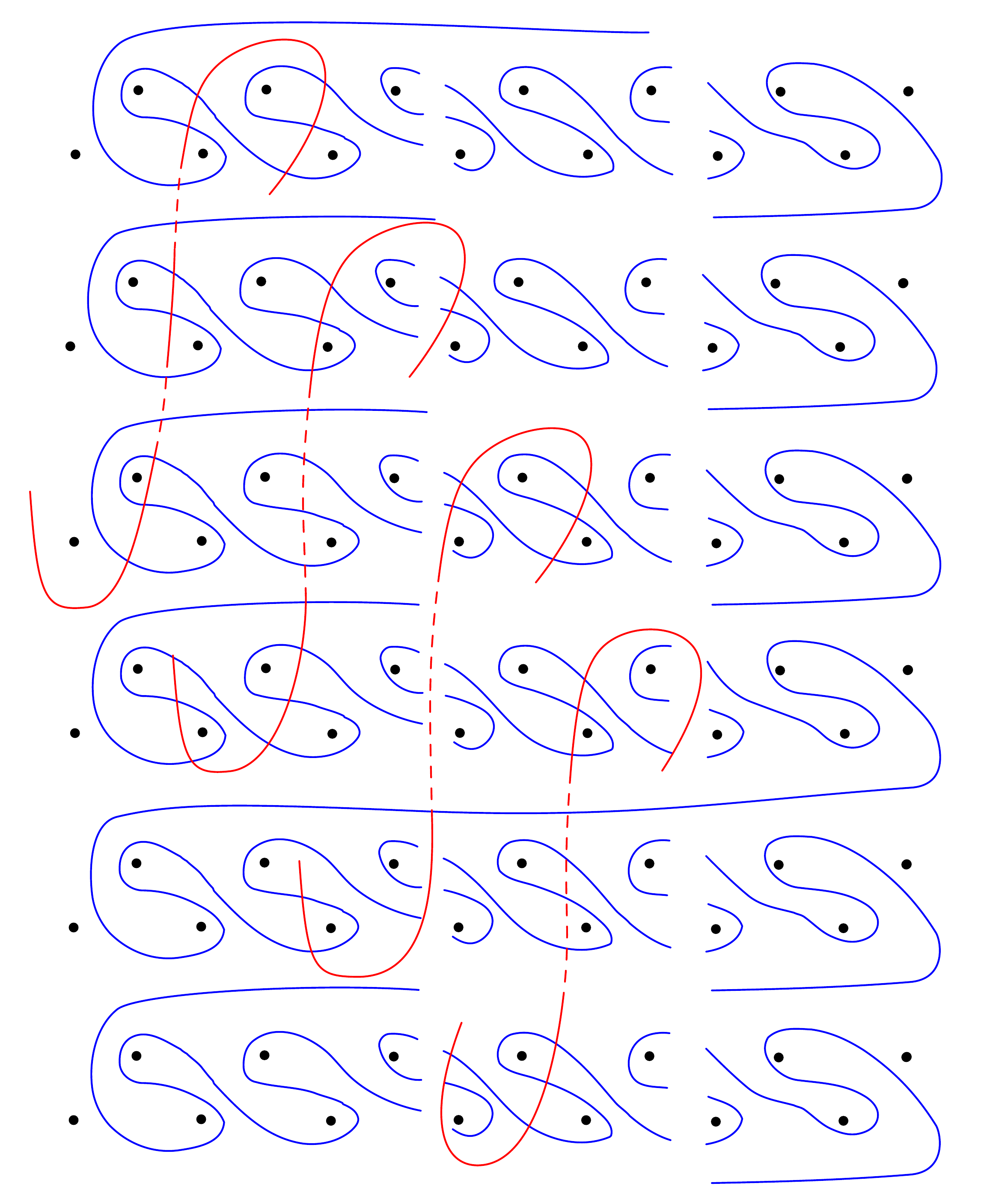}};
\node[rotate=90] at (11.6,8.5) {$\tau(K)$};
\node[rotate=90] at (11.6,12.5) {$\dfrac{j-3}{2}n$};
\node[rotate=90] at (11.6,14.5) {$n$};
\node[font=\tiny] at (6,6.65) {$\tikzcirc[]{2pt}$};
\node[font=\tiny] at (1.5, 9.6) {$\tikzcirc{2pt}$};
\node[font=\tiny] at (1.9, 11.25) {$\tikzcirc{2pt}$};
\node[font=\tiny] at (2, 11.85) {$\tikzcirc{2pt}$};
\node[font=\tiny] at (2.1, 13.6) {$\tikzcirc{2pt}$};
\node[font=\tiny] at (2.15, 14.1) {$\tikzcircle{2pt}$};
\node[font=\tiny] at (2.25, 14.75) {$\tikzcirclee{2pt}$};
\node[font=\tiny] at (2.05, 12.5) {$\tikzcirclee{2pt}$};
\node[font=\tiny] at (1.7, 10.3) {$\tikzcirclee{2pt}$};
\node[font=\tiny] at (5.8,6.8) {$c$};
\node[font=\tiny] at (1.05, 9.6) {$a_{2m+1}$};
\node[font=\tiny] at (1.6, 11.4) {$a_{2m}$};
\node[font=\tiny] at (2.1, 11.6) {$a_{2i+1}$};
\node[font=\tiny] at (1.8, 13.75) {$a_{2i}$};
\node[font=\tiny] at (2.35, 14) {$a_1$};
\node[font=\tiny] at (2.35, 14.9) {$b_1$};
\node[font=\tiny] at (2.4, 12.6) {$b_{2i+1}$};
\node[font=\tiny] at (2.3, 10.4) {$b_{2m+1}$};
\end{tikzpicture}
    \caption{Case $\tau(K)>0$, $\epsilon(K)=1$ and $0\leq n<2\tau(K)$ with $j$ odd.}\label{taupositiveepsilon1nless2taujodd}
\end{figure}

\begin{figure}[!tbp]
  \centering
  \begin{minipage}[b]{0.3\textwidth}
\begin{tikzpicture}\hspace{-.4in}
\node[scale=.6] at (0,0) {\begin{tikzcd}
	{} & \cdots & {a_{2m-2}} \\
	& {a_{2m-3}} && {a_{2m}} \\
	&& {a_{2m-1}} & {a_{2m+1}} & b
	\arrow["V"', from=1-3, to=2-2]
	\arrow["{V^{j+1}}"', from=1-3, to=3-3]
	\arrow["V"', from=2-4, to=3-3]
	\arrow["V", from=2-4, to=3-4]
	\arrow["U"', from=3-5, to=3-4]
\end{tikzcd}};
\end{tikzpicture}
\caption{Subcomplex carrying the cycle that generates $\HFhat(S^3)$ together with horizontal differentials from Figures \ref{taunegepsilonnegnpos} and \ref{tauposepsilonnegngreater2tau}.}\label{subcomplexcasetauposepsilon1nless2taunneg}
  \end{minipage}
  \hspace{1.5in}
  \begin{minipage}[b]{0.3\textwidth}
  \begin{tikzpicture}\space{-.2in}
\node[scale=.6] at (0,0)
{\begin{tikzcd}
	{a_1} & {b_1} \\
	& {a_3} & {b_3} \\
	{a_2} && {} \\
	& {a_4} & \dots & {a_{2m+1}} & {b_{2m+1}} \\
	&& {a_{2m}}
	\arrow["{V^{j+1}}"', from=1-1, to=3-1]
	\arrow["V"', from=2-2, to=3-1]
	\arrow["{V^{j+1}}"', from=2-2, to=4-2]
	\arrow["V"', from=4-4, to=5-3]
	\arrow["U"', from=1-2, to=1-1]
	\arrow["U"', from=2-3, to=2-2]
	\arrow["U"', from=4-5, to=4-4]
\end{tikzcd}};
\end{tikzpicture}
    \caption{Subcomplex carrying cycle that generates $\HFhat(S^3)$ and horizontal differentials from Figures \ref{tauposepsilononenneg}, \ref{taupositiveepsilon1nless2taujodd}, and \ref{taunegepsilon1nless2tau}.}\label{subcomplexcase7}
  \end{minipage}
\end{figure}

\textbf{Case 1.1a $\tau(K)\geq 0$ and $n\geq 2\tau(K)$}:
This case is shown in Figure \ref{taupositiveepsilononenbig2tau}. In that figure, the intersection points labelled $\{a_k\}_{k=1}^{2m+1}$ form a subcomplex that contains the $\F[V]$-free part of $\CFK_{\F[V]}(S^3,Q^{0,j}_n(K))$ and it is easy to see that the intersection point labelled $a_{2m+1}$ generates $\HFhat(S^3)$, obtained by setting $V=1$ in the above sub-complex. Then, the intersection point $a_{2m+1}$ is a distinguished element of some vertically simplified basis of $\CFK^{-}(K)$. Since the intersection point labelled $b$ satisfies $\partial^h(b)=U^2a_{2m+1}$, the cycle $a_{2m+1}$ is a boundary with respect to the horizontal differential, so it follows from \cite[Section 3]{Hom} that $\epsilon(Q^{0,j}_n(K))=1$. It remains to determine $\tau(Q^{0,j}_n(K))=A(a_{2m+1})$. By symmetry of the pairing diagram, we see that the intersection point $c$ satisfies $A(c)=0$, and then by \cite[Lemma 4.1]{Chen} $A(a_{2m+1})=A(a_{2m+1})-A(c)=\ell_{c,a_{2m+1}}\cdot \delta_{w,z}$. Now, to determine the quantity $\ell_{c,a_{2m+1}}\cdot \delta_{w,z}$ we see in Figure \ref{taupositiveepsilononenbig2tau} that the arc $\ell_{c, a_{2m+1}}$ traverses $\tau(K)+\dfrac{j-1}{2}n$ rows up the pairing diagram, and the Alexander grading changes by $j$ for each row we go up in the pairing diagram. Therefore $$A(a_{2m+1})=j\left(\tau(K)+\dfrac{j-1}{2}n\right)=j\tau(K)+\dfrac{j(j-1)}{2}n.$$

\textbf{Case 1.1b $\tau(K)>0$ and $n\leq 0<2\tau(K)$}: This case is shown in Figure \ref{tauposepsilononenneg}. In the pairing diagram, we can see that the subcomplex generated by the intersection points $\{a_k\}_{k=1}^{2m+1}$ together with the vertical differentials carries the $\F[V]$-free part of the $\HFK^{-}(S^3,Q^{0,j}_n(K))$. This subcomplex is shown in Figure \ref{subcomplexcasetauposepsilon1nless2taunneg} together with the horizontal differentials, and that the cycle $\sum_{k=0}^m a_{2k+1}$ survives in $\HFhat(S^3)$. Then this cycle forms the distinguished element of some vertically simplified basis. Further, we see that for each $k$, $\partial^h(b_{2k+1})=Ua_{2k+1}$, so we have $\partial^h(\sum_{k=0}^m b_{2k+1})=U\sum_{k=0}^m a_{2k+1}$. Therefore $\epsilon(Q^{0,j}_n(K))=1$. Now, it remains to determine $A(a_1)=\tau(Q^{0,j}_n(K))$. By symmetry, $A(c)=0$ and $$A(a_1)=A(a_1)-A(c)=\ell_{a_1,c}\cdot \delta_{w,z}=-j\left(\tau(K)+\dfrac{j-1}{2}|n|-2\tau(K)\right)+1,$$ since the Alexander grading changed by $-j$ for each row we go down in the pairing diagram. Simplifying, we see that $$\tau(Q^{0,j}_n(K))=j\tau(K)+\dfrac{j(j-1)}{2}n+1.$$

\textbf{Case 1.1c $\tau(K)>0$ and $0\leq n<2\tau(K)$}: This case is shown in Figure \ref{taupositiveepsilon1nless2taujodd}. The analysis here is exactly as in the previous case. The subcomplex $\{a_k\}$ carries the $\F[V]$-free part of the homology $\HFK^{-}(S^3,Q^{0,j}_n(K))$, and the cycle $\sum_{k=0}^m a_{2k+1}$ survives in $\HFhat(S^3)$, so can be taken to be the distinguished elements of a vertically simplified basis. Further, we have $\partial^h(\sum_{k=0}^m b_{2k+1})=U\sum_{k=0}^m a_{2k+1}$, so just like in the previous case it follows that $\epsilon(Q^{0,j}_n(K))=1$ It remains to determine $A(a_1)$: 
Counting the number of rows between $a_1$ and $c$ in the pairing diagram, we find that $$\tau(Q^{0,j}_n(K))=A(a_1)-A(c)=j\left(\tau(K)+\dfrac{j-1}{2}n\right)+1=j\tau(K)+\dfrac{j(j-1)}{2}n+1.$$

This ends the analysis of the case $\epsilon(K)=1$ and $\tau(K)$ non-negative. Next, we move on to the case $\epsilon(K)=1$ and $\tau(K)$ non-positive. 

\textbf{Case 1.2a $\tau(K)\leq 0$, $\epsilon(K)=1$ and $n< 2\tau(K)$}

In this case, the pairing diagram is shown in Figure \ref{taunegepsilon1nless2tau}. The intersection points labelled $\{a_k\}_{k=1}^{2m+1}$ generate the free part of $\CFK_{\F[V]}(Q^{0,j}_n(K))$, the cycle $\sum_{k=0}^m a_{2k+1}$ is the cycle that survives in $\HFhat(S^3)$, and $\tau(Q^{0,j}_n(K))=A(a_1)$. The intersection points $\{b_{2k+1}\}_{k=0}^m$ satisfy $\partial^h(\sum_{k=0}^m b_{2k+1})=U\sum_{k=0}^m a_{2k+1}$, so $\epsilon(Q^{0,j}_n(K))=1$. Exactly in the previous cases, we find that

$$\tau(Q^{0,j}_n(K))=A(a_1)=-j\left(-\tau(K)+\dfrac{j-1}{2}|n|\right)+1=j\tau(K)+\dfrac{j(j-1)}{2}n+1.$$

\begin{figure}[!tbp]
  \centering
  \begin{minipage}[b]{0.3\textwidth}
  \begin{tikzpicture}
  \hspace{-.5in}
\node[anchor=south west,inner sep=0] at (0,0)    {\includegraphics[width=1.2\textwidth]{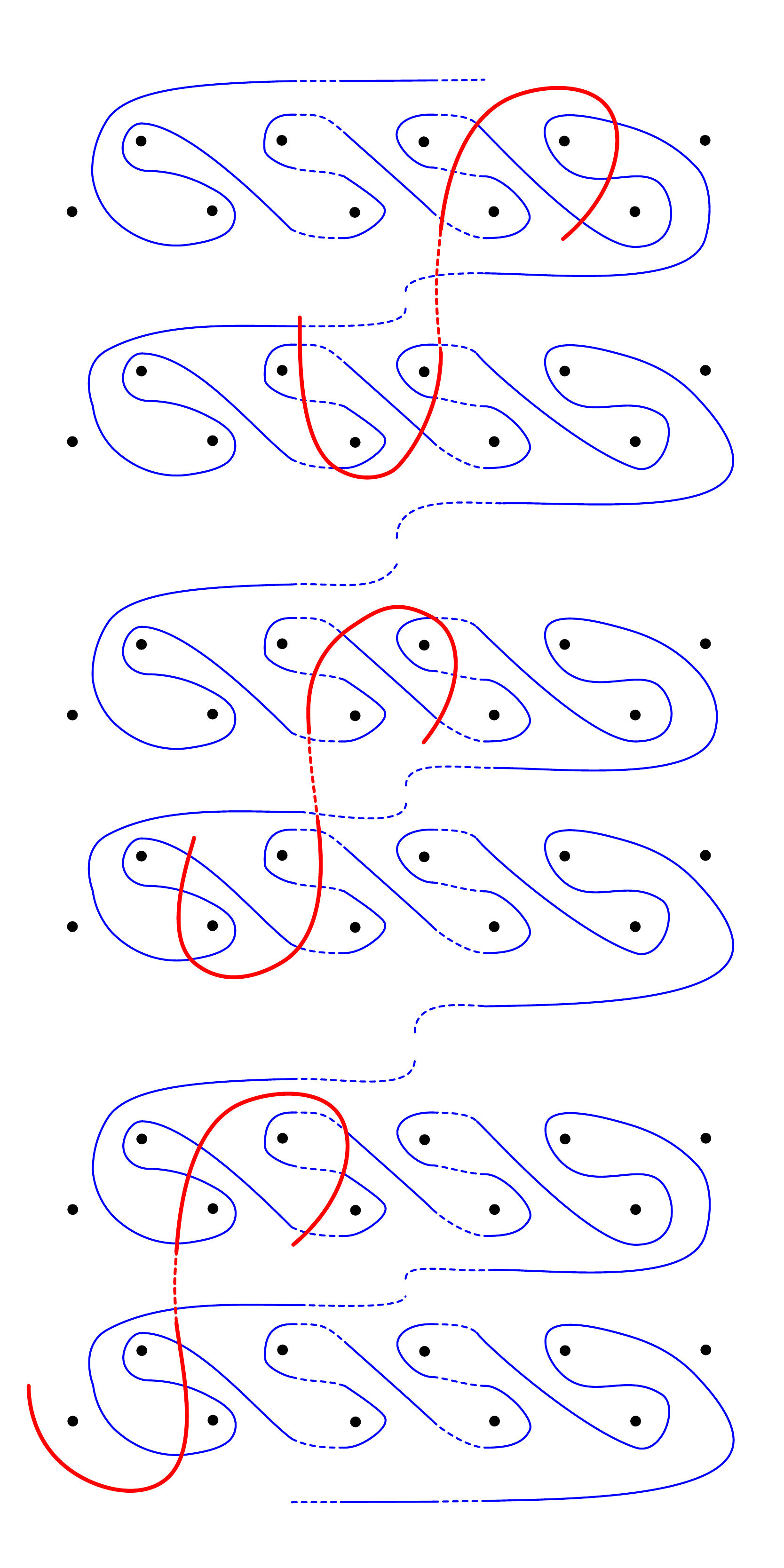}};
\draw[very thin] (0,.7) to (0,3.1);
\draw[very thin] (-.1,.7) to (.1,.7);
\draw[very thin] (-.1,3.1) to (.1,3.1);

\draw[very thin] (5,3.1) to (5,3.9);
\draw[very thin] (4.9,3.1) to (5.1,3.1);
\draw[very thin] (4.9,3.9) to (5.1,3.9);
\draw[very thin] (0,3.1) to (0,6.05);
\draw[very thin] (-.1,3.1) to (.1,3.1);
\draw[very thin] (-.1,6.05) to (.1,6.05);
\draw[very thin] (5,6.05) to (5,6.5);
\draw[very thin] (4.9,6.05) to (5.1,6.05);
\draw[very thin] (4.9,6.5) to (5.1,6.5);

\node[font=\tiny] at (1.2,.75) {$\tikzcirc{2pt}$};
\node[font=\tiny] at (1.17,1.65) {$\tikzcirc{2pt}$};
\node[font=\tiny] at (1.2,2.1) {$\tikzcircle{2pt}$};
\node[font=\tiny] at (1.2,2.5) {$\tikzcirclee{2pt}$};
\node[font=\tiny] at (1.2,1.15) {$\tikzcirclee{2pt}$};

\node[font=\tiny] at (1.65,.65) {$a_{2m+1}$};
\node[font=\tiny] at (.9,1.75) {$a_{2m}$};
\node[font=\tiny] at (1.4,1.9) {$a_1$};
\node[font=\tiny] at (1.2,2.7) {$b_1$};
\node[font=\tiny] at (1.6,1.2) {$b_{2m+1}$};
\node[font=\tiny] at (2.5,6.7) {$c$};

\node[font=\tiny] at (2.5,6.5) {$\tikzcirc[]{2pt}$};

\node[font=\tiny,rotate=90] at (5.2,3.5) {$2|\tau(K)|$};
\node[font=\tiny] at (5.4,6.3) {$|\tau(K)|$};

\node[font=\tiny,rotate=90] at (-.2,2) {$2\tau(K)-n$};
\node[font=\tiny,rotate=90] at (-.4,4.5) {$\dfrac{j-1}{2}|n|$};
\end{tikzpicture}\vspace{1in}
    \caption{$\tau(K)\leq 0$ $\epsilon(K)=1$ and $n\leq2\tau(K)$}\label{taunegepsilon1nless2tau}
  \end{minipage}
  \hspace{2in}
  \begin{minipage}[b]{0.3\textwidth}
  \begin{tikzpicture}
  \hspace{-.5in}
\node[anchor=south west,inner sep=0] at (0,0)    {\includegraphics[width=1.2\textwidth]{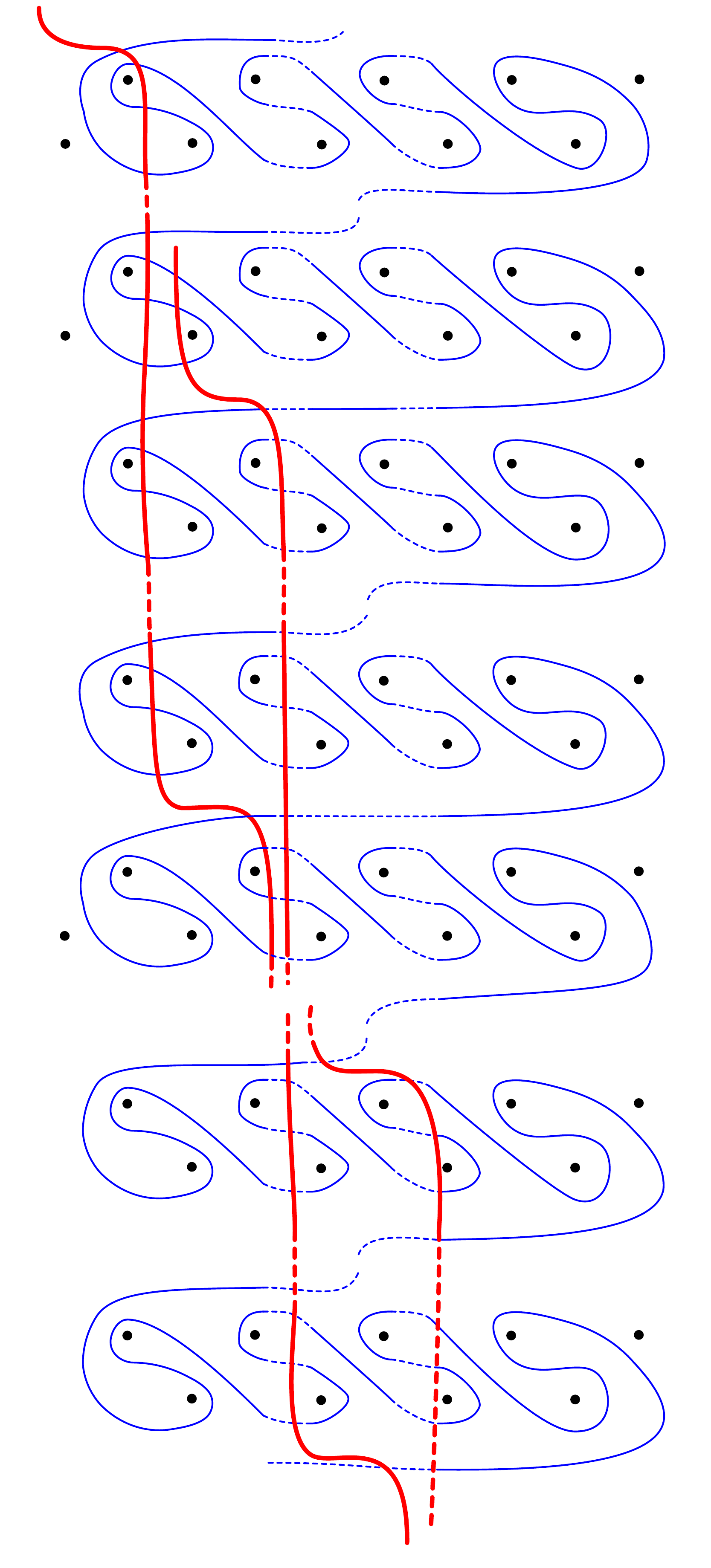}};
\node at (.9,10.75) {$\tikzcirc{2pt}$};
\node at (1,9.9) {$\tikzcirc{2pt}$};
\node at (1,9.45) {$\tikzcirc{2pt}$};
\node at (1,8.55) {$\tikzcirc{2pt}$};
\node at (1,8.1) {$\tikzcirc{2pt}$};
\node at (1.05,7.16) {$\tikzcirc{2pt}$};
\node at (1.1,6.55) {$\tikzcirc{2pt}$};
\node at (1.16,5.65) {$\tikzcirc{2pt}$};
\node at (1.75,5.3) {$\tikzcircle{2pt}$};
\node at (1.9,5.1) {$\tikzcirclee{2pt}$};
\node at (2.5,2.2) {$\tikzcirc[]{2pt}$};
\node at (.9,10.75) {$\tikzcirc{2pt}$};
\node at (1,9.9) {$\tikzcirc{2pt}$};
\node at (1,9.45) {$\tikzcirc{2pt}$};
\node at (1,8.55) {$\tikzcirc{2pt}$};
\node at (1,8.1) {$\tikzcirc[]{2pt}$};
\node at (1,8.1) {$\tikzcirc{2pt}$};

\node[font=\tiny] at (.9,11) {$a_{2m+1}$};
\node[font=\tiny] at (.9,9.8) {$a_{2m}$};
\node[font=\tiny] at (.65,9.55) {$a_{2i+1}$};
\node[font=\tiny] at (.8,8.45) {$a_{2i}$};
\node[font=\tiny] at (.6,8.1) {$a_{2i-1}$};
\node[font=\tiny] at (.8,7.06) {$a_{2i-2}$};
\node[font=\tiny] at (.9,6.65) {$a_3$};
\node[font=\tiny] at (1,5.5) {$a_2$};
\node[font=\tiny] at (1.75,5.45) {$a_1$};
\node[font=\tiny] at (2,5.2) {$b$};
\node[font=\tiny, rotate=90] at (-.1,2.9) {$|\tau(K)|$};
\node[font=\tiny, rotate=90] at (-.2,4.5) {$\dfrac{j-1}{2}n-2|\tau(K)|$};
\node[font=\tiny, rotate=90] at (-.2, 9.6) {$n$};
\node[font=\tiny, rotate=90] at (5.2, 6.5) {$2|\tau(K)|$};

\draw[very thin] (.1,2) to (.1,3.6); 
\draw[very thin] (0,2) to (.2,2); 
\draw[very thin] (0,3.6) to (.2,3.6); 
\draw[very thin] (.1,3.6) to (.1,5.3); 
\draw[very thin] (0,5.3) to (.2,5.3);

\draw[very thin] (5,5.3) to (5,8.3); 
\draw[very thin] (4.9,5.3) to (5.1,5.3); 
\draw[very thin] (4.9,8.3) to (5.1,8.3);
\draw[very thin] (.1,8.3) to (.1,10.9); 
\draw[very thin] (0,8.3) to (.2,8.3); 
\draw[very thin] (0,10.9) to (.2,10.9); 
\end{tikzpicture}\vspace{.8in}
    \caption{$\tau(K)\leq 0$ $\epsilon(K)=1$ and $n\geq 0$}\label{taunegepsilononenpos}
  \end{minipage}
\end{figure}

\textbf{Case 1.2b $\tau(K)\leq 0$, $\epsilon(K)=1$ and $n\geq 0\geq 2\tau(K)$}
In this case, the pairing diagram has the form shown in Figure \ref{taunegepsilononenpos}. In this case the $\F[V]$ free part of the homology is generated by the intersection points $\{a_k\}_{k=1}^{2m+1}$. Further, the intersection point $a_1$ generated $\HFhat(S^3)$. Just as above, the intersection point $b$ satisfies $\partial^h(b)=U^2a_{1}$ and hence $\epsilon(Q^{0,j}_n(K))=1$. Furthermore $A(a_1)=\tau(Q^{0,j}_n(K))$. Inspecting the pairing diagram we find that 
$$A(a_{1})-A(c)=j\left(\tau(K)+\dfrac{j-1}{2}n\right)=j\tau(K)+\dfrac{j(j-1)}{2}n.$$

\begin{figure}[!tbp]
  \centering
  \begin{minipage}[b]{0.3\textwidth}
\begin{tikzpicture}\hspace{-.4in}
\node[scale=.6] at (0,0) {\begin{tikzcd}
	\cdots & {a_4} \\
	&& {a_2} \\
	& {a_3} \\
	&& {a_1} && b
	\arrow["{V^{j+1}}"', from=1-2, to=3-2]
	\arrow["V"', from=2-3, to=3-2]
	\arrow["{V^{j+1}}", from=2-3, to=4-3]
	\arrow["{U^2}"', from=4-5, to=4-3]
\end{tikzcd}};
\end{tikzpicture}
\caption{Subcomplex carrying the cycle that generates $\HFhat(S^3)$ corresponding to the cases in Figures \ref{taupositiveepsilononenbig2tau}, \ref{taunegepsilononenpos}, and \ref{taunegepsilon1nneg}}\label{subcomplexcase10}
  \end{minipage}
  \hspace{1.5in}
  \begin{minipage}[b]{0.3\textwidth}
  \begin{tikzpicture}\space{-.2in}
\node[scale=.6] at (0,0)
{\begin{tikzcd}
	{a_1} & {b_1} \\
	& {a_3} & {b_3} \\
	{a_2} && \cdots \\
	& {a_4}
	\arrow["{V^j}"', from=1-1, to=3-1]
	\arrow["U"', from=1-2, to=1-1]
	\arrow["V"', from=2-2, to=3-1]
	\arrow["{V^{j+1}}", from=2-2, to=4-2]
	\arrow["U"', from=2-3, to=2-2]
\end{tikzcd}};
\end{tikzpicture}
    \caption{Subcomplex carrying cycle that generates $\HFhat(S^3)$ and horizontal differentials from Figures \ref{tauposepsilonminus1nneg} and \ref{taunegepsilonminus1nleq2tau}}\label{subcomplex4}
  \end{minipage}
\end{figure}

\textbf{Case 1.2c $\tau(K)\leq 0$, $\epsilon(K)=1$ and $0\geq n\geq 2\tau(K)$}
The pairing diagram for this case is shown in Figure \ref{taunegepsilon1nneg}. The intersection points $\{a_k\}_{k=0}^{2m+1}$ generate the free part of the homology, and the intersection point $a_1$ generates $\HFhat(S^3)$. In the pairing diagram, the intersection point labelled $b$ satisfied $\partial^h(b)=U^2a_1$, so $\epsilon(Q^{0,j}_n(K))=1$. Further, we compute $$\tau(Q^{0,j}_n(K))=A(a_1)=-j\left(|\tau(K)|+\dfrac{j-1}{2}|n|\right)=j\tau(K)+\dfrac{j(j-1)}{2}n.$$

That finishes the cases where $\epsilon(K)=1$. \\
\begin{figure}[!tbp]
  \centering
    \begin{minipage}[b]{0.3\textwidth}
  \begin{tikzpicture}
\hspace{-.5in}

\node[anchor=south west,inner sep=0] at (0,0)    {\includegraphics[width=1.3\textwidth]{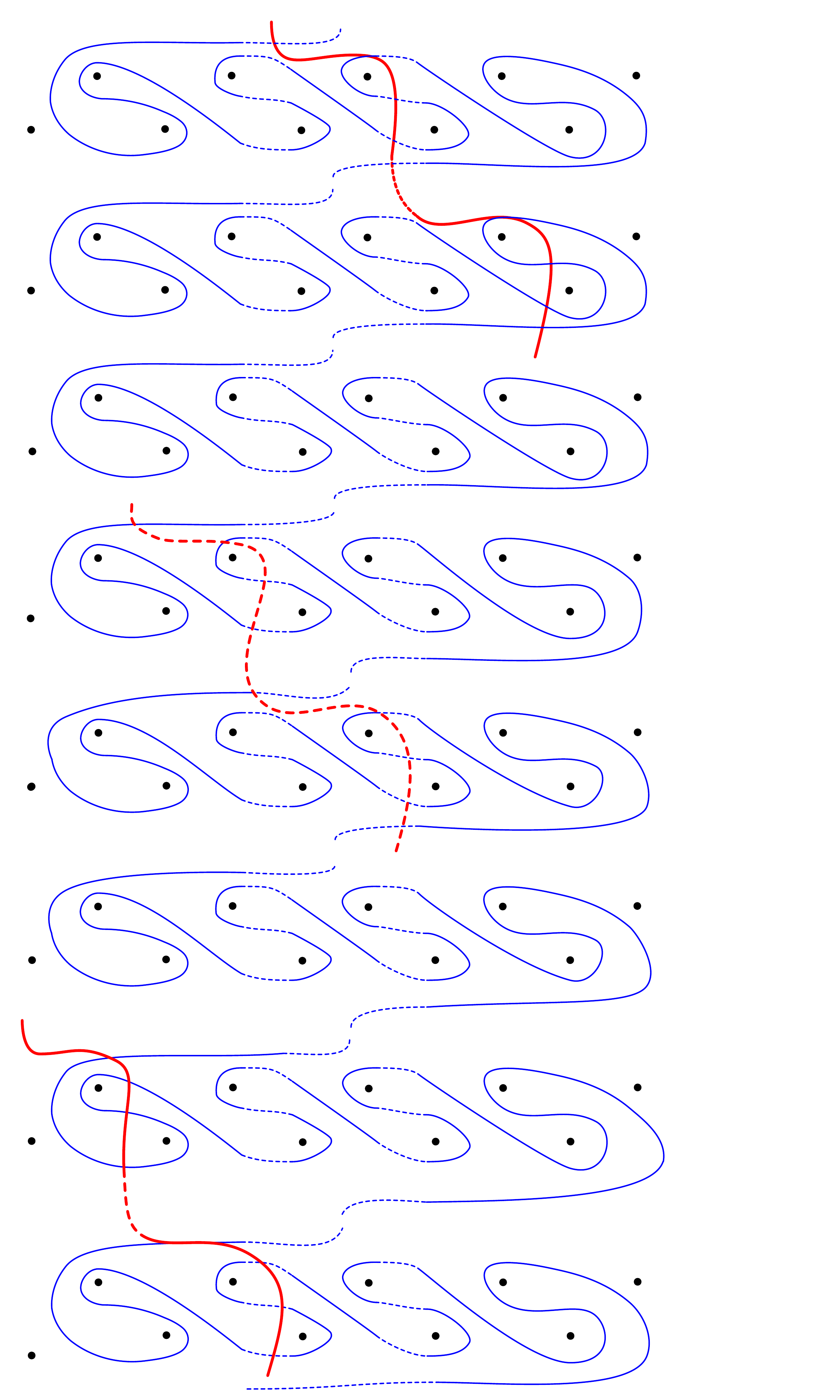}};
\node[font=\tiny] at (1,1) {$\tikzcircle{2pt}$};
\node[font=\tiny] at (.75,1.5) {$\tikzcirc{2pt}$};
\node[font=\tiny] at (.75,2.2) {$\tikzcirc{2pt}$};
\node[font=\tiny] at (1.7,.85) {$\tikzcirclee{2pt}$};
\node[font=\tiny] at (2.2, 6.85) {$\tikzcirc[]{2pt}$};
\draw (0,7.9) to (0,5.2);
\draw (0,1.2) to (0,5.2);
\draw (-.1,1.2) to (.1,1.2);
\draw (-.1,7.9) to (.1,7.9);
\draw (-.1,5.2) to (.1,5.2);

\node[font=\tiny, rotate=90] at (-.4,3) {$\dfrac{j-1}{2}|n|$};
\node[font=\tiny, rotate=90] at (-.3,6.5) {$|\tau(K)|$};
\node[font=\tiny] at (1.15,1.15) {$a_1$};
\node[font=\tiny] at (.9,1.35) {$a_{2m}$};
\node[font=\tiny] at (.75,2.4) {$a_{2m+1}$};
\node[font=\tiny] at (1.8,1) {$b$};
\node[font=\tiny] at (2.2, 7.1) {$c$};
\end{tikzpicture}\vspace{.5in}
    \caption{The pairing diagram when $\tau(K)<0$ $\epsilon(K)=1$ and $0>n>2\tau(K)$}\label{taunegepsilon1nneg}
  \end{minipage}
  \hspace{1in}
  \begin{minipage}[b]{0.3\textwidth}
  \begin{tikzpicture}
  \hspace{-.5in}
\node[anchor=south west,inner sep=0] at (0,0)    {\includegraphics[width=1.2\textwidth]{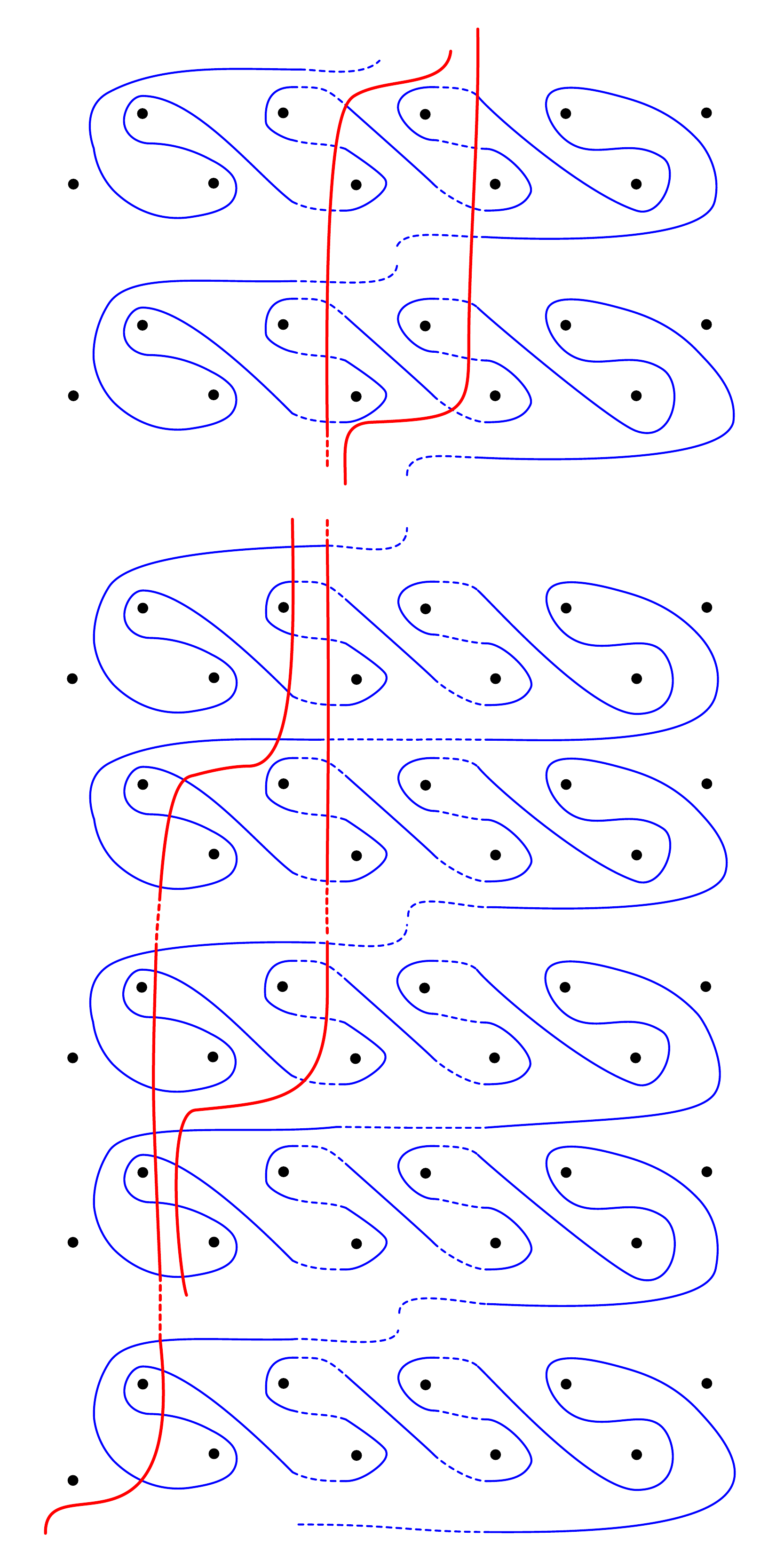}};
\node[font=\tiny] at (1.1,4.8) {$\tikzcirclee{2pt}$};
\node[font=\tiny] at (1,3.55) {$\tikzcirclee{2pt}$};
\node[font=\tiny] at (1,2.3) {$\tikzcirclee{2pt}$};
\node[font=\tiny] at (1.05,1) {$\tikzcirclee{2pt}$};

\node[font=\tiny] at (1.3,4.85) {$b_3$};
\node[font=\tiny] at (1.4,3.6) {$b_{2i-1}$};
\node[font=\tiny] at (1.4,2.4) {$b_{2i+1}$};
\node[font=\tiny] at (1.5,1) {$b_{2m+1}$};

\node[font=\tiny] at (.9,.6) {$\tikzcirc{2pt}$};
\node[font=\tiny] at (1,1.45) {$\tikzcirc{2pt}$};
\node[font=\tiny] at (1,1.9) {$\tikzcirc{2pt}$};
\node[font=\tiny] at (1,2.8) {$\tikzcirc{2pt}$};
\node[font=\tiny] at (1,3.1) {$\tikzcirc{2pt}$};
\node[font=\tiny] at (1,4) {$\tikzcirc{2pt}$};
\node[font=\tiny] at (1,4.4) {$\tikzcirc{2pt}$};
\node[font=\tiny] at (1.8,5.3) {$\tikzcirc{2pt}$};
\node[font=\tiny] at (1.9,5.6) {$\tikzcircle{2pt}$};
\node[font=\tiny] at (1.9,6) {$\tikzcirclee{2pt}$};
\node[font=\tiny] at (.9,.6) {$\tikzcirc{2pt}$};
\node[font=\tiny] at (.9,.6) {$\tikzcirc{2pt}$};
\node[font=\tiny] at (.9,.6) {$\tikzcirc{2pt}$};
\node[font=\tiny] at (.9,.6) {$\tikzcirc{2pt}$};
\node[font=\tiny] at (.9,.6) {$\tikzcirc{2pt}$};
\node[font=\tiny] at (2.5,8.4) {$\tikzcirc{2pt}$};
\draw[very thin] (0,7.1) to (0,8.4);
\draw[very thin] (-.1,7.1) to (.1,7.1);
\draw[very thin] (-.1,8.4) to (.1,8.4);
\draw[very thin] (0,7.1) to (0,5.4);
\draw[very thin] (-.1,7.1) to (.1,7.1);
\draw[very thin] (-.1,5.4) to (.1,5.4);
\draw[very thin] (5.1,5.4) to (5.1,6.4);
\draw[very thin] (5,5.4) to (5.2,5.4);
\draw[very thin] (5,6.4) to (5.2,6.4);

\node[font=\tiny, rotate=90] at (-.2,7.9) {$\tau(K)$};
\node[font=\tiny, rotate=90] at (-.4,6.2) {$\dfrac{j-1}{2}|n|-2\tau(K)$};
\node[font=\tiny, rotate=90] at (5.25,5.7) {$1$};
\node[font=\tiny] at (.9,.4) {$a_{2m+1}$};
\node[font=\tiny] at (.7,1.5) {$a_{2m}$};
\node[font=\tiny] at (.7,1.8) {$a_{2i+1}$};
\node[font=\tiny] at (.6,2.8) {$a_{2i}$};
\node[font=\tiny] at (.6,3.1) {$a_{2i-1}$};
\node[font=\tiny] at (.6,4) {$a_{2i-2}$};
\node[font=\tiny] at (.75,4.4) {$a_3$};
\node[font=\tiny] at (1.6,5.4) {$a_2$};
\node[font=\tiny] at (1.7,5.6) {$a_1$};
\node[font=\tiny] at (2,5.85) {$b_1$};
\node[font=\tiny] at (2.5,8.55) {$c$};
\end{tikzpicture}\vspace{.4in}
    \caption{The pairing diagram when $\tau(K)>0$ $\epsilon(K)=-1$ and $n<0<2\tau(K)$ and $j$ odd}\label{tauposepsilonminus1nneg}
  \end{minipage}
\end{figure}

\textbf{Case 2 $\epsilon(K)=-1$}: As in the case $\epsilon(K)=1$, we distinguish between various sub-cases depending on the sign of $\tau(K)$ and the value of $n$ relative to $2\tau(K)$\\
\textbf{Case 2.1a $\tau(K)\geq 0$, $\epsilon(K)=-1$ and $n\leq 0<2\tau(K)$}

This case is shown in Figure \ref{tauposepsilonminus1nneg}. The intersection points $\{a_k\}_{k=1}^{2m+1}$ generate a subcomplex with respect to the vertical differentials and carries the $\F[V]$-free part of the homology of $\CFK_{\F[V]}(S^3,Q^{0,j}_n(K))$. Setting $V=1$ we see that the cycle $\sum_{k=0}^m a_{2k+1}$ generates $\HFhat(S^3)$. The intersection points $\{b_{2k+1}\}_{k=0}^m$ satisfy $\partial^h(\sum_{k=0}^m b_{2k+1})=U(\sum_{k=0}^m a_{2k+1})$, so $\epsilon(Q^{0,j}_n(K))=1$. We determine $\tau(Q^{0,j}_n(K))=A(a_1)$ from the pairing diagram and find

$$A(a_1)=-j\left(\tau(K)+\dfrac{j-1}{2}|n|-2\tau(K)-1\right)=j(\tau(K)+1)+\dfrac{j(j-1)}{2}n.$$

Note: The case $\tau(K)>0$, $\epsilon(K)=-1$ and $0<n<2\tau(K)$ is similar and left to the reader. 

\textbf{Case 2.1b $\tau(K)\geq 0$, $\epsilon(K)=-1$ and $n \geq 2\tau(K)$}

This case is shown in Figure \ref{tauposepsilonnegngreater2tau}. In this case subcomplex generated by the intersection points $\{a_k\}$ generate the $\F[V]$-free part of the homology. We see that the intersection point labelled $a_{1}$ generates $\HFhat(S^3)$ and that $\partial^h(b)=Ua_{1}$. Therefore $\epsilon(Q^{0,j}_n(K))=1$. Now, it is easy to see from the pairing diagram that $A(a_1)=A(a_3)$, and $$\tau(Q^{0,j}_n(K))=A(a_{1})=j\left(\tau(K)+\dfrac{j-1}{2}n+1\right)=j(\tau(K)+1)+\dfrac{j(j-1)}{2}n.$$

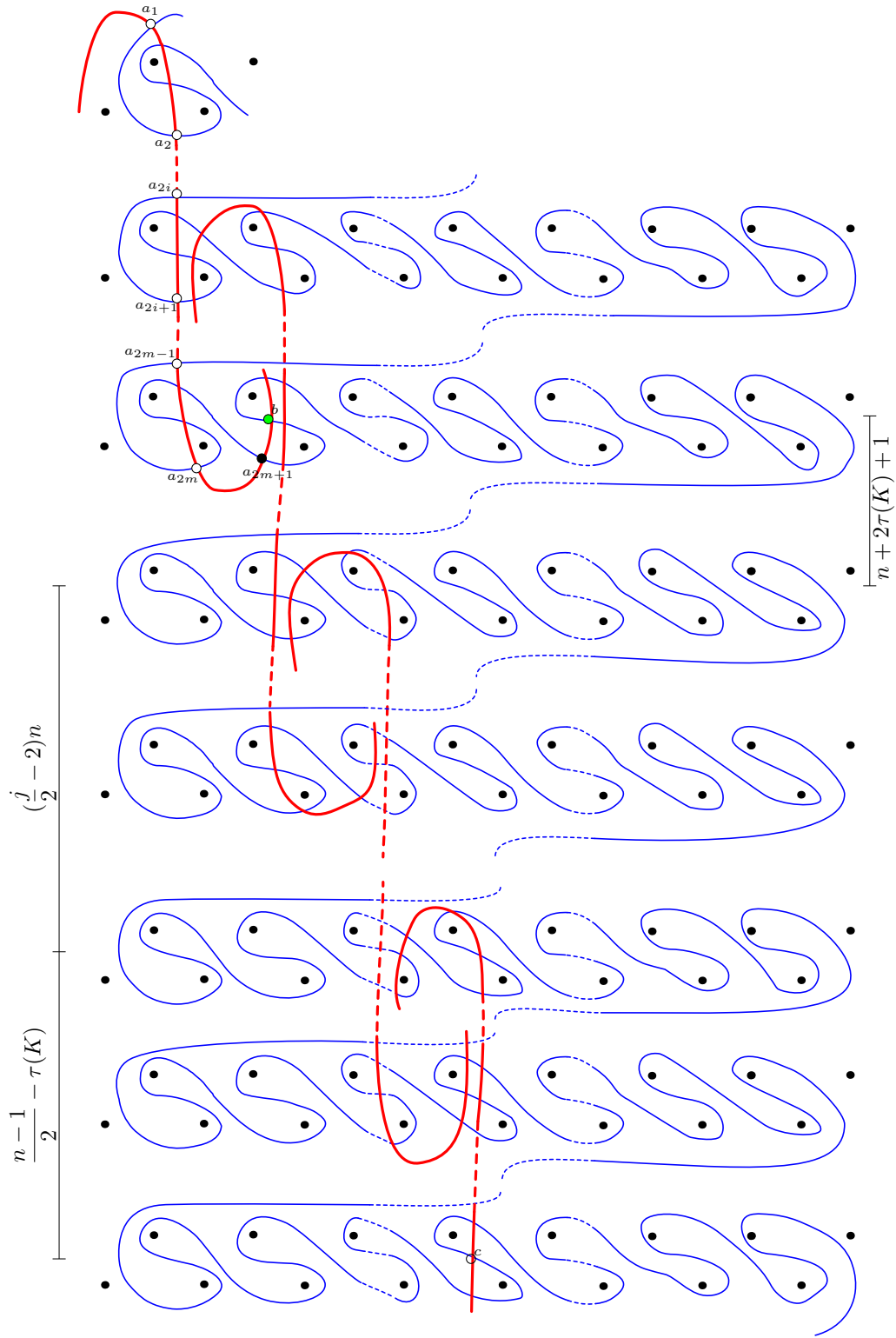
\begin{figure}[!tbp]\vspace{-1in}
  \begin{tikzpicture}\vspace{-2in}
\centering
\node[anchor=south west,inner sep=0] at (0,0)    {\includegraphics[scale=.3]{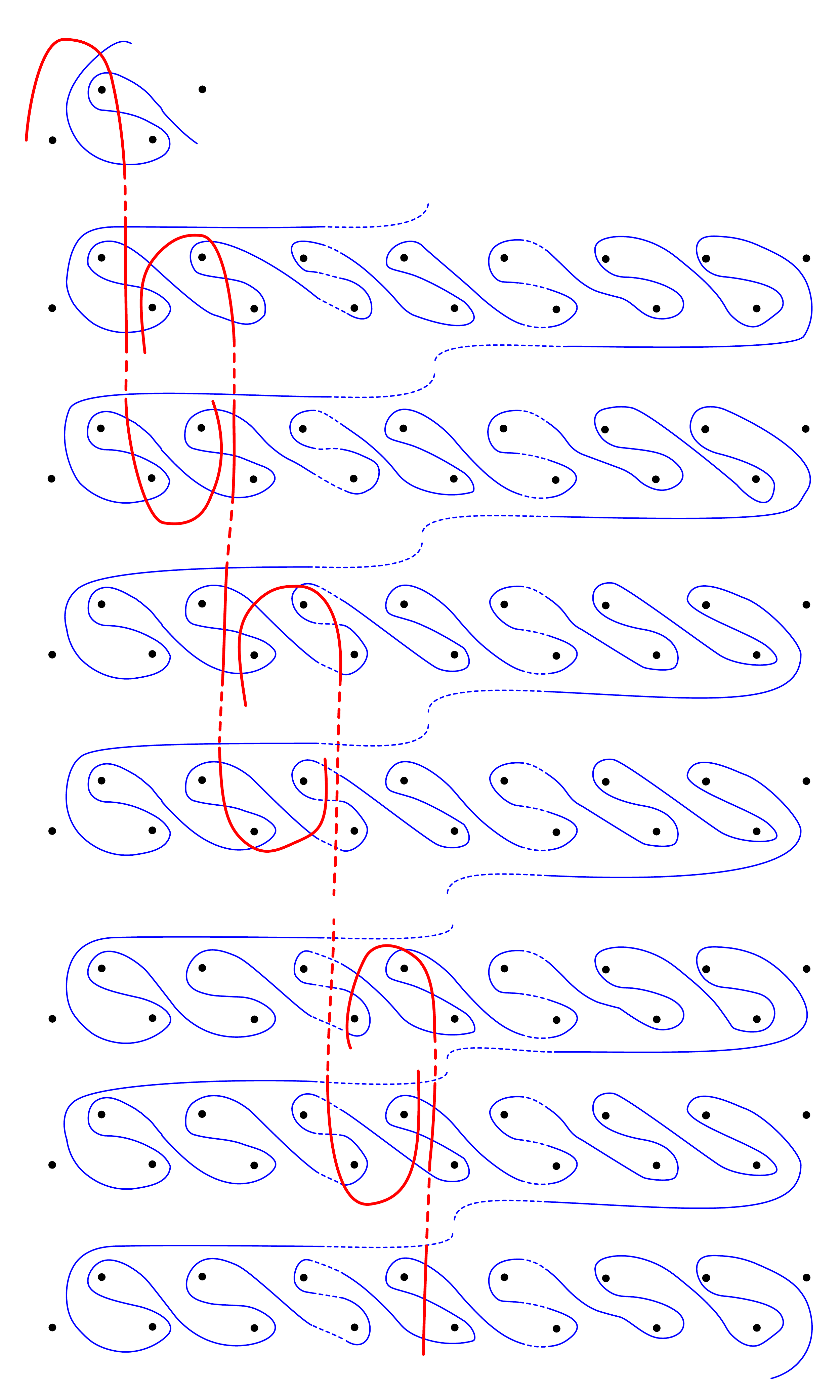}};
\node[font=\tiny] at (6.4,1.5) {$\tikzcirc[]{2pt}$};
\node[font=\tiny] at (2.2,13.6) {$\tikzcirc{2pt}$};
\node[font=\tiny] at (1.9,15.2) {$\tikzcirc{2pt}$};
\node[font=\tiny] at (1.9,16.2) {$\tikzcirc{2pt}$};
\node[font=\tiny] at (1.9,17.8) {$\tikzcirc{2pt}$};
\node[font=\tiny] at (1.9,18.7) {$\tikzcirc{2pt}$};
\node[font=\tiny] at (1.5,20.4) {$\tikzcirc{2pt}$};
\node[font=\tiny] at (3.2,13.75) {$\tikzcircle{2pt}$};
\node[font=\tiny] at (3.3,14.35) {$\tikzcirclee{2pt}$};
\node[font=\small, rotate=90] at (-.3,4) {$\dfrac{n-1}{2}-\tau(K)$};
\node[font=\small, rotate=90] at (-.3,9) {$(\dfrac{j}{2}-2)n$};
\node[font=\small, rotate=90] at (12.7,13) {$n+2\tau(K)+1$};

\draw[very thin] (.1,1.5) to (.1,6.2);
\draw[very thin] (0,1.5) to (.2,1.5);
\draw[very thin] (0,6.2) to (.2,6.2);
\draw[very thin] (.1,6.2) to (.1,11.8);
\draw[very thin] (0,11.8) to (.2,11.8);
\draw[very thin] (12.5,11.8) to (12.5,14.4);
\draw[very thin] (12.4,11.8) to (12.6,11.8);
\draw[very thin] (12.4,14.4) to (12.6,14.4);
\node[font=\tiny] at (6.5,1.6) {$c$};
\node[font=\tiny] at (2,13.45) {$a_{2m}$};
\node[font=\tiny] at (1.5,15.35) {$a_{2m-1}$};
\node[font=\tiny] at (1.6,16.05) {$a_{2i+1}$};
\node[font=\tiny] at (1.65,17.9) {$a_{2i}$};
\node[font=\tiny] at (1.7,18.55) {$a_2$};
\node[font=\tiny] at (1.5,20.6) {$a_1$};
\node[font=\tiny] at (3.3,13.5) {$a_{2m+1}$};
\node[font=\tiny] at (3.4,14.5) {$b$};
\vspace{-1in}\end{tikzpicture}
    \caption{Case $\tau(K)<0$, $\epsilon(K)=-1$ and $n>0>2\tau(K)$ with $j$ even}\label{taunegepsilonnegnpos}
\end{figure}

\textbf{Case 2.2a $\tau(K)\leq 0$, $\epsilon(K)=-1$ and $n\geq 0\geq 2\tau(K)$}
This case is shown in Figure \ref{taunegepsilonnegnpos} where the intersection points $\{a_k\}$ form a subcomplex that carries the $\F[V]$-free part of the homology, and the cycle $a_{2m+1}$ generates $\HFhat(S^3)$. Considering disks that cross the $w$-basepoint, we see that $\partial^h(b)=Ua_{2m+1}$ and so $\epsilon(Q^{0,j}_n(K))=1$. It remains to determine $\tau(Q^{0,j}_n(K))=A(a_{2m+1}$. This is similar to the previous cases, but we point out what happens in the case when $j$ is even and $n$ is odd. In this case the central intersection point $c$ with $A(c)=0$ is shown in Figure \ref{taunegepsilonnegnpos}. Since $j$ is even, the central intersection point occurs along the unstable chain, as in the proof of Theorem \ref{genusnontrivial}. Just as in the previous cases, we find that $$A(a_{2m+1})-A(c)=j\left(\dfrac{n-1}{2}-\tau(K)+(\dfrac{j}{2}-2\right)n+n+2\tau(K)+1)+\dfrac{j}{2}$$

Simplifying, we see that $$\tau(Q^{0,j}_n(K))=A(a_{2m+1})=j(\tau(K)+1)+\dfrac{j(j-1)}{2}n.$$

\textbf{Case $\tau(K)\leq 0$ $\epsilon(K)=-1$ and $0>n>2\tau(K)$}\\
This case is similar to the previous cases and is left to the reader.

\textbf{Case $\tau(K)\leq 0$ $\epsilon(K)=-1$ and $n\leq2\tau(K)$}
The pairing diagram for this case is shown in Figure \ref{taunegepsilonminus1nleq2tau}. In that figure we see that the intersection points labelled $\{a_k\}$ generate the free part of the homology of $\CFKhat_{\F[V]}(S^3,Q^{0,j}_n(K))$ and when we set $V=1$ the cycle $\sum_{k=0}^m a_{2k+1}$ generates $\HFhat(S^3)$. The intersection points $\{b_{2k+1}\}$ are such that $\partial^h(\sum_k b_{2k+1})=U\sum_{k=0}^m a_{2k+1}$, so $\epsilon(Q^{0,j}_n(K))=1$. It remains to determine $\tau(Q^{0,j}_n(K))=A(a_{1})$. Inspecting the pairing diagram, we find that 
$$\tau(Q^{0,j}_n(K))=j(\tau(K)+1)+\dfrac{j(j-1)}{2}n$$\qedhere

\end{proof}

\begin{figure}[!tbp]
  \centering
\begin{minipage}[b]{0.3\textwidth}
  \begin{tikzpicture}
\hspace{-.3in}
  \draw (.2,3.6) to (.2,4.2);
  \draw (.1,3.6) to (.3,3.6);
  \draw (.1,4.2) to (.3,4.2);
  \draw (.2,6.5) to (.2,4.2);
   \draw (.2,6.5) to (.2,8.3);
   \draw (.1,8.3) to (.3,8.3);
\node[anchor=south west,inner sep=0] at (0,0)    {\includegraphics[width=1.2\textwidth]{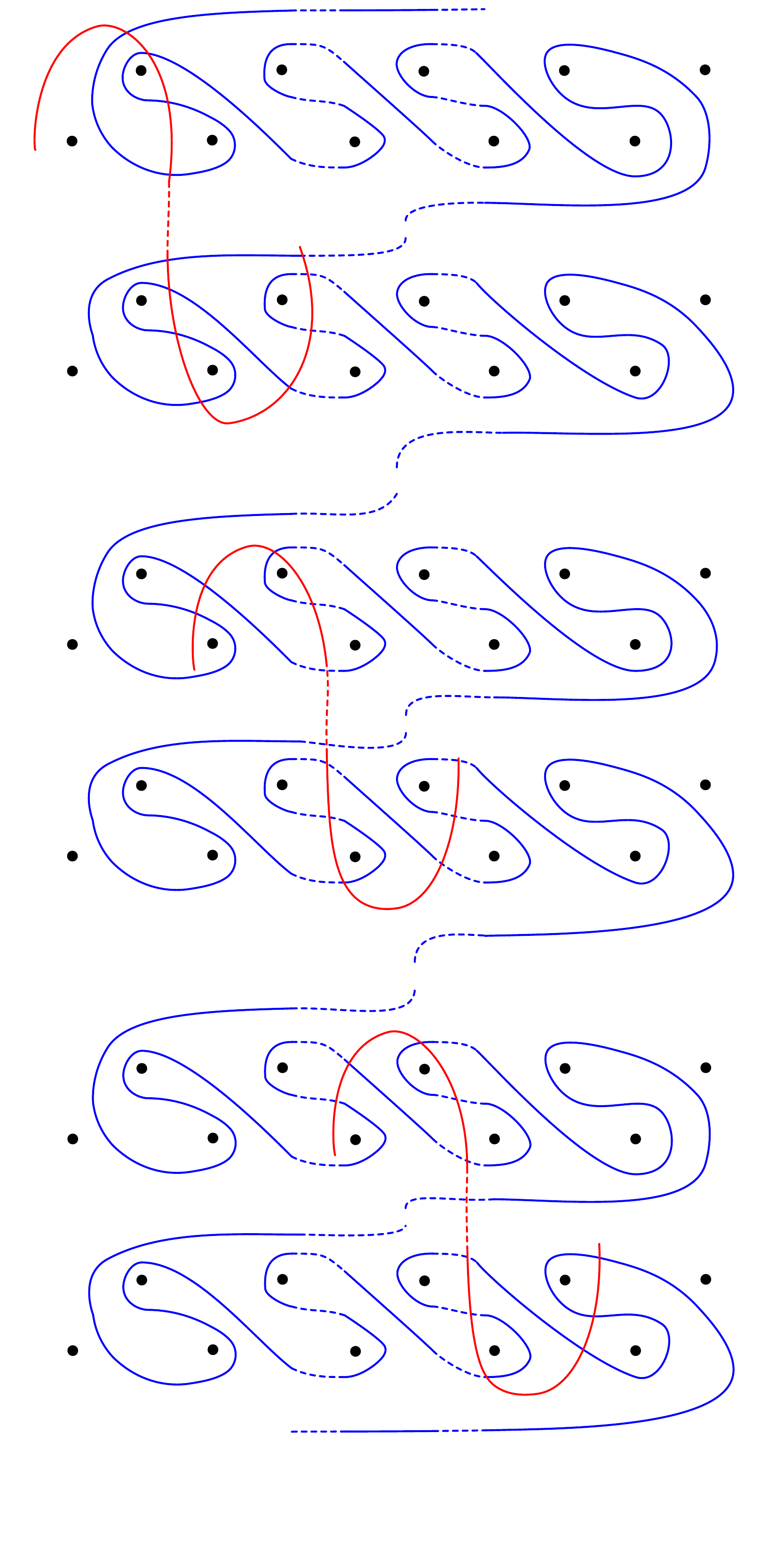}};
\node[font=\tiny] at (2.7,3.75) {$\tikzcirc[]{2pt}$};
\node[font=\tiny] at (.85,9.85) {$\tikzcirc{2pt}$};
\node[font=\tiny] at (1.15,8.95) {$\tikzcirc{2pt}$};
\node[font=\tiny] at (1.15,8.4) {$\tikzcirc{2pt}$};
\node[font=\tiny] at (1.3,7.5) {$\tikzcirc{2pt}$};
\node[font=\tiny] at (1.84,7.6) {$\tikzcircle{2pt}$};
\node[font=\tiny] at (2,7.9) {$\tikzcirclee{2pt}$};
\node[font=\tiny, rotate=90] at (0,3.9) {$\tau(K)$};
\node[font=\tiny, rotate=90] at (-.2,6.3) {$\dfrac{j-1}{2}n+1$};
\node[font=\tiny] at (2.9,3.75) {$c$};
\node[font=\tiny] at (2,7.4) {$a_1$};
\node[font=\tiny] at (2.2,7.90) {$b$};
\node[font=\tiny] at (1.2,7.3) {$a_2$};
\node[font=\tiny] at (.9,8.4) {$a_3$};
\node[font=\tiny] at (.9,8.9) {$a_{2m}$};
\node[font=\tiny] at (.8,10) {$a_{2m+1}$};

\end{tikzpicture}
    \caption{The pairing diagram when $\tau(K)>0$ $\epsilon(K)=-1$ and $n\geq 2\tau(K)$ and $j$ odd.}\label{tauposepsilonnegngreater2tau}
  \end{minipage}
  \hspace{2in}
  \begin{minipage}[b]{0.3\textwidth}
  \begin{tikzpicture}
  \hspace{-.2in}
\node[anchor=south west,inner sep=0] at (0,0)    {\includegraphics[width=1.2\textwidth]{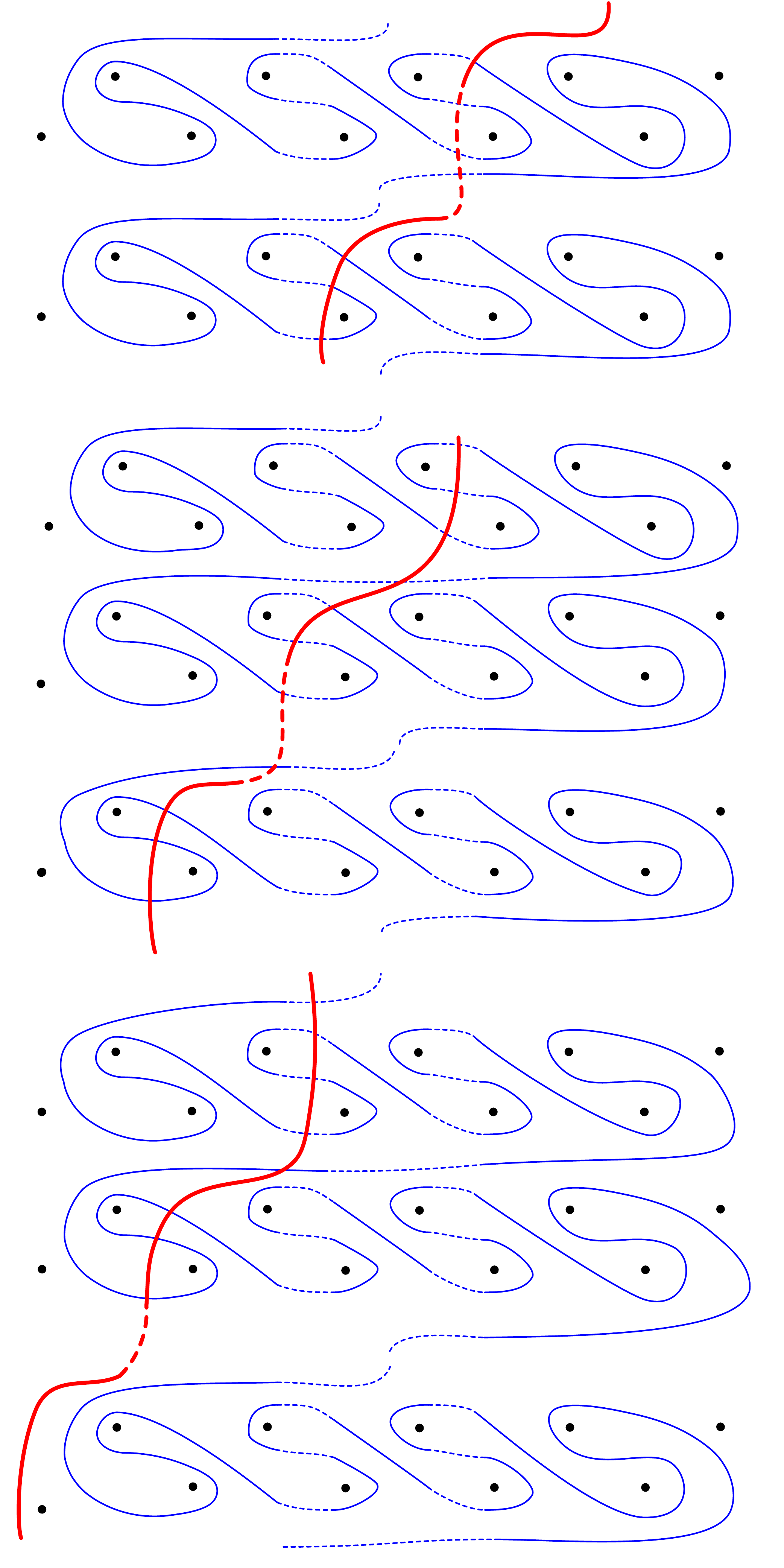}};
\node[font=\tiny] at (.95,1.75) {$\tikzcirc{2pt}$};
\node[font=\tiny] at (1.8,2.55) {$\tikzcirc{2pt}$};
\node[font=\tiny] at (2,2.8) {$\tikzcircle{2pt}$};
\node[font=\tiny] at (2,3.2) {$\tikzcirclee{2pt}$};
\node[font=\tiny] at (2.5,7.5) {$\tikzcirc[]{2pt}$};
\node[font=\tiny] at (1.05,2.1) {$\tikzcirclee{2pt}$};
\draw (0,8.6) to (0,7.2);
\draw (-.1,8.6) to (.1,8.6);
\draw (-.1,7.2) to (.1,7.2);
\draw (0,4) to (0,7.2);
\draw (-.1,4) to (.1,4);

\node[font=\tiny, rotate=90] at (-.2,7.8) {$|\tau(K)|$};
\node[font=\tiny, rotate=90] at (-.3,5.5) {$\dfrac{j-1}{2}n-1$};
\node[font=\tiny] at (1.25,2.15) {$b_3$};
\node[font=\tiny] at (1.25,1.6) {$a_3$};
\node[font=\tiny] at (1.6,2.65) {$a_2$};
\node[font=\tiny] at (2.2,2.7) {$a_1$};
\node[font=\tiny] at (2.15,3.25) {$b_1$};
\node[font=\tiny] at (2.5,7.7) {$c$};
\end{tikzpicture}\vspace{.5in}
    \caption{$\tau(K)<0$ $\epsilon(K)=-1$ and $n \leq 2\tau(K)$.}\label{taunegepsilonminus1nleq2tau}
    \end{minipage}
\end{figure}

\begin{figure}[!tbp]
  \centering
  \begin{minipage}[b]{0.3\textwidth}
  \begin{tikzpicture}
\node[anchor=south west,inner sep=0] at (0,0){\includegraphics[width=1\textwidth]{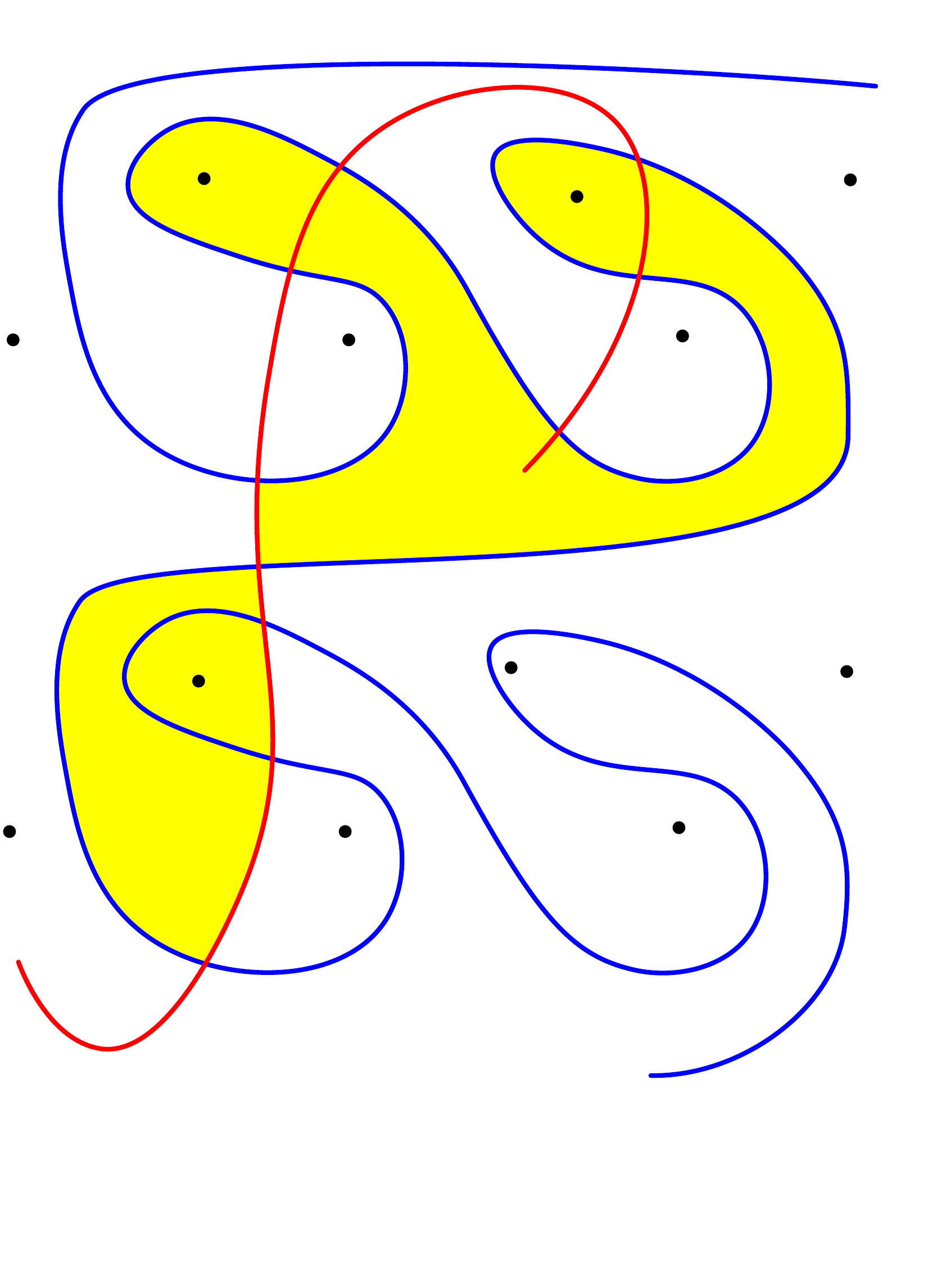}};
\node at (.95,1.45) {$\tikzcirc{2pt}$};
\node at (1.15,3.2) {$\tikzcirc{2pt}$};
\node at (1.2,3.6) {$\tikzcirc{2pt}$};
\node[font=\tiny] at (1,3.45) {$a$};
\end{tikzpicture}
    \caption{The subcomplex that carries the $\F[V]$-free part of the homology before twisting.}\label{subcomplexbeforetwist}
  \end{minipage}
  \hspace{1.3in}
  \begin{minipage}[b]{0.3\textwidth}
  \begin{tikzpicture}
\node[anchor=south west,inner sep=0] at (0,0)    {\includegraphics[width=1\textwidth]{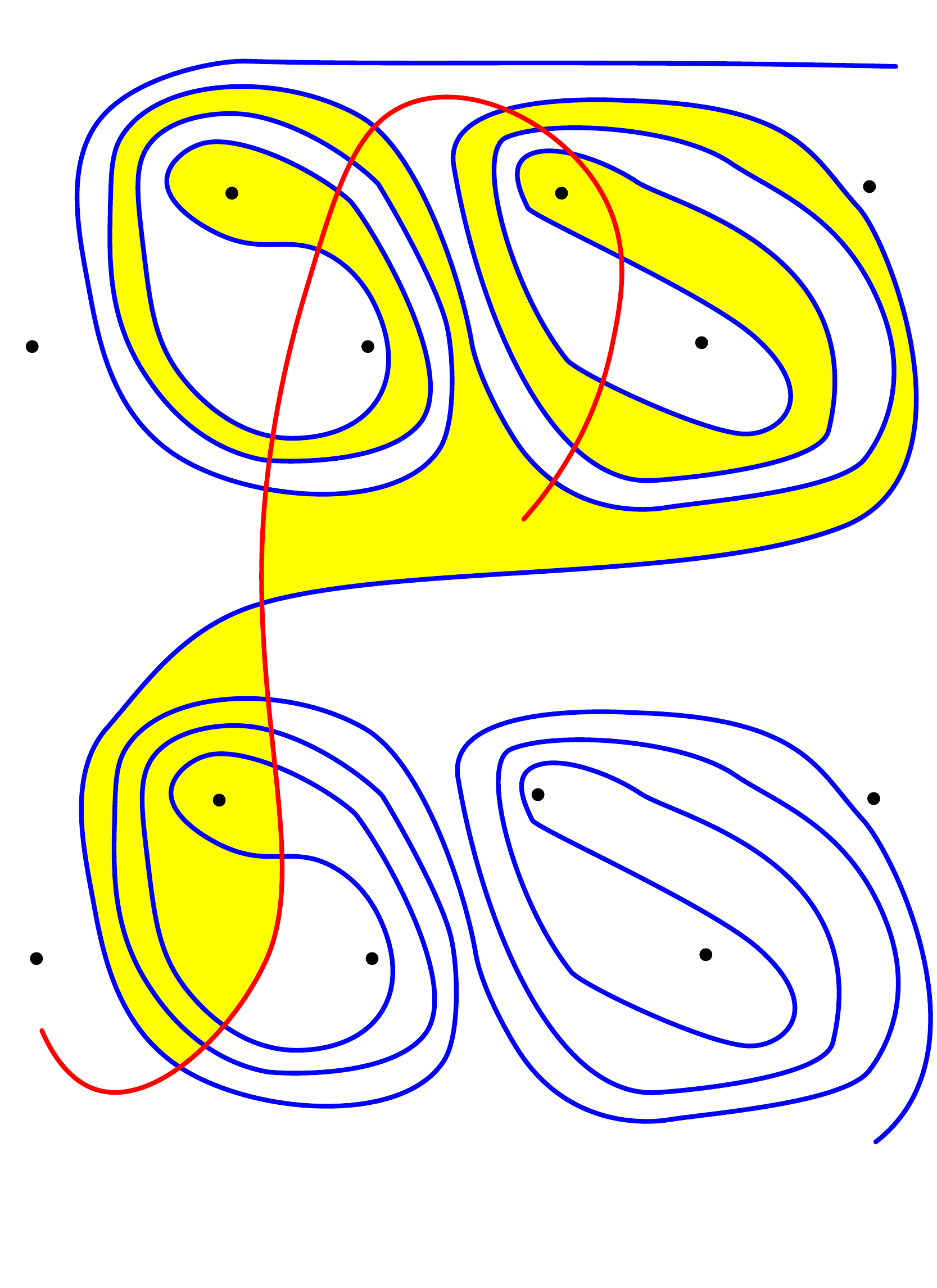}};
\node at (.8,1) {$\tikzcirc{2pt}$};
\node at (1.2,3.1) {$\tikzcirc{2pt}$};
\node at (1.2,3.6) {$\tikzcirc{2pt}$};
\node[font=\tiny] at (1,3.45) {$a$};
\end{tikzpicture}
    \caption{The subcomplex that carries the $\F[V]$-free part of the homology after twisting.}\label{subcomplexaftertwist}
  \end{minipage}
\end{figure}

\begin{lemma}\label{tauindependenti}
    For any $j \in \Z_{>0}$, $n \in \Z$ and $i \in \Z_{\geq 0}$,
    $$\tau(Q^{i,j}_n(K))=\tau(Q^{0,j}_n(K)).$$
\end{lemma}
\begin{proof}
    Inspection of the pairing diagram shows that the intersection points that form a subcomplex of $\CFK_{\F[V]}(S^3,Q^{i,j}_n(K))$ that generates the $\F[V]$ free part of the homology is independent of $i$. That is, twisting up the $\beta$ curve does not change the subcomplex under consideration and as remarked before, does not change the Alexander gradings of the previously existing intersection points. See Figures \ref{subcomplexbeforetwist} and \ref{subcomplexaftertwist}. In particular in all cases the cycle that survives to $\HFhat(S^3)$ and the Alexander grading of that cycle is independent of $i$. 
\end{proof}

\begin{lemma}
    For any $j \in \Z_{>0}$, $n \in \Z$ and $i \in \Z_{\geq 0}$, $$\epsilon(Q^{i,j}_n(K))=\epsilon(Q^{0,j}_n(K)).$$
\end{lemma}
\begin{proof}

There are a few cases depending on the shape of essentail component of the curve $\alpha(K,n)$, but the proof is essentially local in nature so we only indicate the local modification to the complex. Consider the case when the intersection point with Alexander grading $\tau(Q^{i,j}_n(K))$ and the vertical subcomplex nearby this intersection point has the form shown in figures \ref{epsilontwist1} and \ref{epsilontwist2}. For example this covers the cases when $\tau(K)\geq 0$ and $\epsilon(K)=1$ and $n \leq 2\tau(K)$ and $\tau(K)\leq 0$, $\epsilon=1$ and $n \leq 2\tau(K)$. When we twist the $\beta$ curve up once, notice that there are now two intersection points $b$ and $b'$ with a horizontal differential to $a$. However, this does not change the computation of $\epsilon$, since we can perform a change of basis, letting $b_1=b$ and $b_2=b+b'$. Then $\partial^h(b_1)=a$ and $\partial^h(b_2)=0$. We see from figures \ref{epsilontwist3} and \ref{epsilontwist4} that this pattern continues for each addition twist we add to the lifted $\beta$ curve. 
    
\end{proof}

\begin{figure}[!tbp]
  \centering
  \begin{minipage}[b]{0.3\textwidth}
  \begin{tikzpicture}
\node[anchor=south west,inner sep=0] at (0,0){\includegraphics[width=1\textwidth]{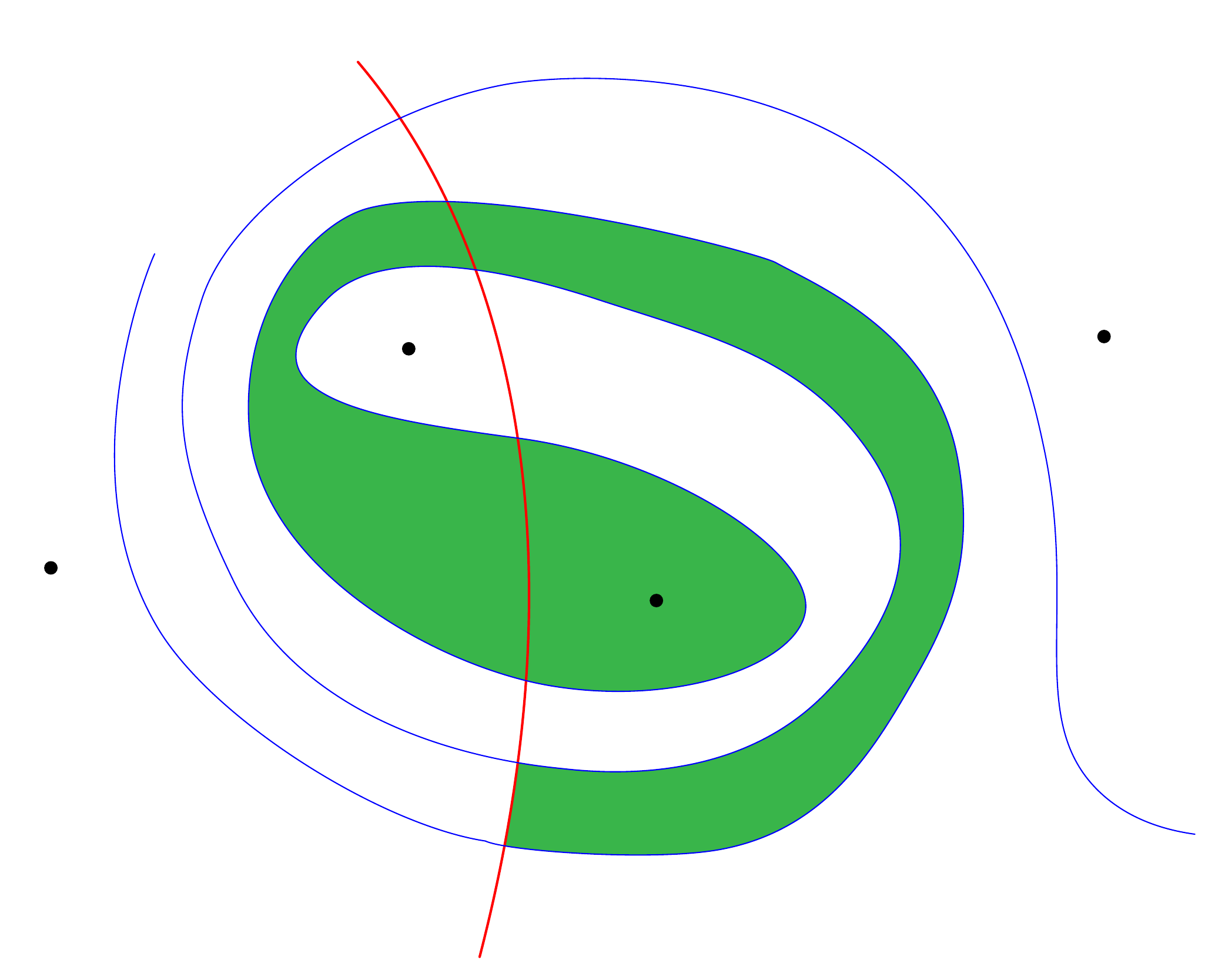}};
\node[font=\tiny] at (.2,1.3) {$w$};
\node[font=\tiny] at (2.1,1.2) {$w$};
\node[font=\tiny] at (1.25,2.1) {$z$};
\node[font=\tiny] at (3.6,2.1) {$z$};
\node[font=\tiny] at (1.75,.45) {$\tikzcirc{2pt}$};
\node[font=\tiny] at (1.8,.75) {$\tikzcirclee{2pt}$};
\node[font=\tiny] at (1.5,.4) {$a$};
\node[font=\tiny] at (1.5,.7) {$b$};
\end{tikzpicture}
    \caption{A horizontal differential to the intersection point that survives the spectral sequence to $\HFhat(S^3)$ when $i=1$}\label{epsilontwist1}
  \end{minipage}
  \hspace{1.3in}
  \begin{minipage}[b]{0.3\textwidth}
  \begin{tikzpicture}
\node[anchor=south west,inner sep=0] at (0,0)    {\includegraphics[width=1\textwidth]{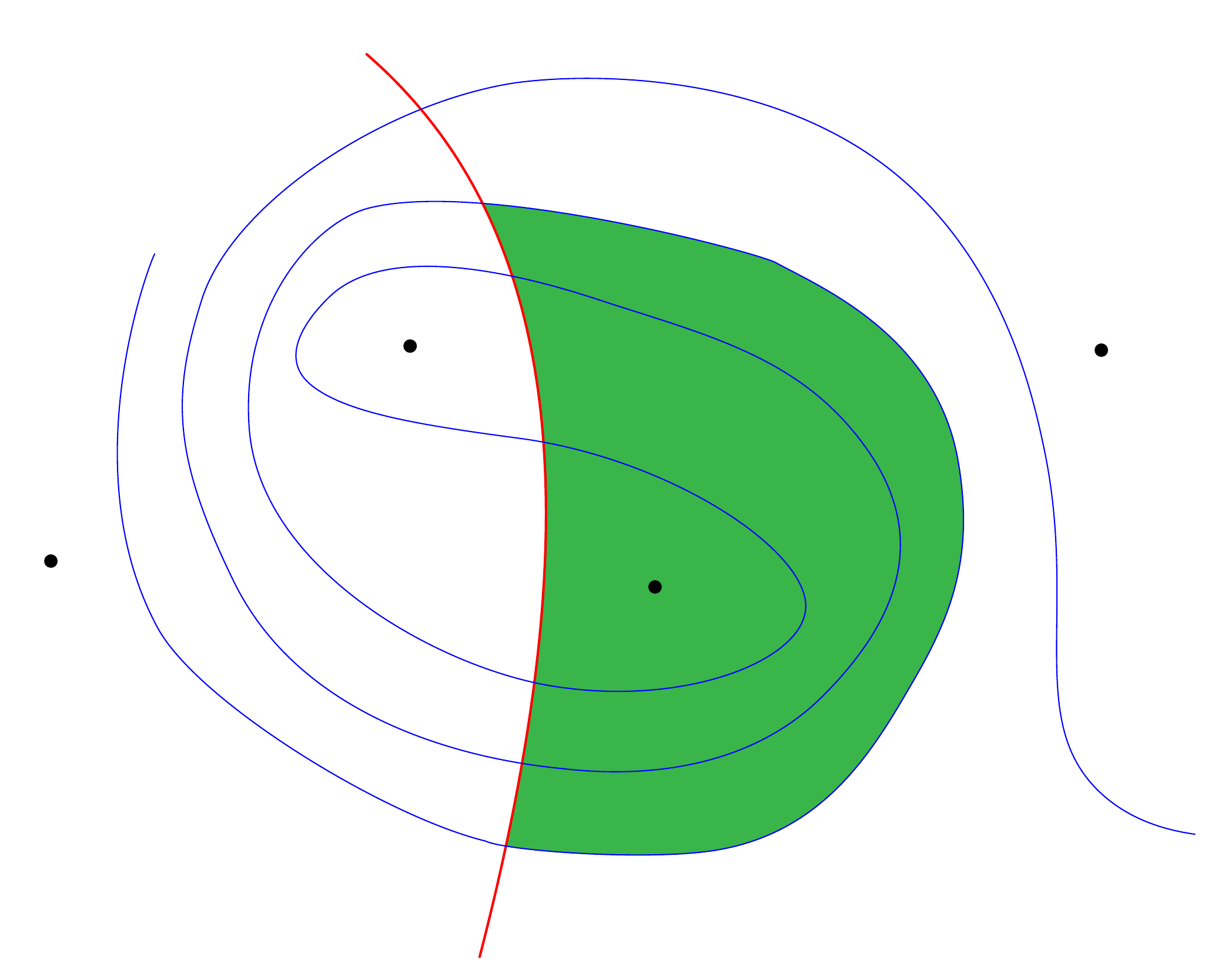}};
\node[font=\tiny] at (.2,1.3) {$w$};
\node[font=\tiny] at (2.1,1.2) {$w$};
\node[font=\tiny] at (1.25,2.1) {$z$};
\node[font=\tiny] at (3.6,2.1) {$z$};
\node[font=\tiny] at (1.75,.45) {$\tikzcirc{2pt}$};
\node[font=\tiny] at (1.65,2.65) {$\tikzcirclee{2pt}$};
\node[font=\tiny] at (1.5,.4) {$a$};
\node[font=\tiny] at (1.4,2.6) {$b'$};
\end{tikzpicture}
    \caption{Another horizontal differential to the intersection point that survives the spectral sequence to $\HFhat(S^3)$ when $i=1$}\label{epsilontwist2}
  \end{minipage}
\end{figure}

\begin{figure}[!tbp]
  \centering
  \begin{minipage}[b]{0.3\textwidth}
  \begin{tikzpicture}
\node[anchor=south west,inner sep=0] at (0,0)
{\includegraphics[width=1\textwidth]{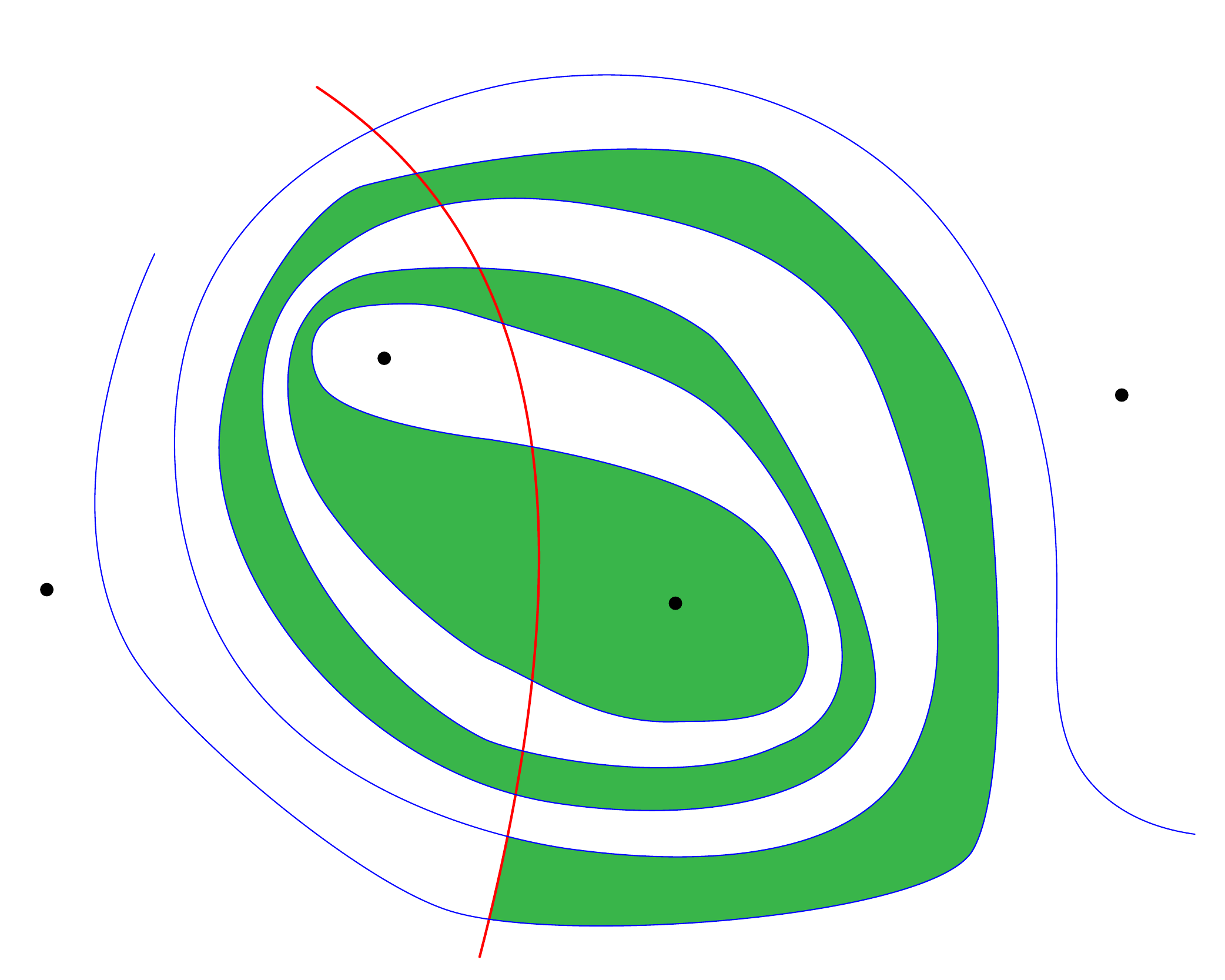}};
\node[font=\tiny] at (0,1.3) {$w$};
\node[font=\tiny] at (2.1,1.2) {$w$};
\node[font=\tiny] at (1.2,2.1) {$z$};
\node[font=\tiny] at (3.6,2.1) {$z$};
\node[font=\tiny] at (1.65,.25) {$\tikzcirc{2pt}$};
\node[font=\tiny] at (1.7,.5) {$\tikzcirclee{2pt}$};
\node[font=\tiny] at (1.4,.2) {$a$};
\node[font=\tiny] at (1.45,.5) {$b$};
\end{tikzpicture}
    \caption{A horizontal differential to the intersection point that survives the spectral sequence to $\HFhat(S^3)$ when $i>1$}\label{epsilontwist3}
  \end{minipage}
  \hspace{1.3in}
  \begin{minipage}[b]{0.3\textwidth}
  \begin{tikzpicture}
\node[anchor=south west,inner sep=0] at (0,0)    {\includegraphics[width=1\textwidth]{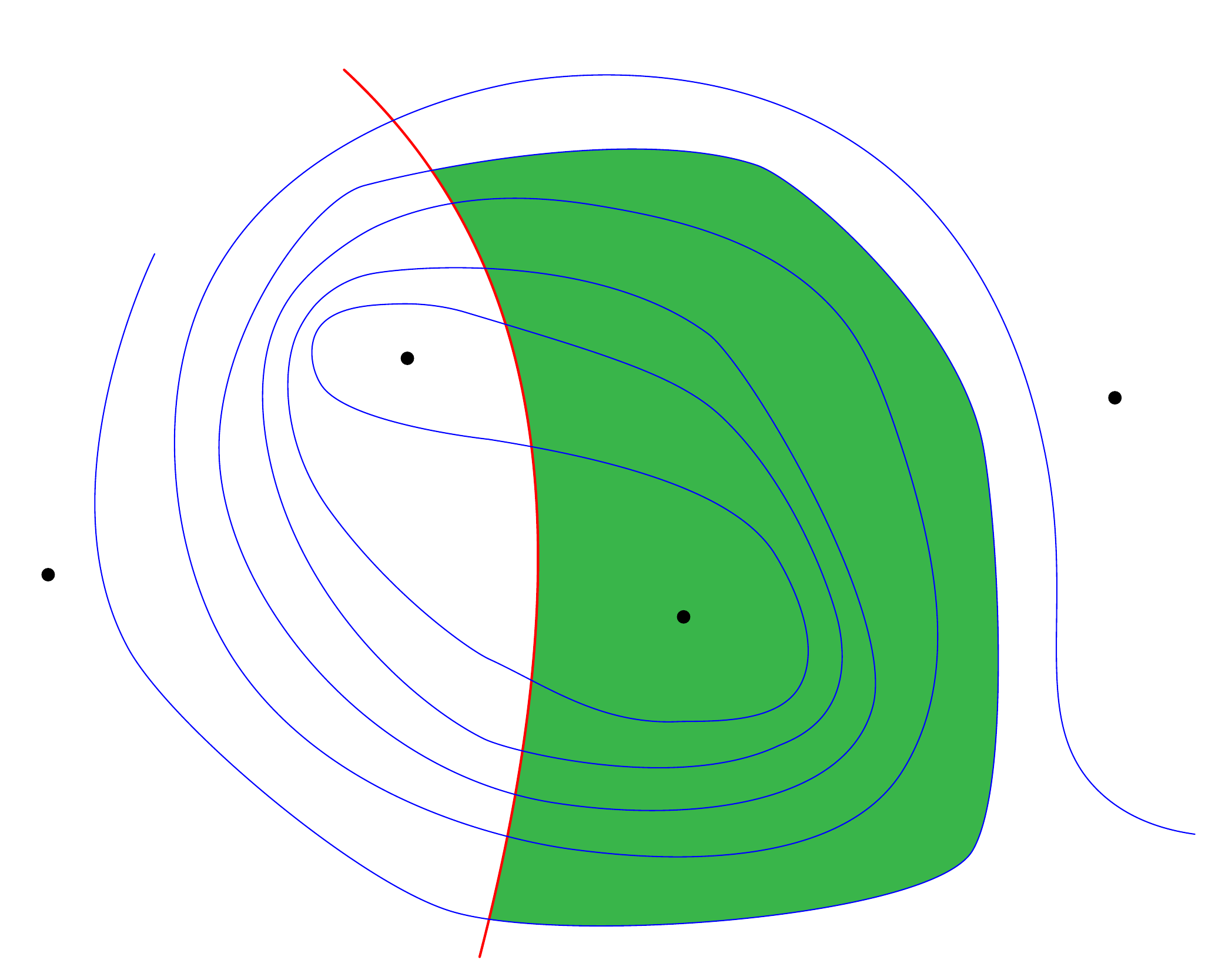}};
\node[font=\tiny] at (0,1.3) {$w$};
\node[font=\tiny] at (2.1,1.2) {$w$};
\node[font=\tiny] at (1.25,2.1) {$z$};
\node[font=\tiny] at (3.6,2.1) {$z$};
\node[font=\tiny] at (1.65,.23) {$\tikzcirc{2pt}$};
\node[font=\tiny] at (1.5,2.8) {$\tikzcirclee{2pt}$};
\node[font=\tiny] at (1.4,.2) {$a$};
\node[font=\tiny] at (1.3,2.85) {$b'$};
\end{tikzpicture}
    \caption{Another horizontal differential to the intersection point that survives the spectral sequence to $\HFhat(S^3)$ when $i>1$}\label{epsilontwist4}
  \end{minipage}
\end{figure}

\bibliography{genmazurbib.bib}

\end{document}